\documentclass[11pt]{article}
\usepackage{xcolor}
\usepackage{etoolbox}
\usepackage{thmtools}
\usepackage{amsmath,amsfonts,amssymb,amsthm}
\usepackage{ucs}

\AtBeginDocument{
	\addtolength{\abovedisplayskip}{-2ex}
	\addtolength{\abovedisplayshortskip}{-2ex}
	\addtolength{\belowdisplayskip}{-2ex}
	\addtolength{\belowdisplayshortskip}{-2ex}
}
\usepackage[toc,page]{appendix}
\usepackage{setspace}
\usepackage{graphicx}
\graphicspath{ {./images/} }
\usepackage{wrapfig}
\usepackage[usenames,dvipsnames]{pstricks}
\usepackage{fancyhdr}
\usepackage{amsmath,amssymb,amsfonts}
\newtheoremstyle{break}
{}
{}
{\itshape}
{}
{\bfseries}
{.}
{\newline}
{}

\theoremstyle{break}
\newtheorem{prop}{\textbf{Proposition}}[section]
\newtheorem{definition}{\textbf{Definition}}[section]
\newtheorem{theorem}{\textbf{Theorem}}[section]

\newtheorem{lemma}[theorem]{\textbf{Lemma}}
\newtheorem{claim}{\underline{\textbf{Claim}}}[section]

\newcommand{\R}{\mathbb{R}}
\newcommand{\N}{\mathbb{N}}

\newcommand{\abs}[1]{\lvert#1\rvert }

\newcommand{\norm}[1]{\left\lVert#1\right\rVert}

\begin{document}
	\title{Systems of Fully Nonlinear Degenerate Elliptic Obstacle problems  \\with Dirichlet boundary conditions}
	\author{S. Andronicou and E. Milakis}
	\date{}
	\maketitle
	\begin{abstract}
		In this paper, we prove existence and uniqueness of viscosity solutions to the following system: For $ i\in\left\lbrace 1,2,\dots,m\right\rbrace  $ \begin{gather}
			\min\biggl\{ F\bigl( y,x,u_{i}(y,x),D u_{i}(y,x),D^2 u_{i}(y,x)\bigl),  u_{i}(y,x)-\max_{j\neq i}\bigl( u_{j}(y,x)-c_{ij}(y,x)\bigl)\biggl\}=0, \left(y,x \right)\in\Omega_{L}\nonumber\\
			 u_{i}(0,x)=g_{i}(x), x\in\bar{\Omega},\ u_i(y,x)=f_i(y,x),  (y,x)\in(0,L)\times\partial{\Omega}\nonumber
		\end{gather}
		where $ \Omega\subset\R^n $ is a bounded domain, $ \Omega_{L}:=(0,L)\times\Omega$ and $ F:\left[ 0,L\right] \times\R^n\times\R\times\R^n\times\mathcal{S}^n\to\R $ is a general second order partial differential operator which covers even the fully nonlinear case\footnote{We will call a second order partial differential operator  $F:\left[ 0,L\right] \times\R^n\times\R\times\R^n\times\mathcal{S}^n\to\R  $ fully nonlinear if and only if, it has the following form \\ 
			\begin{equation}
				F(y,x,u,D_x u,D_{xx}^2 u):=\sum_{\abs{\alpha}=2}\alpha_{\alpha}\left(y,x,u,D_x u,D_{xx}^2 u \right) D^{\alpha}u(y,x)+\alpha_{0}(y,x,u,D_x u)\nonumber
		\end{equation} with the restriction that at least one of the functional coefficients $ \alpha_{\alpha},\ \abs{\alpha}=2, $ contains a partial derivative term of second order. }. Moreover, $ F $ belongs to an appropriate subclass of degenerate elliptic operators. Regarding uniqueness, we establish a comparison principle for viscosity sub and supersolutions of the Dirichlet problem. This system appears among other in the theory of the so-called optimal switching problems on bounded domains.
	\end{abstract}

	\section{Introduction}
	In this article, we examine \emph{existence} and \emph{uniqueness} of viscosity solutions, of a system of possibly fully nonlinear obstacle problems. Specifically, the system consists of variational type inequalities accompanied with interconnected obstacles, which furthermore satisfies initial and boundary conditions. Let $ \Omega\subset\R^n $ be a bounded domain, i.e an open, connected and bounded set in $ \R^n $ and let $ i\in\left\lbrace 1,2,\dots,m \right\rbrace  $ for some positive integer $ m. $ The Boundary Value Problem (BVP) is stated as follows:\\
		\begin{gather}
			\min\biggl\{ F\bigl( y,x,u_{i}(y,x),D_{x}u_{i}(y,x),D_{xx}^2 u_{i}(y,x)\bigl),  u_{i}(y,x)-\max_{j\neq i}\bigl( u_{j}(y,x)-c_{ij}(y,x)\bigl)\biggl\}=0,\ \text{for all\ }\left(y,x \right)\in\Omega_{L}\nonumber\\
			 u_{i}(0,x)=g_{i}(x), x\in\bar{\Omega},\ u_i(y,x)=f_i(y,x),  (y,x)\in(0,L)\times\partial{\Omega}\nonumber
		\end{gather}
	where $ \Omega_{L}:=(0,L)\times\Omega. $
	The solution is considered as a vector mapping, $$ u:=\left( u_1,u_2,\dots,u_m\right):[0,L]\times\bar{\Omega}\rightarrow\R^m  $$ where the component functions $ u_i $ are interconnected through the obstacle $ \mathcal{M}_i $, in such a way  that the component function  $ u_i $ will be above the obstacle 
	$$ \mathcal{M}_i u(y,x):=\max_{j\neq i}\left( u_{j}(y,x)-c_{ij}(y,x)\right).$$
	
	This problem  is included in the general area of the so called optimal switching problems and 
	occurs,  for instance, when one models a facility which splits its production in m modes. In optimal switching problems there are basically two main approaches i.e. the probabilistic approach and the deterministic approach. For the first approach we refer the reader to \cite{AF12}, \cite{DH09}, \cite{DHP10}, \cite{HT07} and \cite{HZ10} for a well developed theory of stochastic differential equations governed by Brownian motion and the connection to multi-modes optimal switching problems. In the local setting, the second (deterministic) approach deals with the well developed theory of  variational inequalities within the general area of partial differential equations (see \cite{A16},  \cite{ADPS09}, \cite{AH09}, \cite{EF79}, \cite{HM}, \cite{LNO14}, and references therein). For applications on the global model, we refer the  reader to \cite{OL2021}.
	
	The main purpose of the present article is to study the elliptic fully nonlinear analogue of variational type inequalities accompanied with interconnected obstacles (which is originated in the fundamental work of Evans and Friedman \cite{EF79}). For this purpose we use the notion of viscosity solutions to fully nonlinear equations which provides a powerful way to prove existence and uniqueness in a very general setting. Our techniques could be considered within the classical approach (i.e comparison principles and Perron's method) which has been very successful the last two decades when seeking existence in problems with thick or thin obstacles (see for instance \cite{kiamlee1}, \cite{kiamlee2}, \cite{ms1}, \cite{ms2} and references therein) and are greatly inspired by the parabolic analogue of a similar problem which has been recently studied by N. Lundstr\"{o}m and M. Olofsson in \cite{LNOO4}. Although we follow the approach in \cite{LNOO4}, especially when dealing with the corresponding Perron type method and comparison principles, we need to pay extra attention in the technical  parts of the proofs.  This has to do with the fact that our problem seems to be more close to obstacle type problems for extension operators of degenerate form (see the pioneer work of Caffarelli-Silvestre in \cite{CS}). 
	
	Throughout the article, the symbolisms $ D_x u_{i}\equiv D u_{i} \in R^n  $, $  D_{xx}^2 u_{i}\equiv D^2 u_{i}\in\mathcal{S}^n $ refer respectively to the gradient and the hessian matrix of $ u_i $ with respect to the variable $ x $, where $ \mathcal{S}^n  $ is the set of real symmetric matrices $ n\times n $. Additionally, the operator $ F:\left[0,L \right] \times\R^n\times\R\times\R^n\times\mathcal{S}^n\to\R $ is considered as  possibly fully nonlinear for which, further assumptions will be set for the study of the problem. The modifier ``\textbf{possibly fully nonlinear}'' stands for ``\textbf{$ F $ may not be linear in any of its arguments}", but of course linear forms of $ F $ are not excluded i.e $ F $   is either fully nonlinear or a quasilinear operator. 
	
	Regarding the structure of the paper, in Section 2 we provide the necessary assumptions and definitions for our problem while in Section 3 we present our main results, existence and uniqueness of the solution. 
	
	\section{Assumptions and Definitions}
	\subsection{Assumptions}
	Let $ \Omega\subset\R^n $ be an open, connected and bounded set, with closure $ \bar{\Omega} $ and boundary $ \partial{\Omega}=\bar{\Omega}/\Omega $ and let $ y\in[0,L]. $
	\subsubsection{ Assumptions for the operator $ F $}
	 We suppose that the operator $  F:[0,L]\times\R^n\times\R\times\R^n\times\mathcal{S}^n\to\R, $ is a \emph{continuous} mapping which satisfies the following:
	\begin{itemize}
		\item [( F1)]  For each $ \left( y,x,p,X\right)\in[0,L]\times\R^n\times\R^n\times\mathcal{S}^n,  $ the mapping $ r\mapsto F\left( y,x,r,p,X\right)-\gamma r  $ is strictly increasing for some constant  $ \gamma>0 $. In specific, for each $ r, s\in \R $ with $ r<s $, holds:
		$$\gamma (s-r)\leq F\left( y,x,s,p,X\right)-F\left( y,x,r,p,X\right) $$
	\end{itemize}
	
	 There exists a \emph{continuous} function  $ \omega :[0,\infty) \to [0,\infty)$ with $ \omega(0+)=0, $ such that:
	\begin{itemize}
		\item [( F2)] $ \text{for all\ }\left(y,x,r\right)\in [0,L]\times\R^n\times\R, \text{for all\ } p,q\in\R^n,\text{for all\ } X,Y\in\mathcal{S}^n $
		$$ \left| F\left(y,x,r,p,X \right)-F\left(y,x,r,q,Y \right) \right|\leq\omega\big(\left| p-q\right|+\norm{X-Y})$$
		\item [( F3)]  $ \text{for all\ }\left( y,r,p\right)\in[0,L)\times\R\times\R^n, \text{for all\ } x,\tilde{x}\in\R^n,\text{for all\ } X,Y\in\mathcal{S}^n, \text{for all\ }\epsilon>0  $, holds
		$$ F\left(y,\tilde{x},r,p,Y \right)-F\left(y,x,r,p,X \right)\leq\omega\Big(\frac{1}{\epsilon}\left|x-\tilde{x} \right|^2+\left|x-\tilde{x} \right|\left( \left|p \right|+1 \right)\Big) $$
		whenever
		$$\begin{pmatrix}	X & 0 \\0 & -Y \end{pmatrix} \leq\frac{3}{\epsilon}\begin{pmatrix}	I & -I \\-I & I \end{pmatrix}.$$
	\end{itemize}

	 From the above conditions for the operator $ F $, we extract that $ F $ is \emph{ (degenerate) elliptic}. In specific,
	$$ \text{for all\ } \left(y,x,r,p \right)\in[0,L]\times\R^n\times\R\times\R^n,\text{for all\ } X,Y\in \mathcal{S}^n,  X\leq Y\Rightarrow F\left(y,x,r,p,X \right)\geq F\left(y,x,r,p,Y \right)  $$
	Consequently, $ F $ is a \emph{non-increasing} function with respect to the fifth variable.
	\subsubsection{Assumptions for the obstacle functions$ \ \left(c_{i,j} \right)  $}
	 For the obstacle functions, $ \left(c_{i,j} \right) $ we make the following assumptions:
	\begin{itemize}
		\item [(O1)] $ c_{i,j}\in C^{0,2}\left( [0,L] \times\bar{\Omega}\right),$ for all $ i,j\in\left\lbrace 1,2,\dots,m \right\rbrace $ i.e $ c_{i,j} $ is a continuous function on $ [0,L] \times\bar{\Omega} $ and for each $ y\in [0,L] $ the function $ c_{i,j}(y,\cdot):\Omega\to\R $ is twice continuously differentiable on $ \Omega. $
		\item [(O2)] $ c_{i,i}(y,x)=0,\text{for all\ } \left( y,x\right)\in[0,L] \times\bar{\Omega}, \text{for all\ } i\in\left\lbrace 1,2,\dots,m\right\rbrace$
		\item[(O3)] For each finite sequence of indices $ i_j\in\left\lbrace 1,2,\dots,m\right\rbrace ,\text{for all\ } j\in\left\lbrace 1,2,\dots,k\right\rbrace , \text{for all\ }\left(y,x \right)\in[0,L]\times\bar{\Omega}$ holds
		$$c_{i_1,i_2}\left( y,x\right) +c_{i_2,i_3}\left( y,x\right) +\cdots+c_{i_{k-1},i_k}\left( y,x\right) +c_{i_k,i_1}\left( y,x\right) >0 .$$
	\end{itemize}
	 For the existence of viscosity solutions of the problem $ (IBVP) $ we demand furthermore the condition that follows:
	\begin{itemize}
		\item[(O4)] $ \text{for all\ }\left( y,x\right)\in[0,L]\times\bar{\Omega}, \text{for all\ } i,j,k\in\left\lbrace 1,2,\dots,m\right\rbrace  $ it holds:
		$$c_{i,k}\left(y,x\right) \leq c_{i,j}\left(y,x\right) +c_{j,k}\left(y,x\right).  $$
		\item[(O5)] Finally, we demand for the boundary data $ f_i,\ i\in\{1,2,\dots,m\} $ to be continuous on $ (0,L)\times\partial\Omega $. Moreover, we demand for the data $ g_i,i\in\left\lbrace 1,2,\dots,m \right\rbrace  $ to be continuous, i.e $g_i\in C\left( \bar{\Omega}\right)   $ and to be compatible with the obstacle functions, in the following way:\\ $ \text{for all\ } x\in\bar{\Omega}, \text{for all\ } i,j\in\left\lbrace 1,2,\dots,m \right\rbrace  $
		$$g_i(x)\geq g_j(x)-c_{i,j}\left(0,x \right). $$
	\end{itemize}
	
	\subsection{Definitions}
	 First, the notions of  \emph{subjets}  and \emph{superjets} are defined, from which we define later the notion of \emph{viscosity solutions} for the problem $(IBVP)$.
	
	\begin{definition}[\emph{Subjets} and \emph{Superjets}]
		 Let $ h:[0,L]\times\bar{\Omega}\to\R $ be a mapping. For $\left( y_0, x_{0}\right)\in(0,L)\times\Omega $ we define the sets $ J^{2,+}h(y_0,x_{0}) $ and $ J^{2,-}h(y_0,x_{0}) $, which are called respectively  SuperJet  of  $ h $ on $\left( y_0,x_{0} \right) $ and SubJet of $ h $ on $\left( y_0,x_{0} \right) $.
		\begin{itemize}
			\item  $\text{for all\ }\left(\alpha,p,X \right)\in\R\times\R^{n}\times \mathcal{S}^n, \left(\alpha,p,X \right)\in J^{2,+}h(y_0,x_{0}) $ if and only if 
			$$ h(y,x)\leq h(y_0,x_{0})+\alpha\left(y-y_0 \right)+\left\langle p,x-x_{0} \right\rangle +\frac{1}{2} \left\langle X(x-x_{0}),x-x_{0} \right\rangle \\ + o\left(\left|y-y_0 \right|+  \left| x-x_{0} \right|^{2}  \right),$$
			while $ \left(0,L \right) \times\Omega\ni\left( y,x\right) \to\left(y_0,x_{0}\right)   $ 
			\item  $\text{for all\ }\left(\alpha,p,X \right)\in\R\times\R^{n}\times  \mathcal{S}^n, \left(\alpha,p,X \right)\in J^{2,-}h(y_0,x_{0}) $ if and only if
			$$  h(y,x)\geq h(y_0,x_{0})+\alpha\left(y-y_0 \right)+\left\langle p,x-x_{0} \right\rangle +\frac{1}{2} \left\langle X(x-x_{0}),x-x_{0} \right\rangle \\ + o\left(\left|y-y_0 \right| +  \left| x-x_{0} \right|^{2}  \right),$$
			while $ \left(0,L \right) \times\Omega\ni\left( y,x\right) \to\left(y_0,x_{0}\right)   $
		\end{itemize}
		 Correspondingly, we define the closures of the $ J^{2,+}h(y_0,x_{0}) $ and $ J^{2,-}h(y_0,x_{0}) $. In specific:
		\begin{itemize}
			\item $ \left(\alpha, p,X \right) \in \bar{J}^{2,+}h(y_0,x_{0}) $ if and only if, $ \text{there exists\ } \left( y_k,x_{k}\right)\in\left(0,L \right)\times \Omega,\text{there exists\ }\left(\alpha_k, p_{k},X_{k}\right)\in J^{2,+}h(y_k,x_{k}), $
			$$ \left(y_k, x_{k},h(y_k, x_{k}),\alpha_k,p_{k},X_{k}\right)\rightarrow\left( y_0,x_{0},h(y_0,x_{0}),\alpha,p,X\right)$$
			
			\item  $ \left(\alpha, p,X \right) \in \bar{J}^{2,-}h(y_0,x_{0}) $ if and only if, $ \text{there exists\ } \left( y_k,x_{k}\right)\in\left(0,L \right)\times \Omega,\text{there exists\ }\left(\alpha_k, p_{k},X_{k}\right)\in J^{2,-}h(y_k,x_{k}), $
			$$ \left(y_k, x_{k},h(y_k, x_{k}),\alpha_k,p_{k},X_{k}\right)\rightarrow\left( y_0,x_{0},h(y_0,x_{0}),\alpha,p,X\right)$$
		\end{itemize}
		
	\end{definition}
	 Next we provide the definition of \emph{Viscosity Solution} of the problem $ (BVP) $.
	
	\begin{definition}[\emph{Viscosity Solutions}]
		 Let $ u=\left(u_{1},u_{2},\dots,u_{m} \right):\left[ 0,L\right] \times\bar{\Omega}\rightarrow\R^{m}  $ be a mapping. Then, $ u:[0,L]\times\bar{\Omega}\rightarrow\R^{m} $ is called:
		\begin{itemize}
			\item[($\alpha$)] \emph{viscosity subsolution} of the problem $ (BVP) $ if and only if, 
			\begin{itemize}
				\item[(i)]  $ u $ is \emph{upper semicontinuous} on  $ \bar{\Omega}_L:=[0,L]\times\bar{\Omega}, $
				\item [(ii)]	$ \forall i\in\left\lbrace 1,2,\dots, m \right\rbrace,\forall \left(y,x \right)\in(0,L)\times\Omega,\forall\left(\alpha, p,X \right)\in \bar{J}^{2,+}u_i(y,x),$
				$$ \min\left\lbrace F\left(y,x,u_i(y,x),p,X\right),u_i(y,x)-\mathcal{M}_i u(y,x)\right\rbrace \leq 0 $$ 
				\item[(iii)] $ u_{i}(0,x)\leq g_{i}(x),\ x\in\bar{\Omega},\quad \land \quad u_i(y,x)\leq f_i(y,x),\ (y,x)\in (0,L)\times\partial{\Omega}  $
			\end{itemize}
			\item[($ \beta $)] \emph{viscosity supersolution} of the problem $ (BVP) $ if and only if,
			\begin{itemize}
				\item[(i)]  $ u $ is \emph{lower semicontinuous} on $ \bar{\Omega}_L:=[0,L]\times\bar{\Omega}, $ 
				\item[(ii)] $ \forall i\in\left\lbrace 1,2,\dots, m \right\rbrace,\ \forall\left(y,x \right)\in(0,L)\times\Omega,\forall\left(\alpha, p,X \right)\in \bar{J}^{2,-}u_i(y,x),$
				$$\min\left\lbrace F\left(y,x,u_i(y,x),p,X\right),u_i(y,x)-\mathcal{M}_i u(y,x)\right\rbrace \geq 0 $$ 
				\item[(iii)] $ u_{i}(0,x)\geq g_{i}(x),\ x\in\bar{\Omega},\quad \wedge\quad u_i(y,x)\geq f_i(y,x),\ (y,x)\in (0,L)\times\partial{\Omega} $
			\end{itemize} 
			\item[($\gamma$)] \emph{viscosity solution} of the problem $ (BVP) $ if and only if,
			$ u $ is bounded on $ [0,L]\times\bar{\Omega} $, continuous on $ [0,L)\times\bar{\Omega} $ and satisfies  $ (ii) $ and $ (iii) $ of definitions ($ \alpha $) and ($ \beta $)
		\end{itemize}
		
	\end{definition}


	\section{ Main Results}
	\subsection*{Uniqueness of Solution}
	In this article, when we refer to the notion of uniqueness of viscosity solution of $ (BVP), $ we mean that, for every two functions $  u, w:[0,L]\times\bar{\Omega}\to\R^m $ that satisfy the definition of viscosity solution, have to be identical on the set $ [0,L)\times\bar{\Omega} $ i.e
	\begin{gather}
		\text{for all\ } i\in\{1,2,\dots,m\},\ \text{for all\ }(y,x)\in[0,L)\times\bar{\Omega},\ u_i(y,x)=w_i(y,x). 
	\end{gather} 
 	For this task, it is sufficent to prove an appropriate Comparison Principle.
	\begin{theorem}[Comparison Principle] 
		\label{Comparison Principle}
		Let the axioms $ (F1)-(F3) $ and $ (O_1)-(O_3)$ hold. If $ u:[0,L]\times\bar{\Omega}\to\R^m $ and $ v:[0,L]\times\bar{\Omega}\to\R^m $ are  subsolution and  supersolution respectively of the problem $ (BVP) $, then $$ u(y,x)\leq v(y,x), \text{for all\ } (y,x)\in [0,L)\times\bar{\Omega}. $$
	\end{theorem}
	
	\begin{proof}[Proof]
		 Throughout of the proof, we consider that functions $ u$ and $ v $ are  subsolution and  supersolution respectively of the problem $ (BVP). $ Although the proof follows the standard proof of \cite{CIL}, the appearance of interconnected obstacles does not allow us to apply this theory right away. Therefore we have to carefully implement the proof  through the following steps:\\
		\textbf{Step 1}\  From the assumption we have that $ v_i\in LSC\left( [0,L]\times\bar{\Omega}\right).  $ So equivalently it must be $ -v_i\in USC\left( [0,L]\times\bar{\Omega}\right).  $ Then we have $ u_i-v_i\in USC\left( [0,L]\times\bar{\Omega}\right), \text{for all\ } i\in\left\lbrace 1,2,\dots,m \right\rbrace  $. Because $ u_i-v_i $ is upper semicontinuous on the compact set $ [0,L]\times\bar{\Omega},  $ from Maximum Principle for upper semicontinuous functions, it will exists $ \left(y_i,x_i \right)\in [0,L]\times\bar{\Omega}  $ such that $ \sup_{[0,L]\times\bar{\Omega}}\left(u_i-v_i \right)\left( y,x\right)=u_i(y_i,x_i) -v_i(y_i,x_i)  .$ Let $ i_0\in\left\lbrace 1,2,\dots,m\right\rbrace  $ be such that:
		$$\max_{k\in\left\lbrace 1,2,\dots,m\right\rbrace}\sup_{[0,L]\times\bar{\Omega}}\left( u_k-v_k\right)(y,x)=\sup_{[0,L]\times\bar{\Omega}}\left( u_{i_0}-v_{i_0}\right)(y,x)=u_{i_0}(y_{i_0},x_{i_0})-v_{i_0}(y_{i_0},x_{i_0})\equiv\delta. $$
		
		We define 
		$$I:=\big\{ i\in\{ 1,2,\dots,m\} \ : \ u_i\left( y_{i_0},x_{i_0}\right) -v_i\left( y_{i_0},x_{i_0}\right) =\delta \big\}.$$
		Clearly we see that $ I\neq\emptyset $ because $ i_0\in I $ and in addition, from the definition of $ I $, we get that \\ $ \text{for all\ } i\in I, \sup_{[0,L]\times\bar{\Omega}}\left(u_i-v_i \right)\left( y,x\right) =\delta.  $
		
		Let us assume the opposite of what we seek to prove. As a result of this, we receive that 
		\begin{gather}
			\text{there exists\ } \left( \tilde{y},\tilde{x}\right) \in [0,L)\times\bar{\Omega}: \rceil \left(u\left( \tilde{y},\tilde{x}\right) \leq v\left(\tilde{y},\tilde{x} \right)  \right)\nonumber\\
			\Longleftrightarrow \text{there exists\ } \left( \tilde{y},\tilde{x}\right) \in [0,L)\times\bar{\Omega},\ \text{there exists\ } \tilde{i}\in\left\lbrace 1,2,\dots,m\right\rbrace: u_{\tilde{i}}\left(\tilde{y},\tilde{x} \right)-v_{\tilde{i}}\left( \tilde{y},\tilde{x}\right) >0.\nonumber
		\end{gather}
		So we get
		\begin{gather}
			\delta\geq\sup_{[0,L]\times\bar{\Omega}}\left( u_{\tilde{i}}-v_{\tilde{i}}\right) \left( y,x\right) \geq u_{\tilde{i}}\left(\tilde{y},\tilde{x} \right)-v_{\tilde{i}}\left( \tilde{y},\tilde{x}\right) >0\nonumber
		\end{gather}
		in consequence, $ \delta>0. $   
		Due to the fact that $ u,v  $ are subsolution and supersolution respectively, we receive from their definition, that
		\begin{gather}
			u_i(0,x)\leq g_i(x)\leq v_i(0,x),\ \text{for all\ } x\in\bar{\Omega}\\
			u_i(y,x)\leq f_i(y,x)\leq v_i(y,x),\ \text{for all\ } (y,x)\in(0,L)\times\partial{\Omega}.
		\end{gather}
		From the above inequalities, we conclude that $ y_{i_0}>0  $  and $ x_{i_0}\in\Omega. $ In the continuity of our analysis, it is sufficent to further consider (the explanation is given in the end of the proof) that, $ y_{i_0}<L. $
		We choose a random  $ i\in I $ in the analysis that follows. For each $ \epsilon>0, $ we define the following mappings:
		$$\phi_{\epsilon},\Phi_{\epsilon}:[0,L]\times\bar{\Omega}\times\bar{\Omega}\to\R$$
		defined as
		\begin{gather}
			\phi_{\epsilon}(y,x,\tilde{x}):=\frac{1}{2\epsilon}\abs{x-\tilde{x}}^2+\abs{x-x_{i_0}}^4+\abs{\tilde{x}-x_{i_0}}^4+\abs{y-y_{i_0}}^2\nonumber\\
			\Phi_{\epsilon}(y,x,\tilde{x}):=u_i(y,x)-v_i(y,\tilde{x})-\phi_{\epsilon}(y,x,\tilde{x}).\nonumber
		\end{gather}
		 The $ \Phi_{\epsilon} $ as an upper semicontinuous on the compact set $ [0,L]\times\bar{\Omega}\times\bar{\Omega}, $ attains a maximum value. Let $ \left(y_{\epsilon},x_{\epsilon},\tilde{x}_{\epsilon} \right)\in [0,L]\times\bar{\Omega}\times\bar{\Omega}  $ be such that
		\begin{gather}
			\Phi_{\epsilon}\left(y_{\epsilon},x_{\epsilon},\tilde{x}_{\epsilon}  \right) =\max_{[0,L]\times\bar{\Omega}\times\bar{\Omega}} \Phi_{\epsilon}\left(y,x,\tilde{x}\right).\label{eq_3}
		\end{gather}
		 Next, an important estimate of the distance between the terms of $ x_{\epsilon} $ and $ \tilde{x}_{\epsilon} $ is proved,. We will subsequently see that it is used to prove the limiting behaviour of an appropriate sequence of points which is extracted from the following family of points $ \left(y_{\epsilon},x_{\epsilon},\tilde{x}_{\epsilon} \right)_{\epsilon>0}.  $
		\begin{claim}\label{claim_1}
			$$\text{For all\ } \epsilon>0,\quad \frac{1}{2\epsilon}\abs{x_{\epsilon}-\tilde{x}_{\epsilon}}^2\leq u_i(y_{\epsilon},x_{\epsilon})-v_i(y_{\epsilon},\tilde{x}_{\epsilon})-\delta $$
		\end{claim}
		\begin{proof}[ Proof]
			 We observe that 
			\begin{gather}
				\Phi_{\epsilon}\left(y_{i_0},x_{i_0},x_{i_0} \right)=u_i\left(y_{i_0},x_{i_0} \right) - v_i\left(y_{i_0},x_{i_0} \right)=\delta\quad \left( \text{because}\ i\in I \right) \nonumber
			\end{gather}
				but at the same time, from $ \left( \ref{eq_3}\right)  $ we conclude that $ \Phi_{\epsilon}\left(y_{\epsilon},x_{\epsilon},\tilde{x}_{\epsilon}  \right)\ge\Phi_{\epsilon}\left(y_{i_0},x_{i_0},x_{i_0} \right)=\delta $.  Therefore, we receive that
			\begin{gather}
				u_i(y_{\epsilon},x_{\epsilon})-v_i(y_{\epsilon},\tilde{x}_{\epsilon})-\phi_{\epsilon}(y_{\epsilon},x_{\epsilon},\tilde{x}_{\epsilon})\ge\delta\nonumber\\
				\Longleftrightarrow u_i(y_{\epsilon},x_{\epsilon})-v_i(y_{\epsilon},\tilde{x}_{\epsilon})-\Big{(}\frac{1}{2\epsilon}\abs{x_{\epsilon}-\tilde{x}_{\epsilon}}^2+\abs{x_{\epsilon}-x_{i_0}}^4+\abs{\tilde{x}_{\epsilon}-x_{i_0}}^4+\abs{y_{\epsilon}-y_{i_0}}^2\Big{)} \ge\delta\nonumber\\
				\Longrightarrow u_i(y_{\epsilon},x_{\epsilon})-v_i(y_{\epsilon},\tilde{x}_{\epsilon})-\frac{1}{2\epsilon}\abs{x_{\epsilon}-\tilde{x}_{\epsilon}}^2\ge\delta\nonumber\\
				\Longleftrightarrow \frac{1}{2\epsilon}\abs{x_{\epsilon}-\tilde{x}_{\epsilon}}^2\leq u_i(y_{\epsilon},x_{\epsilon})-v_i(y_{\epsilon},\tilde{x}_{\epsilon})-\delta\label{eq_4}.
			\end{gather}
		\end{proof}
		 We consider a random sequence $ \left( \epsilon_{\tilde{k}}\right)_{\tilde{k}\in\N}  $ in $ (0,\infty) $ such that $ \lim_{\tilde{k}\to\infty}\epsilon_{\tilde{k}}=0. $ Hence, a corresponding sequence of points $ \left\lbrace \left( y_{\epsilon_{\tilde{k}}},x_{\epsilon_{\tilde{k}}},\tilde{x}_{\epsilon_{\tilde{k}}}\right)_{\tilde{k}\in\N}\right\rbrace \subset[0,L]\times\bar{\Omega}\times\bar{\Omega}$ will exists. Since the  $ [0,L]\times\bar{\Omega}\times\bar{\Omega} $ is compact, it will  simultaneously also be sequentially compact. Therefore, the sequence of points $  \left( y_{\epsilon_{\tilde{k}}},x_{\epsilon_{\tilde{k}}},\tilde{x}_{\epsilon_{\tilde{k}}}\right)_{\tilde{k}\in\N}  $ has a subsequence, let's say $ \left( y_{\epsilon_k},x_{\epsilon_k},\tilde{x}_{\epsilon_k}\right)_{k\in\N}, $ such that, $$ \left( y_{\epsilon_k},x_{\epsilon_k},\tilde{x}_{\epsilon_k}\right)\xrightarrow{k\to\infty}\left(\hat{y},\hat{x},\hat{\tilde{x}} \right) \in[0,L]\times\bar{\Omega}\times\bar{\Omega} $$ and in addition it must be $ \lim_{k\to\infty}\epsilon_k=0 \quad \left(\text{since}\left( \epsilon_k\right)_{k\in\N}\ \text{is a subsequence of} \left(\epsilon_{\tilde{k}} \right)_{\tilde{k}\in\N}  \right).$\\
		
		\begin{claim}\label{claim_2}
			$$\lim_{k\to\infty}\abs{x_{\epsilon_k}-\tilde{x}_{\epsilon_k}}=0 \text{ and } \hat{x}=\hat{\tilde{x}}$$
		\end{claim}
		\begin{proof}[ Proof]
			We set $ M:=\max_{[0,L]\times\bar{\Omega}}u(y,x)-\min_{[0,L]\times\bar{\Omega}}v(y,\tilde{x}) $. Then $ \text{for all\ } k\in\N $ we have \\ $ u\left( y_{\epsilon_k},x_{\epsilon_k}\right)- v\left( y_{\epsilon_k},\tilde{x}_{\epsilon_k}\right) \leq M. $ From $ (\ref{eq_4}) $, the following holds
			\begin{align}
				0\leq\abs{x_{\epsilon_k}-\tilde{x}_{\epsilon_k}}^2&\leq 2\ \epsilon_k \left( u_i\left(y_{\epsilon_k},x_{\epsilon_k} \right)-v_i\left(y_{\epsilon_k},\tilde{x}_{\epsilon_k} \right)-\delta  \right) \nonumber\\
				&\leq 2\ \epsilon_k \left( M-\delta\right) <^{\left(\delta>0 \right) } 2\ \epsilon_k M\xrightarrow{k\to\infty} 0\label{eq_5}.
			\end{align}
			From $ (\ref{eq_5}) $ we get $ \lim_{k\to\infty}\abs{x_{\epsilon_k}-\tilde{x}_{\epsilon_k}}=0.$ But at the same time, $ x_{\epsilon_k}\xrightarrow{k\to\infty}\hat{x} $ and $ \tilde{x}_{\epsilon_k}\xrightarrow{k\to\infty}\hat{y} $. Therefore, $ \abs{x_{\epsilon_k}-\tilde{x}_{\epsilon_k}}\xrightarrow{k\to\infty} \abs{\hat{x}-\hat{\tilde{x}}}.$
			In consequence, from the uniqueness of limit, we have $ \hat{x}=\hat{\tilde{x}}. $
		\end{proof}
		
		\begin{claim}\label{claim_3}
			$$ \lim_{k\to\infty}\frac{\abs{x_{\epsilon_k}-\tilde{x}_{\epsilon_k}}^2}{ 2\ \epsilon_k} =0$$
		\end{claim}
		\begin{proof}[ Proof]
				Because $ u\in USC\left([0,L]\times\bar{\Omega} \right),  $ from the sequential criterion for upper semicontinuous functions, the following holds, $ \limsup_{k\to\infty} u\left( y_{\epsilon_k},x_{\epsilon_k}\right)\leq u\left( \hat{y},\hat{x}\right)   $. Correspondingly and for\\ $ v\in LSC\left([0,L]\times\bar{\Omega} \right) $, from the sequential criterion for lower semicontinuous functions, the following holds $ \liminf_{k\to\infty} v\left( y_{\epsilon_k},\tilde{x}_{\epsilon_k}\right)\ge v\left( \hat{y},\hat{\tilde{x}}\right). $ In addition
			\begin{align}
				0\leq\liminf_{k\to\infty}\frac{\abs{x_{\epsilon_k}-\tilde{x}_{\epsilon_k}}^2}{ 2\ \epsilon_k}&\leq\limsup_{k\to\infty}\frac{\abs{x_{\epsilon_k}-\tilde{x}_{\epsilon_k}}^2}{ 2\ \epsilon_k}\nonumber\\
				&\leq^{(\ref{eq_4})} \limsup_{k\to\infty}(u_i( y_{\epsilon_k},x_{\epsilon_k})-v_i( y_{\epsilon_k},\tilde{x}_{\epsilon_k})) -\delta\nonumber\\
				&\leq\limsup_{k\to\infty} u_i( y_{\epsilon_k},x_{\epsilon_k})+\limsup_{k\to\infty}(-v_i( y_{\epsilon_k},\tilde{x}_{\epsilon_k})) -\delta\nonumber\\
				&=\limsup_{k\to\infty} u_i(y_{\epsilon_k},x_{\epsilon_k})-\liminf_{k\to\infty}(v_i( y_{\epsilon_k},\tilde{x}_{\epsilon_k})-\delta\nonumber\\
				&\leq u_i( \hat{y},\hat{x})-v_i( \hat{y},\hat{\tilde{x}})-\delta\nonumber\\ &=^{(\hat{x}=\hat{\tilde{x}} )}u_i(\hat{y},\hat{x})-v_i(\hat{y},\hat{x})-\delta\leq 0 \nonumber 
			\end{align}
			where the last inequality holds from the fact that, $ i\in I \text{\ and thus,}\  \delta=\sup_{[0,L]\times\bar{\Omega}}(u_i-v_i ) ( y,x).  $ Therefore, we conclude that 
			$$\liminf_{k\to\infty}\frac{\abs{x_{\epsilon_k}-\tilde{x}_{\epsilon_k}}^2}{ 2\ \epsilon_k}=0=\limsup_{k\to\infty}\frac{\abs{x_{\epsilon_k}-\tilde{x}_{\epsilon_k}}^2}{ 2\ \epsilon_k}$$
			thus, $ \lim_{k\to\infty}\frac{\abs{x_{\epsilon_k}-\tilde{x}_{\epsilon_k}}^2}{ 2\ \epsilon_k}=0 $
		\end{proof}
		
		\begin{claim}\label{claim_4}
			$$ \lim_{k\to\infty}y_{\epsilon_k}=y_{i_0}, \lim_{k\to\infty}x_{\epsilon_k}=x_{i_0}, \lim_{k\to\infty}\tilde{x}_{\epsilon_k}=x_{i_0} $$
		\end{claim}
		\begin{proof}[ Proof]
			We had seen in Claim \ref{claim_1} that $ \Phi_{\epsilon}( y_{\epsilon},x_{\epsilon},\tilde{x}_{\epsilon})\ge\delta  $. Therefore, it applies that: \\
			$ \Phi_{\epsilon_k}( y_{\epsilon_k},x_{\epsilon_k},\tilde{x}_{\epsilon_k})\ge\delta  $. Thus, we obtain that:
			$$u_i(y_{\epsilon_k},x_{\epsilon_k})-v_i(y_{\epsilon_k},\tilde{x}_{\epsilon_k})-\Big{( }\frac{1}{2\epsilon_k}\abs{x_{\epsilon_k}-\tilde{x}_{\epsilon_k}}^2+\abs{x_{\epsilon_k}-x_{i_0}}^4+\abs{\tilde{x}_{\epsilon_k}-x_{i_0}}^4+\abs{y_{\epsilon_k}-y_{i_0}}^2\Big{)} \ge\delta.$$
			From the previous inequality, we receive that:
			\begin{gather}
				u_i(y_{\epsilon_k},x_{\epsilon_k})-v_i(y_{\epsilon_k},\tilde{x}_{\epsilon_k})-\delta\ge \abs{y_{\epsilon_k}-y_{i_0}}^2 \nonumber\\
				\wedge\	u_i(y_{\epsilon_k},x_{\epsilon_k})-v_i(y_{\epsilon_k},\tilde{x}_{\epsilon_k})-\delta\ge \abs{x_{\epsilon_k}-x_{i_0}}^4 \nonumber\\
				\wedge\ u_i(y_{\epsilon_k},x_{\epsilon_k})-v_i(y_{\epsilon_k},\tilde{x}_{\epsilon_k})-\delta\ge \abs{\tilde{x}_{\epsilon_k}-x_{i_0}}^4. \label{eq_6}
			\end{gather}
			We will prove the first of the three limits $ \lim_{k\to\infty}y_{\epsilon_k}=y_{i_0} $. The rest can be proved analogously. We observe that, similarly with the proof of the Claim \ref{claim_3}, the following holds: 
			\begin{gather}
				\limsup_{k\to\infty}( u_i(y_{\epsilon_k},x_{\epsilon_k})-v_i(y_{\epsilon_k},\tilde{x}_{\epsilon_k})) -\delta\leq 0 \label{eq_7}.
			\end{gather}
			Furthermore, from the first inequality of $ (\ref{eq_6}) $ we have that:
			\begin{gather}
				\limsup_{k\to\infty}\left( u_i(y_{\epsilon_k},x_{\epsilon_k})-v_i(y_{\epsilon_k},\tilde{x}_{\epsilon_k})\right) -\delta\ge\limsup_{k\to\infty}\abs{y_{\epsilon_k}-y_{i_0}}^4\ge 0 \label{eq_8}.
			\end{gather}
			From $(\ref{eq_7})  $ and  $ (\ref{eq_8}) $ we receive that $ \limsup_{k\to\infty}\abs{y_{\epsilon_k}-y_{i_0}}^4=0 $. Additionally, because\\ $\limsup_{k\to\infty}\abs{y_{\epsilon_k}-y_{i_0}}^4\ge \liminf_{k\to\infty}\abs{y_{\epsilon_k}-y_{i_0}}^4 \ge 0  $, it follows that:
			$$\limsup_{k\to\infty}\abs{y_{\epsilon_k}-y_{i_0}}^4=0=\liminf_{k\to\infty}\abs{y_{\epsilon_k}-y_{i_0}}^4.$$
			Thus, $ \lim_{k\to\infty}\abs{y_{\epsilon_k}-y_{i_0}}^4=0 $ equivalently $  \lim_{k\to\infty}y_{\epsilon_k}=y_{i_0} $.
		\end{proof}
		
		From the above Claim, combined with the uniqueness of the limit, we have that $ y_{i_0}=\hat{y} $ and $ x_{i_0}=\hat{x} $. Furthermore, due to the fact that $ x_{i_0}\in\Omega $ and $ y_{i_0}\in(0,L) $, it follows that $ x_{i_0},y_{i_0}   $  are internal points of $ \Omega $  and $(0,L) $ respectively. Then, from the Claim \ref{claim_4}, an index  $ k^*\in\N $ will exist, such that:
		\begin{gather} \text{for all\ } k\ge k^*,\ x_{\epsilon_k}\in\Omega\ \wedge\ \tilde{x}_{\epsilon_k}\in\Omega\ \wedge\ y_{\epsilon_k}\in (0,L). \label{eq_8(1)}
		\end{gather}
		
		\begin{claim}\label{claim_5}
			$$  \limsup_{k\to\infty}u_i\left(y_{\epsilon_k},x_{\epsilon_k} \right)=u_i\left(y_{i_0},x_{i_0} \right) \text{ and } \liminf_{k\to\infty}v_i\left(y_{\epsilon_k},\tilde{x}_{\epsilon_k} \right)=v_i\left(y_{i_0},x_{i_0} \right) $$
		\end{claim}
		\begin{proof}[Proof]
			 From the Claim \ref{claim_4} it holds that$ \left(y_{\epsilon_k},x_{\epsilon_k} \right)\xrightarrow{k\to\infty}\left( y_{i_0},x_{i_0}\right)  $ and $ \left(y_{\epsilon_k},\tilde{x}_{\epsilon_k} \right)\xrightarrow{k\to\infty}\left( y_{i_0},x_{i_0}\right). $ Because  $ u_i $ and $ v_i $ are upper and lower semicontinuous respectively, from the sequential criterion, we obtain that:
			$ \limsup_{k\to\infty}u_i\left( y_{\epsilon_k},x_{\epsilon_k}\right)\leq u_i\left(y_{i_0},x_{i_0} \right)  $ and $ \liminf_{k\to\infty} v_i\left( y_{\epsilon_k},\tilde{x}_{\epsilon_k}\right)\geq v_i\left(y_{i_0},x_{i_0} \right). $ At first, we will prove that $\limsup_{k\to\infty}u_i\left(y_{\epsilon_k},x_{\epsilon_k} \right)=u_i\left(y_{i_0},x_{i_0} \right)$. We suppose by contradiction, that the output is not true. Then, it is necessary from the above that $\limsup_{k\to\infty}u_i\left(y_{\epsilon_k},x_{\epsilon_k} \right)<u_i\left(y_{i_0},x_{i_0} \right)$. As previously, we have $ \text{for all\ } k\in\N $
			\begin{gather}
				u_i\left( y_{i_0},x_{i_0}\right) -v_i\left(y_{i_0},x_{i_0}\right)=\Phi_{\epsilon_k}\left(y_{i_0},x_{i_0},x_{i_0} \right)\leq \Phi_{\epsilon_k}\left(y_{\epsilon_k},x_{\epsilon_k},\tilde{x}_{\epsilon_k} \right) \leq u_i\left( y_{\epsilon_k},x_{\epsilon_k}\right) -v_i\left(y_{\epsilon_k},\tilde{x}_{\epsilon_k} \right)\label{eq_9}.
			\end{gather}
			Then:
			\begin{align}
				u_i\left( y_{i_0},x_{i_0}\right) -v_i\left(y_{i_0},x_{i_0}\right)&=\limsup_{k\to\infty}\left(	u_i\left( y_{i_0},x_{i_0}\right)-v_i\left(y_{i_0},x_{i_0}\right) \right)\nonumber\\
				&\leq^{(\ref{eq_9})}\limsup_{k\to\infty}\left(  u_i\left( y_{\epsilon_k},x_{\epsilon_k}\right) -v_i\left(y_{\epsilon_k},\tilde{x}_{\epsilon_k} \right)\right) \nonumber\\
				&\leq \limsup_{k\to\infty}\left(  u_i\left( y_{\epsilon_k},x_{\epsilon_k}\right)\right)  -\liminf_{k\to\infty}\left( v_i\left(y_{\epsilon_k},\tilde{x}_{\epsilon_k} \right)\right)\nonumber\\
				&<u_i\left(y_{i_0},x_{i_0} \right)-\liminf_{k\to\infty}\left( v_i\left(y_{\epsilon_k},\tilde{x}_{\epsilon_k} \right)\right)\nonumber.
			\end{align}
			Thus, $ v_i\left(y_{i_0},x_{i_0}\right)>\liminf_{k\to\infty}\left( v_i\left(y_{\epsilon_k},\tilde{x}_{\epsilon_k} \right)\right) $, which is a contradiction, because $ v_i $ is lower semicontinuous. Therefore, we have $\limsup_{k\to\infty}u_i\left(y_{\epsilon_k},\tilde{x}_{\epsilon_k} \right)=u_i\left(y_{i_0},x_{i_0} \right)$. 
			Correspondigly, if we make the assumption that the following is not true, $ \liminf_{k\to\infty}v_i\left(y_{\epsilon_k},\tilde{x}_{\epsilon_k} \right)=v_i\left(y_{i_0},x_{i_0} \right), $ then from the above we receive that $\liminf_{k\to\infty}v_i\left(y_{\epsilon_k},\tilde{x}_{\epsilon_k} \right)>v_i\left(y_{i_0},x_{i_0} \right).  $ Like before, we get that
			\begin{align} 
					u_i\left( y_{i_0},x_{i_0}\right) -v_i\left(y_{i_0},x_{i_0}\right)&\leq\limsup_{k\to\infty}\left(  u_i\left( y_{\epsilon_k},x_{\epsilon_k}\right)\right)  -\liminf_{k\to\infty}\left( v_i\left(y_{\epsilon_k},\tilde{x}_{\epsilon_k} \right)\right)\nonumber\\
					&<\limsup_{k\to\infty}\left(u_i\left( y_{\epsilon_k},x_{\epsilon_k}\right)\right)-v_i\left(y_{i_0},x_{i_0}\right).
			\end{align}
			Hence, $ u_i\left( y_{i_0},x_{i_0}\right)<\limsup_{k\to\infty}\left(u_i\left( y_{\epsilon_k},x_{\epsilon_k}\right)\right) $ which is a contradiction, because  $ u_i $ is upper semicontinouous. In consequence, $ \liminf_{k\to\infty}v_i\left(y_{\epsilon_k},x_{\epsilon_k} \right)=v_i\left(y_{i_0},x_{i_0} \right). $
		\end{proof}
		
		\textbf{ Step 2}: $ \left(\text{Obstacle Avoidance} \right) $ \\
		 First, we will prove that exists at least one component $ u_j $ of the subsolution $ u $ which lies strictly above its obstacle $ \mathcal{ M}_{j}u $ at point $ \left(y_{i_0},x_{i_0} \right). $ 
		In specific, we have the following claim.
		\begin{claim}\label{claim_6}
			\begin{gather}
				\text{There exists} \ j_0\in I,\ u_{j_0}\left(y_{i_0},x_{i_0} \right)>\mathcal{ M}_{j_0} u\left( y_{i_0},x_{i_0}\right).\label{eq_2-1}
			\end{gather}
		\end{claim}
		\begin{proof}[Proof]
			 We make the assumption that the above claim is false. Therefore: 
			\begin{gather}
				\text{for all\ } i\in I, u_i\left( y_{i_0},x_{i_0}\right)\leq\mathcal{M}_i u\left(y_{i_0},x_{i_0} \right) =\max_{j\neq i}\left( u_j\left( y_{i_0},x_{i_0}\right) -c_{i,j}\left( y_{i_0},x_{i_0}\right) \right)\nonumber
			\end{gather}
			We randomly select a fixed index $ l_1\in I $. Subsequently, an index  $ l_2\in\left\lbrace 1,2,\dots,l_1-1,l_1+1,\dots,m\right\rbrace  $ exists, such that: 
			\begin{gather}
				u_{l_1}\left( y_{i_0},x_{i_0}\right)\leq u_{l_2}\left( y_{i_0},x_{i_0}\right) -c_{l_1,l_2}\left( y_{i_0},x_{i_0}\right) =\max_{j\neq l_1}\left( u_j\left( y_{i_0},x_{i_0}\right) -c_{l_1,j}\left( y_{i_0},x_{i_0}\right)\right) \nonumber\\
				\iff u_{l_1}\left( y_{i_0},x_{i_0}\right)+c_{l_1,l_2}\left( y_{i_0},x_{i_0}\right)\leq u_{l_2}\left( y_{i_0},x_{i_0}\right)\label{eq_2-1_1}.
			\end{gather}
			Due to the fact that $ v $ is a supersolution, it follows that: 
			\begin{gather}
				v_{l_1}\left( y_{i_0},x_{i_0} \right)\geq\mathcal{M}_{l_1} v\left( y_{i_0},x_{i_0}\right) 
				=\max_{j\neq l_1}\left( v_j\left( y_{i_0},x_{i_0}\right) -c_{l_1,j}\left( y_{i_0},x_{i_0}\right)\right) 
				\geq v_{l_2}\left( y_{i_0},x_{i_0}\right) -c_{l_1,l_2}\left( y_{i_0},x_{i_0}\right) \nonumber
			\end{gather}
			thus
			\begin{gather}
				v_{l_2}\left(y_{i_0},x_{i_0}\right) -c_{l_1,l_2}\left( y_{i_0},x_{i_0}\right)\leq v_{l_1}\left( y_{i_0},x_{i_0} \right)\label{eq_2-1_2}.
			\end{gather}
			By adding $ (\ref{eq_2-1_1}) $ and $ (\ref{eq_2-1_2}) $ we obtain:
			\begin{gather}
				u_{l_1}\left( y_{i_0},x_{i_0}\right)-v_{l_1}\left(y_{i_0},x_{i_0}\right)\leq u_{l_2}\left( y_{i_0},x_{i_0}\right)-v_{l_2}\left( y_{i_0},x_{i_0}\right)\label{eq_2-1_3}.
			\end{gather}
			Since $ l_1\in I, $ we have that $ \delta=u_{l_1}\left( y_{i_0},x_{i_0}\right)-v_{l_1}\left(y_{i_0},x_{i_0}\right) $ and at the same time we have that:
			$$u_{l_2}\left( y_{i_0},x_{i_0}\right)-v_{l_2}\left( y_{i_0},x_{i_0}\right)\leq\sup_{[0,L]\times\bar{\Omega}}\left(u_{l_2}-v_{l_2} \right)\left(y,x \right)\leq\max_{j\in\left\lbrace1,2,\dots,m \right\rbrace } \sup_{[0,L]\times\bar{\Omega}}\left(u_j-v_j \right)\left(y,x \right)=\delta. $$
			Considering the above, combined with $ (\ref{eq_2-1_3}), $ we get that $ u_{l_2}\left( y_{i_0},x_{i_0}\right)-v_{l_2}\left( y_{i_0},x_{i_0}\right)=\delta, $ therefore $ l_2\in I\setminus\left\lbrace l_1\right\rbrace . $
			Next, we repeat the above procedure, considering that $ l_1 $ is replaced  with  $ l_2 $, and we obtain the existence of $ l_3\in I\setminus\left\lbrace l_2 \right\rbrace  $ such that:
			\begin{gather}
				u_{l_2}\left( y_{i_0},x_{i_0} \right) \leq u_{l_3}\left( y_{i_0},x_{i_0} \right) - c_{l_2,l_3}\left(y_{i_0},x_{i_0} \right) \label{eq_2-1_4}.
			\end{gather}
			Combining  $ (\ref{eq_2-1_1}) $ and $ (\ref{eq_2-1_4}) $ we obtain:
			\begin{gather}
				u_{l_1}\left( y_{i_0},x_{i_0}\right)+c_{l_1,l_2}\left( y_{i_0},x_{i_0}\right)\leq u_{l_2}\left( y_{i_0},x_{i_0}\right)\leq u_{l_3}\left( y_{i_0},x_{i_0} \right) - c_{l_2,l_3}\left(y_{i_0},x_{i_0} \right)\nonumber
			\end{gather}
			thus, we have that
			\begin{gather}
				u_{l_1}\left( y_{i_0},x_{i_0}\right)+c_{l_1,l_2}\left(y_{i_0},x_{i_0} \right)+c_{l_2,l_3}\left(y_{i_0},x_{i_0} \right)\leq u_{l_3}\left( y_{i_0},x_{i_0} \right).
			\end{gather}
			 Repeating this procedure, as many times as needed, we construct in this way a finite sequence of indexes $ \left\lbrace l_1,l_2,\dots,l_p,l_1\right\rbrace  $ with $ l_j\neq l_{j+1}, \text{for all\ } j=1,2,\dots,p-1 $ such that:
			$$u_{l_1}\left( y_{i_0},x_{i_0}\right)+c_{l_1,l_2}\left(y_{i_0},x_{i_0} \right)+c_{l_2,l_3}\left(y_{i_0},x_{i_0} \right)+\dots,c_{l_p,l_1}\left(y_{i_0},x_{i_0} \right)\leq u_{l_1}\left( y_{i_0},x_{i_0} \right)$$
			from which it is extracted that: 
			$$c_{l_1,l_2}\left(y_{i_0},x_{i_0} \right)+c_{l_2,l_3}\left(y_{i_0},x_{i_0} \right)+\dots,c_{l_p,l_1}\left(y_{i_0},x_{i_0} \right)\leq 0$$
			which leads to a contradiction with respect to the axiom $ (O_3) $. In consequence,  at least one index $ j_0\in I $ exists, such that: $ u_{j_0}\left(y_{i_0},x_{i_0} \right)>\mathcal{ M}_{j_0} u\left( y_{i_0},x_{i_0}\right). $
		\end{proof}
			Next, throughtout the analysis that follows, for the $ j_0\in I $ of the Claim  \ref{claim_6},  all the results that were proved at step 1, can be applied, keeping the same symbols for the sequence $ \left( y_{\epsilon},x_{\epsilon},\tilde{x}_{\epsilon}\right)  $.
		Using  Claim \ref{claim_6}, it is proved within the next Claim \ref{claim_7}, that a subsequence exists, $ \left(\mu_k \right)_{k\in\N}   $, of the sequence $\left( \epsilon_k\right)_{k\in\N},   $ and an index $ \hat{k}\in\N $ such that $ \text{for all\ } k\ge\hat{k}, u_{j_0}\left(y_{\mu_k},x_{\mu_k} \right)>\mathcal{ M}_{j_0} u\left( y_{\mu_k},x_{\mu_k}\right). $ As will be demonstrated later, the previous Claim, has the defining impact to finilize negative sign of function $ F $, when it acts on the sequence: $ \left( y_{\mu_k},x_{\mu_k},u_{j_0}\left(y_{\mu_k},x_{\mu_k} \right),D_x\phi_{\mu_k}\left(y_{\mu_k},x_{\mu_k},\tilde{x}_{\mu_k} \right),X  \right)_{k\in\N}, $ given that the following is true: $ \left( \alpha,D_x\phi_{\mu_k}\left(y_{\mu_k},x_{\mu_k},\tilde{x}_{\mu_k} \right),X\right) \in\bar{J}^{2,+}u_{j_0}\left( y_{\mu_k},x_{\mu_k}\right). $ A corresponding information holds for the non negativity of function $ F. $
		
		\begin{claim}\label{claim_7}
			 There exists a subsequence $(\mu_k)_{k\in\N}$ of $(\epsilon_k)_{k\in\N} $ and an index $\hat{k}\in\N$ such that $\text{for all\ } k\ge\hat{k}$
		\begin{itemize}
			\item The sequence of points $ \left(y_{\mu_k},x_{\mu_k}, \tilde{x}_{\mu_k} \right)  $ is inside of $ (0,L)\times\Omega\times\Omega. $
			\item If $ \left( \alpha,D_x\phi_{\mu_k}\left(y_{\mu_k},x_{\mu_k},\tilde{x}_{\mu_k}\right),X \right) \in\bar{J}^{2,+}u_{j_0}\left( y_{\mu_k},x_{\mu_k}\right) $, then
			\begin{gather}
				F\left( y_{\mu_k},x_{\mu_k},u_{j_0}\left(y_{\mu_k},x_{\mu_k} \right),D_x\phi_{\mu_k}\left(y_{\mu_k},x_{\mu_k},\tilde{x}_{\mu_k} \right),X  \right) \leq 0 \label{eq_2-2}.
			\end{gather}
			\item If $ \left( \alpha,-D_{\tilde{x}}\phi_{\mu_k}\left(y_{\mu_k},x_{\mu_k},\tilde{x}_{\mu_k} \right),Y\right) \in\bar{J}^{2,-}v_{j_0}\left( y_{\mu_k},\tilde{x}_{\mu_k}\right) $, then
			\begin{gather}
				F\left( y_{\mu_k},\tilde{x}_{\mu_k},v_{j_0}\left(y_{\mu_k},\tilde{x}_{\mu_k} \right),-D_{\tilde{x}}\phi_{\mu_k}\left(y_{\mu_k},x_{\mu_k},\tilde{x}_{\mu_k} \right),Y  \right) \ge 0 \label{eq_2-3}.
			\end{gather}
		\end{itemize}

	\end{claim}
	
	\begin{proof}[Proof]
			From Claim \ref{claim_6},  we get that $u_{j_0}\left(y_{i_0},x_{i_0} \right)>\mathcal{ M}_{j_0} u\left( y_{i_0},x_{i_0}\right)$. In addition, we have that $ \text{for all\ } j\in{1,2,\dots,m} $, the function $ u_j $ is upper semicontinuous. Therefore, from the sequential criterion for upper semicontinuous functions, we have that $ \limsup_{k\to\infty}u_j\left(y_{\epsilon_k},x_{\epsilon_k} \right)\leq u_j\left(y_{i_0},x_{i_0} \right).  $ In specific, if $ j\in I $, then from the Claim \ref{claim_5} we get that $ \limsup_{k\to\infty}u_j\left(y_{\epsilon_k},x_{\epsilon_k} \right)= u_j\left(y_{i_0},x_{i_0} \right). $
		We observe that 
		\begin{align}
			\mathcal{ M}_{j_0}u\left( y_{i_0},x_{i_0}\right)&=\max_{j\ne j_0}\left(u_j\left(y_{i_0},x_{i_0} \right) -c_{j_0,j}\left(y_{i_0},x_{i_0} \right)  \right)\nonumber\\
			&=\max_{j\ne j_0}\left(\limsup_{k\to\infty} u_j\left(y_{\epsilon_k},x_{\epsilon_k} \right) -c_{j_0,j}\left(y_{i_0},x_{i_0} \right)  \right)\nonumber\\
			&=\max_{j\ne j_0}\limsup_{k\to\infty} \left(u_j\left(y_{\epsilon_k},x_{\epsilon_k} \right) -c_{j_0,j}\left(y_{\epsilon_k},x_{\epsilon_k} \right)  \right)\nonumber\\
			&\geq\limsup_{k\to\infty}\max_{j\ne j_0} \left(u_j\left(y_{\epsilon_k},x_{\epsilon_k} \right) -c_{j_0,j}\left(y_{\epsilon_k},x_{\epsilon_k} \right)  \right)\nonumber\\
			&=\limsup_{k\to\infty}\mathcal{ M}_{j_0}u\left( y_{\epsilon_k},x_{\epsilon_k}\right). 
		\end{align}
		Thus, $ \mathcal{ M}_{j_0}u\left( y_{i_0},x_{i_0}\right)\geq \limsup_{k\to\infty}\mathcal{ M}_{j_0}u\left( y_{\epsilon_k},x_{\epsilon_k}\right) $ and since $$\limsup_{k\to\infty}u_{j_0}\left(y_{\epsilon_k},x_{\epsilon_k} \right)=^{\left(j_0\in I \right)} u_{j_0}\left(y_{i_0},x_{i_0} \right)>\mathcal{ M}_{j_0} u\left( y_{i_0},x_{i_0}\right),$$ we obtain that 
		\begin{gather}
			\limsup_{k\to\infty}u_{j_0}\left(y_{\epsilon_k},x_{\epsilon_k} \right)>\limsup_{k\to\infty}\mathcal{ M}_{j_0}u\left( y_{\epsilon_k},x_{\epsilon_k}\right).
		\end{gather}
		Now here exists a strictly increasing sequence of natural numbers $ \left( \sigma_k\right)_{k\in\N}  $ such that, 
		\begin{gather}
			\lim_{k\to\infty} u_{j_0}\left(y_{\epsilon_{\sigma_k}},x_{\epsilon_{\sigma_k}} \right)>\lim_{k\to\infty}\mathcal{ M}_{j_0}u\left(y_{\epsilon_{\sigma_k}},x_{\epsilon_{\sigma_k}}\right)\label{eq_2-4}.
		\end{gather}
		Then, from the above inequality, we can extract an index $ \tilde{k}\in\N $ such that $ \text{for all\ } k\geq\tilde{k} $
		\begin{gather}
			u_{j_0}\left(y_{\epsilon_{\sigma_k}},x_{\epsilon_{\sigma_k}} \right)>\mathcal{ M}_{j_0}u\left(y_{\epsilon_{\sigma_k}},x_{\epsilon_{\sigma_k}}\right)\label{eq_2-5}
		\end{gather}
		and from now on, for practical reasons, the subsequence, $ \left( \epsilon_{\sigma_k}\right) _{k\in\N} $ of $ \left( \epsilon_k\right)_{k\in\N}, $ will be denoted as $\left(  \mu_k\right)_{k\in\N}  $.
		From  $ (\ref{eq_8(1)}) $ and $ (\ref{eq_2-5}) $, if we set  $ \hat{k}=\max\left\lbrace k^*,\tilde{k}\right\rbrace  $, we have that $ \text{for all\ } k\geq\hat{k} $
		\begin{gather}
			u_{j_0}\left(y_{\mu_k},x_{\mu_k} \right)>\mathcal{ M}_{j_0}u\left(y_{\mu_k},x_{\mu_k}\right)\label{eq_2-6}\\
			\text{and}\nonumber\\
			x_{\mu_k}\in\Omega\ \wedge\ \tilde{x}_{\mu_k}\in\Omega\ \wedge y_{\mu_k}\in(0,L)\label{eq_2-7}.
		\end{gather}
		Next, we consider a random $ k\geq\hat{k} $. We have the following cases:
		\begin{itemize}
			\item Let $ \left( \alpha,D_x\phi_{\mu_k}\left(y_{\mu_k},x_{\mu_k},\tilde{x}_{\mu_k}\right),X \right) \in\bar{J}^{2,+}u_{j_0}\left( y_{\mu_k},x_{\mu_k}\right) $. Because $ u $ is  subsolution of the problem $ (IBVP) $ and $ \left(y_{\mu_k},x_{\mu_k},\tilde{x}_{\mu_k}\right)\in (0,L) \times\Omega\times\Omega $ from $ (\ref{eq_2-7}) $, it follows that: 
			$$\min\left\lbrace F\left(y_{\mu_k},x_{\mu_k},u_{j_0}\left(y_{\mu_k},x_{\mu_k} \right),D_x\phi_{\mu_k}\left(y_{\mu_k},x_{\mu_k},\tilde{x}_{\mu_k}\right),X  \right),u_{j_0}\left( y_{\mu_k},x_{\mu_k}\right) -\mathcal{ M}_{j_0}u\left(y_{\mu_k},x_{\mu_k} \right)  \right\rbrace \leq 0$$
			But from $ (\ref{eq_2-6}) $ it is necessary that:
			\begin{gather}
				F\left(y_{\mu_k},x_{\mu_k},u_{j_0}\left(y_{\mu_k},x_{\mu_k} \right),D_x\phi_{\mu_k}\left(y_{\mu_k},x_{\mu_k},\tilde{x}_{\mu_k}\right),X  \right)\leq 0 \label{eq_2-8}.
			\end{gather}
			\item Let $ \left( \alpha,-D_{\tilde{x}}\phi_{\mu_k}\left(y_{\mu_k},x_{\mu_k},\tilde{x}_{\mu_k} \right),Y\right) \in\bar{J}^{2,-}v_{j_0}\left( y_{\mu_k},\tilde{x}_{\mu_k}\right) $. Since $ v $ is a supersolution of the problem $ (BVP) $ and $ \left(y_{\mu_k},x_{\mu_k},\tilde{x}_{\mu_k}\right)\in (0,L) \times\Omega\times\Omega $ from $ (\ref{eq_2-7}) $, it follows that
			\begin{gather}
				\min\biggl\{ F\left(y_{\mu_k},\tilde{x}_{\mu_k},v_{j_0}\left(y_{\mu_k},\tilde{x}_{\mu_k} \right),-D_{\tilde{x}}\phi_{\mu_k}\left(y_{\mu_k},x_{\mu_k},\tilde{x}_{\mu_k}\right),Y  \right),\\ v_{j_0}\left( y_{\mu_k},\tilde{x}_{\mu_k}\right) -\mathcal{ M}_{j_0}v\left(y_{\mu_k},\tilde{x}_{\mu_k} \right)  \biggl\} \geq 0.
			\end{gather}
			
			Therefore, we conclude from the above that
			\begin{gather}
				F\left(y_{\mu_k},\tilde{x}_{\mu_k},v_{j_0}\left(y_{\mu_k},\tilde{x}_{\mu_k} \right),-D_{\tilde{x}}\phi_{\mu_k}\left(y_{\mu_k},x_{\mu_k},\tilde{x}_{\mu_k}\right),Y  \right)\geq 0 \label{eq_2-9}.
			\end{gather}
			
		\end{itemize}
	\end{proof}
	
	\textbf{Step 3}: (Approaching the Contradiction) 
	
	 In this step, we demonstrate how we end up with a contradiction. In order to reach this conclusion, a special form of lemma is required, the so-called maximum principle for semicontinuous functions. In specific, the Lemma \ref{Lemma_3.1}, which is presented below, is a special result case, which corresponds to Theorem 8.3 of \cite{CIL}.

	At first, we have to mention that from  Claim \ref{claim_7},  the subsequence of maximum points $ \left( y_{\mu_k},x_{\mu_k},\tilde{x}_{\mu_k}\right)_{k\in\N}  $ of functions $\Phi_{\mu_k}\left(y,x,\tilde{x} \right) := u_{j_0}\left( y,x\right)-v_{j_0}\left(y,\tilde{x} \right) -\phi_{\mu_k}\left(y,x,\tilde{x} \right),   $ on the set, $ \left[0,L \right] \times\bar{\Omega}\times\bar{\Omega}, $ is located finally  within the set $ (0,L)\times\Omega\times\Omega. $ In consequence, the points that maximize the function  $\Phi_{\mu_k}\left(y,x,\tilde{x} \right) := u_{j_0}\left( y,x\right)-v_{j_0}\left(y,\tilde{x} \right) -\phi_{\mu_k}\left(y,x,\tilde{x} \right)   $ defined on the set $ (0,L)\times\Omega\times\Omega, $ are the same with points that maximize $ \Phi_{\mu_k} $ defined on the set  $ \left[0,L \right] \times\bar{\Omega}\times\bar{\Omega} $.
	
	\begin{lemma}\label{Lemma_3.1}
		Let $ \left( y_{\mu_k},x_{\mu_k},\tilde{x}_{\mu_k}\right)_{k\geq\hat{k}}  $ be a maximum point of 
		\begin{gather}
			\Phi_{\mu_k}\left(y,x,\tilde{x} \right) :=u_{j_0}\left( y,x\right)-v_{j_0}\left(y,\tilde{x} \right) -\phi_{\mu_k}\left(y,x,\tilde{x} \right)\nonumber
		\end{gather}
		on the set $  (0,L)\times\Omega\times\Omega. $ Then, for each $ \theta>0, $ there exists $ X_{\theta},Y_{\theta}\in \mathcal{S}^n $ such that:
		\begin{itemize}
			\item \begin{gather}
				\left(\alpha,D_x\phi_{\mu_k}\left(y_{\mu_k},x_{\mu_k},\tilde{x}_{\mu_k} \right),X_{\theta}  \right)\in \bar{J}^{2,+}u_{j_0}\left(y_{\mu_k},x_{\mu_k} \right)\label{3_1}\\
				\left(\tilde{\alpha},-D_{\tilde{x}}\phi_{\mu_k}\left(y_{\mu_k},x_{\mu_k},\tilde{x}_{\mu_k} \right),Y_{\theta}  \right)\in \bar{J}^{2,-}v_{j_0}\left(y_{\mu_k},\tilde{x}_{\mu_k} \right)\label{3_2}
			\end{gather}
			\item \begin{gather}
				\begin{pmatrix}
					X_{\theta} & O \\
					O & -Y_{\theta} 
				\end{pmatrix}
				\leq A+\theta A^2
			\end{gather}
		\end{itemize}
		where $ A:=D^2_x\phi_{\mu_k}\left(y_{\mu_k},x_{\mu_k},\tilde{x}_{\mu_k} \right)  $ is the hessian matrix of $ \phi_{\mu_k}\left(y,x,\tilde{x} \right)  $ with respect of $ x  $ and $ \tilde{x}. $
	\end{lemma}

		At first we observe that for the vectors $ x=\left(x^{(1)},x^{(2)},\dots,x^{(n)} \right)  $ and $ \tilde{x}=\left(\tilde{x}^{(1)},\tilde{x}^{(2)},\dots,\tilde{x}^{(n)} \right)  $ on $ \R^n $ holds
	\begin{gather}
		D_x\phi_{\mu_k}\left(y_{\mu_k},x_{\mu_k},\tilde{x}_{\mu_k} \right) =\frac{1}{\mu_k}\left( x_{\mu_k}-\tilde{x}_{\mu_k}\right) +4\abs{x_{\mu_k}-x_{i_0}}^2\left(x_{\mu_k}-x_{i_0} \right) \nonumber\\
		-D_{\tilde{x}}\phi_{\mu_k}\left(y_{\mu_k},x_{\mu_k},\tilde{x}_{\mu_k} \right) =\frac{1}{\mu_k}\left( x_{\mu_k}-\tilde{x}_{\mu_k}\right) -4\abs{\tilde{x}_{\mu_k}-x_{i_0}}^2\left(\tilde{x}_{\mu_k}-x_{i_0} \right)\nonumber\\
		\frac{\partial^2\phi_{\mu_k}}{\partial x^{(m)\ 2}}\left(y_{\mu_k},x_{\mu_k},\tilde{x}_{\mu_k} \right) =\frac{1}{\mu_k}+4 \left( 2\left(x_{\mu_k}^{(m)} -x_{i_0}^{(m)}\right)^2+\abs{x_{\mu_k}-x_{i_0}}^2 \right) \nonumber\\
		\frac{\partial^2\phi_{\mu_k}}{\partial \tilde{x}^{(m)\ 2}}\left(y_{\mu_k},x_{\mu_k},\tilde{x}_{\mu_k} \right) =\frac{1}{\mu_k}+4 \left( 2\left(\tilde{x}_{\mu_k}^{(m)} -x_{i_0}^{(m)}\right)^2+\abs{\tilde{x}_{\mu_k}-x_{i_0}}^2 \right) \nonumber\\
		\frac{\partial^2\phi_{\mu_k}}{\partial x^{(m)}\partial x^{(\lambda)}}\left(y_{\mu_k},x_{\mu_k},\tilde{x}_{\mu_k} \right) =8\left(x_{\mu_k}^{(m)} -x_{i_0}^{(m)}\right)\left(x_{\mu_k}^{(\lambda)} -x_{i_0}^{(\lambda)}\right),\ \lambda\neq m\nonumber\\
		\frac{\partial^2\phi_{\mu_k}}{\partial \tilde{x}^{(m)}\partial \tilde{x}^{(\lambda)}}\left(y_{\mu_k},x_{\mu_k},\tilde{x}_{\mu_k} \right) =8\left(\tilde{x}_{\mu_k}^{(m)} -x_{i_0}^{(m)}\right)\left(\tilde{x}_{\mu_k}^{(\lambda)} -x_{i_0}^{(\lambda)}\right),\ \lambda\neq m\nonumber\\
		\frac{\partial^2\phi_{\mu_k}}{\partial x^{(m)}\partial \tilde{x}^{(m)}}\left(y_{\mu_k},x_{\mu_k},\tilde{x}_{\mu_k} \right) =-\frac{1}{\mu_k}=\frac{\partial^2\phi_{\mu_k}}{\partial \tilde{x}^{(m)}\partial x^{(m)}}\left(y_{\mu_k},x_{\mu_k},\tilde{x}_{\mu_k} \right)\nonumber\\
		\frac{\partial^2\phi_{\mu_k}}{\partial x^{(m)}\partial \tilde{x}^{(\lambda)}}\left(y_{\mu_k},x_{\mu_k},\tilde{x}_{\mu_k} \right) =0=\frac{\partial^2\phi_{\mu_k}}{\partial \tilde{x}^{(\lambda)}\partial x^{(m)}}\left(y_{\mu_k},x_{\mu_k},\tilde{x}_{\mu_k} \right),\ \lambda\neq m\nonumber.
	\end{gather}
	From the above, the matrix $ A:=D^2_x\phi_{\mu_k}(y_{\mu_k},x_{\mu_k},\tilde{x}_{\mu_k} ) $ is written as:
	\begin{gather}
		A=\frac{1}{\mu_k}\left(\begin{array}{cc}
			I & -I \\
			-I & I 
		\end{array}\right)
		+ \hat{\tilde{B}}( y_{\mu_k},x_{\mu_k},\tilde{x}_{\mu_k}) \nonumber
	\end{gather}
	where $ \hat{\tilde{B}}\in\R^{2n\times 2n}, $ is an appropriate matrix, of which its elements $ \hat{\tilde{B}}_{i,j}(y,x,\tilde{x}), $ are continuous functions of non fractional type, such that $ \hat{\tilde{B}}_{i,j}(y_{\mu_k},x_{\mu_k},\tilde{x}_{\mu_k})= \hat{\tilde{h}}_{i,j}(\abs{x_{\mu_k}-x_{i_0}},\abs{\tilde{x}_{\mu_k}-x_{i_0}}) $ with $ \hat{\tilde{h}}_{i,j} $ to be an appropriate continuous function of two variables on $ \R^2. $ Based on $ A $ the matrix $ A^2 $ is defined as
	\begin{gather}
		A^2=\frac{2}{\mu_k^2}\left(\begin{array}{cc}
			I & -I \\
			-I & I
		\end{array}\right)
		+ \tilde{B}( y_{\mu_k},x_{\mu_k},\tilde{x}_{\mu_k}) \nonumber
	\end{gather}
	where $ \tilde{B}\in\R^{2n\times 2n} $ is an appropriate matrix, of which its elements $ \tilde{B}_{i,j}(y,x,\tilde{x}), $ are continuous functions, of non fractional type, such that $ \tilde{B}_{i,j}(y_{\mu_k},x_{\mu_k},\tilde{x}_{\mu_k})= \tilde{h}_{i,j}(\abs{x_{\mu_k}-x_{i_0}},\abs{\tilde{x}_{\mu_k}-x_{i_0}}) $ with $ \tilde{h}_{i,j} $ to be an appropriate continuous function of two variables on $ \R^2. $ By selecting $ \theta=\mu_k, $ from Lemma \ref{Lemma_3.1} we obtain the existence of symmetric matrices $ X_{\mu_k} $ and $ Y_{\mu_k} $ such that: 
	\begin{gather}
		\left(\alpha,D_x\phi_{\mu_k}\left(y_{\mu_k},x_{\mu_k},\tilde{x}_{\mu_k} \right),X_{\mu_k}  \right)\in \bar{J}^{2,+}u_{j_0}\left(y_{\mu_k},x_{\mu_k} \right)  \text{ and }\\ \left(\tilde{\alpha},-D_{\tilde{x}}\phi_{\mu_k}\left(y_{\mu_k},x_{\mu_k},\tilde{x}_{\mu_k} \right),Y_{\mu_k}  \right)\in \bar{J}^{2,-}v_{j_0}\left(y_{\mu_k},\tilde{x}_{\mu_k} \right) \label{3_3} 
	\end{gather}
	with \begin{gather}
		\left(\begin{array}{cc}
			X_{\mu_k} & O \\
			O & -Y_{\mu_k} 
		\end{array}\right)
		\leq A+\theta A^2=\frac{3}{\mu_k} \left(\begin{array}{cc}
			I & -I \\
			-I & I
		\end{array}\right)
		+ B( y_{\mu_k},x_{\mu_k},\tilde{x}_{\mu_k}) \label{3_4}
	\end{gather}
	where $ B\in\R^{2n\times 2n} $ is an appropriate matrix, of which its elements $ B_{i,j}(y,x,\tilde{x}) $ are continuous functions of non fractional type, such that $ B_{i,j}(y_{\mu_k},x_{\mu_k},\tilde{x}_{\mu_k})= h_{i,j}\left(\abs{x_{\mu_k}-x_{i_0}},\abs{\tilde{x}_{\mu_k}-x_{i_0}}\right) $ with $ h_{i,j} $ be an appropriate continuous function of two variables on $ \R^2$, such that, $ h_{i,j}(0,0)=0 $. Then, we obtain $ \text{for all\ } i,j\in\left\lbrace 1,2,\dots,2n\right\rbrace \lim_{k\to\infty} B_{i,j}(y_{\mu_k},x_{\mu_k},\tilde{x}_{\mu_k})=0. $
	\begin{claim}\label{claim_8}
			For each $ \xi>0 $ exists natural number $ k_{\xi}\geq\hat{k} $, such that, $ \text{for all\ } k\geq k_{\xi} $
		$ (\ref{3_4}) $ holds and
		\begin{gather}
			\begin{pmatrix}
				X_{\mu_{k}} & O \\
				O & -Y_{\mu_{k}} 
			\end{pmatrix}
			\leq\frac{3}{\mu_{k}} \begin{pmatrix}
				I & -I \\
				-I & I
			\end{pmatrix}
			+\xi \begin{pmatrix}
				I & O \\
				O & I 
			\end{pmatrix}
			\iff \begin{pmatrix}
				X_{\mu_{k}}-\xi I & O \\
				O & -\left(Y_{\mu_{k}}+\xi I \right)  
			\end{pmatrix}
			\leq\frac{3}{\mu_{k}}\begin{pmatrix}
				I & -I \\
				-I & I \label{3_5}
			\end{pmatrix}.
		\end{gather}
	\end{claim}
	\begin{proof}[ Proof]
		Let $ \xi>0 $ be a random fixed number. At first, we will prove that there exists a natural number $ k_{\xi}\geq \hat{k} $ such that $ \text{for all\ } k\geq k_{\xi} $:
		\begin{gather}
			B\left( y_{\mu_{k}},x_{\mu_{k}},\tilde{x}_{\mu_{k}} \right)< \xi \begin{pmatrix}
				I & O \\
				O & I  
			\end{pmatrix}\nonumber\\
			\iff \text{for all\ } z\in\R^{2n\times 1}\setminus\left\lbrace 0 \right\rbrace ,\ 
			z^{T}B\left( y_{\mu_{k}},x_{\mu_{k}},\tilde{x}_{\mu_{k}} \right) z-\xi\ z^{T} \begin{pmatrix}
				I & O \\
				O & I  
			\end{pmatrix}
			z< 0\label{3_6}.
		\end{gather}
		Since matrix $ M:=\xi \begin{pmatrix}
			I & O \\
			O & I  
		\end{pmatrix}  $ is positive definite, it follows that it exists $ \lambda>0 $ such that, $ \text{for all\ } z\in\R^{2n\times 1}\setminus\left\lbrace 0 \right\rbrace,\ z^T M z\geq \lambda\norm{z}^2_2 $. Furthermore, $ \text{for all\ } z\in\R^{2n\times 1},\ \text{for all\ } k\geq \hat{k}$ holds that
		\begin{gather}
			z^T B\left( y_{\mu_{k}},x_{\mu_{k}},\tilde{x}_{\mu_{k}} \right) z=\left\langle z,B\left( y_{\mu_{k}},x_{\mu_{k}},\tilde{x}_{\mu_{k}} \right) z\right\rangle\nonumber\\ \leq^{(C-S)}\norm{z}_2\norm{B\left( y_{\mu_{k}},x_{\mu_{k}},\tilde{x}_{\mu_{k}} \right) z}_2\nonumber\\
			\leq\norm{B\left( y_{\mu_{k}},x_{\mu_{k}},\tilde{x}_{\mu_{k}} \right)}\norm{z}^2_2\nonumber.
		\end{gather} 
		Now $\text{for all\ } i,j\in\left\lbrace1,2,\dots,2n \right\rbrace  B_{i,j}\left( y_{\mu_{k}},x_{\mu_{k}},\tilde{x}_{\mu_{k}} \right)\xrightarrow{k\to\infty} 0  $ holds, therefore $ \norm{B\left( y_{\mu_{k}},x_{\mu_{k}},\tilde{x}_{\mu_{k}} \right)}\xrightarrow{k\to\infty}0 $. In consequence, an index $ k_{\xi}\geq \hat{k} $ will exist, such that $ \text{for all\ } k\geq k_{\xi},\ \norm{B\left( y_{\mu_{k}},x_{\mu_{k}},\tilde{x}_{\mu_{k}} \right)}<\lambda $. Hence, we obtain that $ \text{for all\ } z\in\R^{ 2n\times 1 }\setminus\left\lbrace 0 \right\rbrace,\ \text{for all\ } k\geq k_{\xi} $
		\begin{gather}
			z^T M z\geq \lambda\norm{z}^2_2>\norm{B\left( y_{\mu_{k_{\xi}}},x_{\mu_{k_{\xi}}},\tilde{x}_{\mu_{k_{\xi}}} \right)}\norm{z}^2_2\geq z^T B\left( y_{\mu_{k_{\xi}}},x_{\mu_{k_{\xi}}},\tilde{x}_{\mu_{k_{\xi}}} \right) z.
		\end{gather}
		and $ (\ref{3_6}) $ is proved. Then, through (\ref{3_4}), using the transition identity for matrix ordering, we obtain the desirable result.
	\end{proof}
	Next, throughout our analysis, we consider that $ k\geq k_{\xi}. $ From  Claim \ref{claim_7},  we obtain that $ \text{for all\ } k\geq k_{\xi} $
	\begin{gather}
		F\left( y_{\mu_k},x_{\mu_k},u_{j_0}\left(y_{\mu_k},x_{\mu_k} \right),D_x\phi_{\mu_k}\left(y_{\mu_k},x_{\mu_k},\tilde{x}_{\mu_k} \right),X  \right) \leq 0\nonumber\\
		F\left( y_{\mu_k},\tilde{x}_{\mu_k},v_{j_0}\left(y_{\mu_k},\tilde{x}_{\mu_k} \right),-D_{\tilde{x}}\phi_{\mu_k}\left(y_{\mu_k},x_{\mu_k},\tilde{x}_{\mu_k} \right),Y  \right) \ge 0\nonumber
	\end{gather}
	from which we receive that: 
	\begin{gather}
		0\leq F\left( y_{\mu_k},\tilde{x}_{\mu_k},v_{j_0}\left(y_{\mu_k},\tilde{x}_{\mu_k} \right),-D_{\tilde{x}}\phi_{\mu_k}\left(x_{\mu_k},\tilde{x}_{\mu_k},\tilde{x}_{\mu_k} \right),Y  \right)-\nonumber\\ 
		-F\left( y_{\mu_k},x_{\mu_k},u_{j_0}\left(y_{\mu_k},x_{\mu_k} \right),D_x\phi_{\mu_k}\left(y_{\mu_k},x_{\mu_k},\tilde{x}_{\mu_k} \right),X  \right)\label{3_8}.
	\end{gather}
	From the axiom $ (F2), $ for $ X=Y $ and $ p:=-D_{\tilde{x}}\phi_{\mu_k}\left(y_{\mu_k},x_{\mu_k},\tilde{x}_{\mu_k} \right) , q:=D_x\phi_{\mu_k}\left(y_{\mu_k},x_{\mu_k},\tilde{x}_{\mu_k} \right), $ we obtain that 
	\begin{gather}
		-F\left( y_{\mu_k},x_{\mu_k},u_{j_0}\left( y_{\mu_k},x_{\mu_k}\right),q,X \right) \leq\omega\left(\abs{p-q} \right)-F\left( y_{\mu_k},x_{\mu_k},u_{j_0}\left( y_{\mu_k},x_{\mu_k}\right),p,X \right) \nonumber
	\end{gather}
	i.e.,
	\begin{gather}
		-F\left( y_{\mu_k},x_{\mu_k},u_{j_0}\left( y_{\mu_k},x_{\mu_k}\right),D_x\phi_{\mu_k}\left(y_{\mu_k},x_{\mu_k},\tilde{x}_{\mu_k} \right) ,X \right) \leq\omega\left(4\left(\abs{x_{i_0}-x_{\mu_k}}^3+\abs{x_{i_0}-\tilde{x}_{\mu_k}} ^3\right)  \right) \nonumber\\-F\left( y_{\mu_k},x_{\mu_k},u_{j_0}\left( y_{\mu_k},x_{\mu_k}\right),-D_{\tilde{x}}\phi_{\mu_k}\left(y_{\mu_k},x_{\mu_k},\tilde{x}_{\mu_k} \right),X \right) \label{3_7}.
	\end{gather} 
	From $ (\ref{3_8})  $ and $(\ref{3_7})  $ we obtain that
	\begin{align}
		0&\leq F\left( y_{\mu_k},\tilde{x}_{\mu_k},v_{j_0}\left(y_{\mu_k},\tilde{x}_{\mu_k} \right),-D_{\tilde{x}}\phi_{\mu_k}\left(y_{\mu_k},x_{\mu_k},\tilde{x}_{\mu_k} \right),Y  \right)-\nonumber\\
		&-F\left(y_{\mu_k},x_{\mu_k},u_{j_0}\left(y_{\mu_k},x_{\mu_k} \right),D_x\phi_{\mu_k}\left(y_{\mu_k},x_{\mu_k},\tilde{x}_{\mu_k} \right),X\right)\nonumber\\
		&\leq F\left( y_{\mu_k},\tilde{x}_{\mu_k},v_{j_0}\left(y_{\mu_k},\tilde{x}_{\mu_k} \right),-D_{\tilde{x}}\phi_{\mu_k}\left(y_{\mu_k},x_{\mu_k},\tilde{x}_{\mu_k} \right),Y  \right)+\omega\left(4\left(\abs{x_{i_0}-x_{\mu_k}}^3+\abs{x_{i_0}-\tilde{x}_{\mu_k}} ^3\right)  \right)-\nonumber\\
		&-F\left( y_{\mu_k},x_{\mu_k},u_{j_0}\left( y_{\mu_k},x_{\mu_k}\right),-D_{\tilde{x}}\phi_{\mu_k}\left(y_{\mu_k},x_{\mu_k},\tilde{x}_{\mu_k}\right),X \right)\label{3_9}.
	\end{align}
	From the Claim \ref{claim_1}, it follows that $ \text{for all\ } k\geq k_{\xi},\ 
	$
	\begin{gather}
		0<\delta\leq\delta+\frac{1}{2\mu_k}\abs{x_{\mu_k}-\tilde{x}_{\mu_k}}^2\leq u_{j_0}\left(y_{\mu_k},x_{\mu_k} \right) -v_{j_0}\left(y_{\mu_k},\tilde{x}_{\mu_k} \right)\nonumber.
	\end{gather}
	Setting on the axiom $ (F1),$ for $ s=u_{j_0}\left( y_{\mu_k},x_{\mu_k}\right) , r=v_{j_0}\left(y_{\mu_k},\tilde{x}_{\mu_k} \right)  $, from the previous, we get that $ r<s, $ thus we obtain 
	\begin{align}
		&F\left( y_{\mu_k},\tilde{x}_{\mu_k},v_{j_0}\left(y_{\mu_k},\tilde{x}_{\mu_k} \right),-D_{\tilde{x}}\phi_{\mu_k}\left(y_{\mu_k},x_{\mu_k},\tilde{x}_{\mu_k} \right),Y  \right)\nonumber\\
		&\leq F\left( y_{\mu_k},\tilde{x}_{\mu_k},u_{j_0}\left(y_{\mu_k},x_{\mu_k} \right),-D_{\tilde{x}}\phi_{\mu_k}\left(y_{\mu_k},x_{\mu_k},\tilde{x}_{\mu_k} \right),Y  \right)-\gamma\left(u_{j_0}\left(y_{\mu_k},x_{\mu_k} \right)-v_{j_0}\left(y_{\mu_k},\tilde{x}_{\mu_k} \right) \right) \label{3_10}
	\end{align}
	and from  $ (\ref{3_9}) $ and $ (\ref{3_10}) $ we conclude that
	\begin{align}
		0&\leq F\left( y_{\mu_k},\tilde{x}_{\mu_k},v_{j_0}\left(y_{\mu_k},\tilde{x}_{\mu_k} \right),-D_{\tilde{x}}\phi_{\mu_k}\left(y_{\mu_k},x_{\mu_k},\tilde{x}_{\mu_k} \right),Y  \right)-\nonumber\\
		&-F\left(y_{\mu_k},x_{\mu_k},u_{j_0}\left(y_{\mu_k},x_{\mu_k} \right),D_x\phi_{\mu_k}\left(y_{\mu_k},x_{\mu_k},\tilde{x}_{\mu_k} \right),X\right)\nonumber\\
		&\leq F\left( y_{\mu_k},\tilde{x}_{\mu_k},u_{j_0}\left(y_{\mu_k},x_{\mu_k} \right),-D_{\tilde{x}}\phi_{\mu_k}\left(y_{\mu_k},x_{\mu_k},\tilde{x}_{\mu_k} \right),Y  \right)-\gamma\left(u_{j_0}\left(y_{\mu_k},x_{\mu_k} \right)-v_{j_0}\left(y_{\mu_k},\tilde{x}_{\mu_k} \right) \right)\nonumber\\
		&+\omega\left(4\left(\abs{x_{i_0}-x_{\mu_k}}^3+\abs{x_{i_0}-\tilde{x}_{\mu_k}} ^3\right)  \right)
		-F\left( y_{\mu_k},x_{\mu_k},u_{j_0}\left( y_{\mu_k},x_{\mu_k}\right),-D_{\tilde{x}}\phi_{\mu_k}\left(y_{\mu_k},x_{\mu_k},\tilde{x}_{\mu_k}\right),X \right).
	\end{align}
	Hence,
	\begin{align}
		0&\leq F\left( y_{\mu_k},\tilde{x}_{\mu_k},u_{j_0}\left(y_{\mu_k},x_{\mu_k} \right),-D_{\tilde{x}}\phi_{\mu_k}\left(y_{\mu_k},x_{\mu_k},\tilde{x}_{\mu_k} \right),Y  \right)-\nonumber\\
		&-F\left( y_{\mu_k},x_{\mu_k},u_{j_0}\left( y_{\mu_k},x_{\mu_k}\right),-D_{\tilde{x}}\phi_{\mu_k}\left(y_{\mu_k},x_{\mu_k},\tilde{x}_{\mu_k}\right),X \right)-\nonumber\\
		&-\gamma\left(u_{j_0}\left(y_{\mu_k},x_{\mu_k} \right)-v_{j_0}\left(y_{\mu_k},\tilde{x}_{\mu_k} \right) \right)+\omega\left(4\left(\abs{x_{i_0}-x_{\mu_k}}^3+\abs{x_{i_0}-\tilde{x}_{\mu_k}} ^3\right)  \right).
		\label{3_11}
	\end{align}
	Next, we set $ X:=X_{\mu_k}-\xi I,\ Y:=Y_{\mu_k}+\xi I,\ p:=-D_{\tilde{x}}\phi_{\mu_k}\left(y_{\mu_k},x_{\mu_k},\tilde{x}_{\mu_k} \right). $ From $ (\ref{3_5}) $ of the Claim \ref{claim_8}, using the axiom $ (F3) $, we have  
	\begin{gather}
		F\left( y_{\mu_k},\tilde{x}_{\mu_k},u_{j_0}\left( y_{\mu_k},x_{\mu_k}\right),p,Y_{\mu_k}+\xi I \right)-F\left( y_{\mu_k},x_{\mu_k},u_{j_0}\left( y_{\mu_k},x_{\mu_k}\right),p,X_{\mu_k}-\xi I \right)\nonumber\\
		\leq\omega\left(\frac{1}{\mu_k}\abs{x_{\mu_k}-\tilde{x}_{\mu_k}}^2+\abs{x_{\mu_k}-\tilde{x}_{\mu_k}}\left( \abs{p}+1\right)  \right) \label{3_12}.
	\end{gather}
	In addition, with the double use of axiom $ (F2) $, setting $ p=q $, we obtain: 
	\begin{align}
		F\left( y_{\mu_k},x_{\mu_k},u_{j_0}\left( y_{\mu_k},x_{\mu_k}\right),p,X_{\mu_k}-\xi I \right)-F\left( y_{\mu_k},x_{\mu_k},u_{j_0}\left( y_{\mu_k},x_{\mu_k}\right),p,X_{\mu_k} \right)		&\leq\omega\left(\norm{X_{\mu_k}-\left(X_{\mu_k}-\xi I \right) } \right)\nonumber\\
		&=\omega\left( \xi\right) \label{3_13}
		\end{align}
	\begin{align}
		F\left( y_{\mu_k},\tilde{x}_{\mu_k},u_{j_0}\left( y_{\mu_k},x_{\mu_k}\right),p,Y_{\mu_k} \right)-F\left( y_{\mu_k},y_{\mu_k},u_{j_0}\left( y_{\mu_k},x_{\mu_k}\right),p,Y_{\mu_k}+\xi I \right)
		&\leq\omega\left(\norm{Y_{\mu_k}-\left(Y_{\mu_k}+\xi I \right) } \right)\nonumber\\
		&=\omega\left( \xi\right) \label{3_14}
	\end{align}
	and by summing  $ (\ref{3_13}) $ and $ (\ref{3_14}) $, we obtain 
	\begin{gather}
		F\left( y_{\mu_k},x_{\mu_k},u_{j_0}\left( y_{\mu_k},x_{\mu_k}\right),p,X_{\mu_k}-\xi I \right)-F\left( y_{\mu_k},x_{\mu_k},u_{j_0}\left( y_{\mu_k},x_{\mu_k}\right),p,X_{\mu_k} \right)\nonumber\\
		+F\left( y_{\mu_k},\tilde{x}_{\mu_k},u_{j_0}\left( y_{\mu_k},x_{\mu_k}\right),p,Y_{\mu_k} \right)-F\left( y_{\mu_k},y_{\mu_k},u_{j_0}\left( y_{\mu_k},x_{\mu_k}\right),p,Y_{\mu_k}+\xi I \right)
		\leq 2\omega\left( \xi\right) \nonumber\\
		\iff  F\left( y_{\mu_k},y_{\mu_k},u_{j_0}\left( y_{\mu_k},x_{\mu_k}\right),p,Y_{\mu_k} \right)-F\left( y_{\mu_k},x_{\mu_k},u_{j_0}\left( y_{\mu_k},x_{\mu_k}\right),p,X_{\mu_k} \right)\nonumber\\
		\leq F\left( y_{\mu_k},\tilde{x}_{\mu_k},u_{j_0}\left( y_{\mu_k},x_{\mu_k}\right),p,Y_{\mu_k}+\xi I \right)-F\left( y_{\mu_k},x_{\mu_k},u_{j_0}\left( y_{\mu_k},x_{\mu_k}\right),p,X_{\mu_k}-\xi I \right)+2\omega\left( \xi\right)
		\label{3_15}.
	\end{gather}
	By combining $ (\ref{3_12}) $ and $ (\ref{3_15}) $ we receive that
	\begin{gather}
		F\left( y_{\mu_k},\tilde{x}_{\mu_k},u_{j_0}\left( y_{\mu_k},x_{\mu_k}\right),p,Y_{\mu_k} \right)-F\left( y_{\mu_k},x_{\mu_k},u_{j_0}\left( y_{\mu_k},x_{\mu_k}\right),p,X_{\mu_k} \right)\leq 2\omega\left( \xi\right) +
		\nonumber\\ +\omega\left(\frac{1}{\mu_k}\abs{x_{\mu_k}-\tilde{x}_{\mu_k}}^2+\abs{x_{\mu_k}-\tilde{x}_{\mu_k}}\left( \abs{p}+1\right)  \right)\label{3_16}
	\end{gather}
	and finally, through a combination of $ (\ref{3_11}) $ and $ (\ref{3_16}) $, we obtain that
	\begin{gather}
		0\leq \omega\left(4\left(\abs{x_{i_0}-x_{\mu_k}}^3+\abs{x_{i_0}-\tilde{x}_{\mu_k}} ^3\right)  \right)
		+2\omega\left(\xi \right)-\gamma\left(u_{j_0}\left(y_{\mu_k},x_{\mu_k} \right)-v_{j_0}\left(y_{\mu_k},\tilde{x}_{\mu_k} \right) \right)+\nonumber\\ +\omega\left(\frac{1}{\mu_k}\abs{x_{\mu_k}-\tilde{x}_{\mu_k}}^2+\abs{x_{\mu_k}-\tilde{x}_{\mu_k}}\left( \abs{-D_{\tilde{x}}\phi_{\mu_k}\left(y_{\mu_k},x_{\mu_k},\tilde{x}_{\mu_k} \right)}+1\right)  \right).
	\end{gather}
 Next, it is clear that
	\begin{gather}
		0\leq\limsup_{k\to\infty}\left(-\gamma\left(u_{j_0}\left(y_{\mu_k},x_{\mu_k} \right)-v_{j_0}\left(y_{\mu_k},\tilde{x}_{\mu_k} \right) \right)\right)  + \limsup_{k\to\infty}\left( \omega\left(4\left(\abs{x_{i_0}-x_{\mu_k}}^3+\abs{x_{i_0}-\tilde{x}_{\mu_k}} ^3\right)  \right)\right) +\nonumber\\
		+2\omega\left(\xi \right)
		+\limsup_{k\to\infty}\left( \omega\left(\frac{1}{\mu_k}\abs{x_{\mu_k}-\tilde{x}_{\mu_k}}^2+\abs{x_{\mu_k}-\tilde{x}_{\mu_k}}\left( \abs{-D_{\tilde{x}}\phi_{\mu_k}\left(y_{\mu_k},x_{\mu_k},\tilde{x}_{\mu_k} \right)}+1\right)  \right)\right) \label{3_17}
	\end{gather}
	and in specific we have the following estimates for each quantity that contains $ \limsup: $
	\begin{align}
		\limsup_{k\to\infty}\left(-\gamma\left(u_{j_0}\left(y_{\mu_k},x_{\mu_k} \right)-v_{j_0}\left(y_{\mu_k},\tilde{x}_{\mu_k} \right) \right)\right)
		&\leq-\gamma\ \limsup_{k\to\infty}\left(u_{j_0}\left(y_{\mu_k},x_{\mu_k} \right)-v_{j_0}\left(y_{\mu_k},\tilde{x}_{\mu_k} \right) \right) \nonumber\\
		&\leq -\gamma \left( \limsup_{k\to\infty}\left(u_{j_0}\left(y_{\mu_k},x_{\mu_k} \right)\right)+\limsup_{k\to\infty}\left(-v_{j_0}\left(y_{\mu_k},\tilde{x}_{\mu_k} \right)\right)  \right) \nonumber\\
		&=-\gamma \left( \limsup_{k\to\infty}\left(u_{j_0}\left(y_{\mu_k},x_{\mu_k} \right)\right)-\liminf_{k\to\infty}\left(v_{j_0}\left(y_{\mu_k},\tilde{x}_{\mu_k} \right)\right)  \right) \nonumber\\
		&=-\gamma \left( u_{j_0}\left(y_{i_0},x_{i_0} \right)-v_{j_0}\left(y_{i_0},x_{i_0}\right)\right)  \nonumber\\
		&=-\gamma\ \delta.
	\end{align}
	It is true that for the sequence of real numbers $ b_k:=4\left(\abs{x_{i_0}-x_{\mu_k}}^3+\abs{x_{i_0}-\tilde{x}_{\mu_k}} ^3\right), $ that $ b_k\xrightarrow{k\to\infty}0 $. Then,  from the sequential criterion of continuity of function $ \omega $, and in addition with the limit relation of a function and continuity, we have 
	$$\lim_{k\to\infty}\omega(b_k)=\omega(0)=\lim_{x\to 0^+}\omega(x)\equiv\omega(0^+)=0.$$
	Thus, we obtain $ \limsup_{k\to\infty}\omega(b_k)=0 $. In a similar way, it is proved that $ \limsup_{k\to\infty}\omega(m_k)=0 $ where $ m_k:=\frac{1}{\mu_k}\abs{x_{\mu_k}-\tilde{x}_{\mu_k}}^2+\abs{x_{\mu_k}-y\tilde{x}_{\mu_k}}\left( \abs{-D_{\tilde{x}}\phi_{\mu_k}\left(y_{\mu_k},x_{\mu_k},\tilde{x}_{\mu_k} \right)}+1\right). $
	From the above, and  $ (\ref{3_17}), $ we arrive at a conclusion where: 
	$$\gamma \delta\leq 2\omega(\xi).$$
	The above inequality was proved for a random $ \xi>0 $.  Then we have 
	$ \gamma \delta=\lim_{\xi\to 0^+} \gamma\ \delta\leq\lim_{\xi\to 0^+} 2\ \omega(\xi)=2\ \omega(0^+)=0.$
	In consequence, $ \gamma\ \delta\leq 0, $ which is a contratiction, since $ \gamma, \delta>0. $
	
	The proof of the above result, was established in the case that $ y_{i_0}<L. $ In the case when $ y_{i_0}= L,$ (see step 1), we can process in a different way, by using an appropriate modification of $ u $.
	Specifically, while having in mind that the functions $ u $ and $ v $ are viscosity subsolutions and supersolutions respectively for the problem $ \left(BVP\right)  $, we define the following modified function for a random parameter $ \theta>0. $
	\begin{gather}
		u^{\theta}:[0,L)\times\bar{\Omega}\to\R^m\nonumber\\
		u^{\theta}(y,x):=u(y,x)+h^{\theta}(y,x)\nonumber\\
	\end{gather}
	where  $ h^{\theta}(y,x) $ is an appropriate function such that the following  identities hold: $ \text{for all\ } i\in\left\lbrace 1,2,\dots,m\right\rbrace   $
	\begin{itemize}
		\item [($\alpha$)] $\text{for all\ }(y,x)\in\left[ 0,L\right)\times\bar{\Omega},\ \lim_{\theta\to 0^+}u^{\theta}_i(y,x)=u_i(y,x)$
		\item[($\beta$)]$ \lim_{y\to L^-}u_i^{\theta}(y,x)=-\infty$ uniformly on $ \bar{\Omega} $\\
		with the following interpretation: We set $ (g^i_y)_{y\in\left[ 0,L\right) }$ to be a family of functions\\ $ g^i_y(x):=u_i^{\theta}(y,x),\ x\in\bar{\Omega} $ for which we demand to converge uniformly on $ \bar{\Omega}, $ i.e.,  
		$$\text{for all\ } M>0,\text{there exist}\ \delta(M)>0,\ \text{for all\ } y\in(L-\delta,L),\ \text{for all\ } x\in\bar{\Omega}: u^{\theta}_i(y,x)<-M.$$
	\end{itemize}
	 Next, we will see a Proposition that provides us a sufficent condition, such that the limit of a sum of two functions is $ -\infty $, uniformly on $ \bar{\Omega} $  while $ y\to L^-. $ Furthermore, it is proved that, if an upper semicontinuous function  satisfies the above limiting behaviour, then it is necessary that it will attain a maximum point in $ \left[0,L \right)\times\bar{\Omega}.  $
	\begin{prop}\label{prop_4_1}
			Let $ \kappa,\rho:\left[ 0,L\right)\times\bar{\Omega}\to\R   $ be mappings, such that $ \kappa $ is upper bounded and \\ $ \lim_{y\to L^-}\rho(y,x)=-\infty $ uniformly on $\bar{\Omega}. $ Then, $ \lim_{y\to L^-}\left(\kappa(y,x)+\rho(y,x) \right) =-\infty $ uniformly on $ \bar{\Omega}. $
	\end{prop}
	\begin{proof}[Proof]
		 Let $ M>0 $ be random. Due to the fact that the function  $ \kappa $  is upper bounded, there will be a positive constant $ C>0,$ such that, $ \text{for all\ } (y,x)\in \left[0,L \right)\times\bar{\Omega},\ \kappa(y,x)\leq C. $ Furthermore, because $ \lim_{y\to L^-}\left(\kappa(y,x)+\rho(y,x) \right) =-\infty $, in specific for $ \tilde{M}=M+C>0 $, it will exists $ \delta(\tilde{M}) >0,\\ \text{for all\ } y\in (L-\delta,L),\ \text{for all\ } x\in \bar{\Omega},\rho(y,x)<-\tilde{M}=-M-C$. Then $ \text{for all\ } y\in(L-\delta,L),\ \text{for all\ } x\in\bar{\Omega} $
		$$\kappa(y,x)+\rho(y,x)\leq M+\rho(y,x)<C+(-M-C)=-M$$
		In consequence, we have the desirable result.
	\end{proof}

	\begin{prop}\label{prop_4_2}
			Let $ \kappa:\left[ 0,L\right)\times\bar{\Omega}\to\R   $ be upper semicontinuous with the identity $ \lim_{y\to L^-}\kappa(y,x)=-\infty $ uniformly on $\bar{\Omega}. $ Then exists $  (\tilde{y},\tilde{x}) \in \left[0,L \right)\times\bar{\Omega} $ such that $ \kappa(\tilde{y},\tilde{x})=\max_{\left[0,L \right)\times\bar{\Omega} }\kappa(y,x) $.
	\end{prop}
	
	\begin{proof}[ Proof]
		 Let $ \lim_{y\to L^-}\kappa(y,x)=-\infty, $ uniformly on $ \bar{\Omega}. $ We consider $ y^*\in\left[0,L \right)  $ and $ x^*\in\bar{\Omega} $ random. We define  $ M:=\abs{\kappa(y^*,x^*)}+1>0. $ Because $ \lim_{y\to L^-}\kappa(y,x)=-\infty, $ uniformly on $ \bar{\Omega} $, it holds that
		\begin{gather}\label{prop_4_2_eq_1}
			\text{there exist}\tilde{\delta}(M_{x^*})>0,\ \text{for all\ } y\in(L-\tilde{\delta},L),\ \text{for all\ } x\in\bar{\Omega},\ u(y,x)<-M_{x^*}=-\abs{\kappa(y^*,x^*)}-1
		\end{gather}
		We select $ \delta_0\in(0,\tilde{\delta}), $ such that $ y^*<L-\delta_0. $ Then the $ \kappa $ as an upper semicontinuous on the compact $ \left[0,L-\delta_0 \right] \times\bar{\Omega} $ will attain maximum value. Thus,
		\begin{gather}\label{prop_4_2_eq_2}
			\text{there exist}(\tilde{y},\tilde{x})\in\left[ 0,L-\delta_0\right]\times\bar{\Omega},\ \text{for all\ }(y,x)\in\left[ 0,L-\delta_0\right]\times\bar{\Omega},\ \kappa(\tilde{y},\tilde{x})\geq\kappa(y,x).
		\end{gather} 
		Since $ y^*<L-\delta_0 $ and $ x^*\in\bar{\Omega} $, will apply that $ \text{for all\ } y\in(L-\delta_0,L)\subset(L-\tilde{\delta},L),\ \text{for all\ } x\in\bar{\Omega}, $
		\begin{gather}
			\kappa(\tilde{y},\tilde{x})\geq^{(\ref{prop_4_2_eq_2})}\kappa(y^*,x^*)\geq-\abs{\kappa(y^*,x^*)}>-\abs{\kappa(y^*,x^*)}-1>^{(\ref{prop_4_2_eq_1})}\kappa(y,x)\nonumber.
		\end{gather}
		Therefore, it is shown that $ \text{for all\ } y\in\left[0,L \right)\times\bar{\Omega},\ \kappa(\tilde{y},\tilde{x}) \geq\kappa(y,x). $
	\end{proof}
	
	An example of a function $ u^{\theta}:\left[0,L \right)\times\bar{\Omega}\to\R^m  $ that satisfies the assumptions $ (\alpha) $ and $ (\beta) $, is
	\begin{gather}
		u^{\theta}_i(y,x):=u_i(y,x)-\frac{\theta}{L-y},\ \text{for all\ } (y,x)\in \left[ 0,L\right) \times\bar{\Omega}
	\end{gather}
	specifically,  $ h^{\theta}(y,x):=-\frac{\theta}{L-y},\ \text{for all\ } (y,x)\in \left[ 0,L\right) \times\bar{\Omega}. $ We observe that $ h $ as continuous, will be simultaneously upper and lower semicontinuous and furthermore, as $ y $ tends to $ L^- $, its limit will be $ -\infty $, uniformly on $ \bar{\Omega}$. Furthermore, we have that $ u_i^{\theta}(y,x)-v_i(y,x)=u_i(y,x)-v_i(y,x)+h^{\theta}(y,x). $ The first statement is that  $ u_i-v_i $ is upper bounded on  $ \left[ 0,L\right)\times\bar{\Omega}. $ Indeed,
	\begin{align}
		\text{for all\ } i\in\left\lbrace 1,2,,\dots,m \right\rbrace,\ \sup_{[0,L)\times\bar{\Omega}}(u_i-v_i)(y,x)&\leq \sup_{[0,L]\times\bar{\Omega}}(u_i-v_i)(y,x)\nonumber\\
		&\leq \max_{j\in\{1,2,\dots,m \}}\sup_{[0,L]\times\bar{\Omega}}(u_i-v_i)(y,x)\nonumber\\
		&=\sup_{[0,L]\times\bar{\Omega}}(u_{i_0}-v_{i_0})(y,x)=u_{i_0}(L,x_{i_0})-v_{i_0}(L,x_{i_0})=\delta
	\end{align}
	 then from  Proposition \ref{prop_4_1}, we obtain that the limit of $ u_i^{\theta}-v_i$ as $ y\to L^- $ is $ -\infty $ uniformly on $ \bar{\Omega}. $ Having in mind that $ u_i^{\theta}-v_i $ is upper semicontinuous as a sum of upper semicontinuous functions, we obtain from Proposition \ref{prop_4_2} that there will be a maximum point of $ u_i^{\theta}-v_i $ inside of $ \left[0,L \right)\times\bar{\Omega}.  $\\
	\\
	 Next, it is proved, through the claims below, that $ u^{\theta}:\left[0,L \right)\times\bar{\Omega}\to\R^m $ is a viscosity subsolution of an appropriate modified system. 
	
	\newpage
	\begin{claim}\label{Isxyrismos_epektasi}
		$\text{For all\ } (y,x)\in(0,L)\times\Omega,\ \text{for all\ } p\in\R^n,\ \text{for all\ } X\in\mathcal{S}^n,\  \text{for all\ } i\in\left\lbrace 1,2,\dots,m\right\rbrace $ holds that:
		\begin{itemize}
			\item[( i)] $ u_i(y,x)>\mathcal{ M}_i u(y,x)\iff u_i^{\theta}(y,x)>\mathcal{ M}u^{\theta}(y,x) $
			\item[( ii)] $F\left(y,x,u_i^{\theta}(y,x),p,X \right)\leq F\left(y,x,u_i(y,x),p,X \right)$
		\end{itemize}
	\end{claim}
	\begin{proof}[ Proof]
	
		\begin{itemize}
			\item[(i)] Let $ (y,x)\in(0,L)\times\Omega $ be a random point. We observe that $ \text{for all\ } i\in \left\lbrace 1,2,\dots,m \right\rbrace  $
			\begin{align}
				u_i(y,x)-\mathcal{ M}_i u(y,x)&=u_i(y,x)-\frac{\theta}{L-y}+\frac{\theta}{L-y}-\max_{j\ne i}\left(u_j(y,x)-c_{i,j}(y,x) \right) \nonumber\\
				&=\left( u_i(y,x)-\frac{\theta}{L-y}\right) -\max_{j\ne i}\left[\left( u_j(y,x)-\frac{\theta}{L-y}\right)-c_{i,j}(y,x) \right] \nonumber\\
				&=u_i^{\theta}(y,x)-\max_{j\ne i}\left( u_j^{\theta}(y,x)-c_{i,j}(y,x) \right)\nonumber\\
				&=u_i^{\theta}(y,x)-\mathcal{M}_i u^{\theta}(y,x)
			\end{align}
			and the desired result is immediate.
			\item[(ii)] At first, we see that $ \text{for all\ } (y,x)\in(0,L)\times\Omega,\ u_i^{\theta}(y,x)<u_i(y,x). $ We set $ r:=u_i^{\theta}(y,x) $ and \\$ s:=u_i(y,x).  $ From axiom $ (F_1), $ we obtain that:
			\begin{gather}
				\gamma(s-r)\leq F(y,x,s,p,X)-F(y,x,r,p,X)\nonumber\\
				\text { i.e, }\ \gamma(u_i(y,x)-u_i^{\theta}(y,x))\leq F(y,x,(u_i(t,x),p,X)-F(y,x,u_i^{\theta}(y,x),p,X)\nonumber\\
				\begin{align} \text{ equivalently, }\ F(y,x,u_i^{\theta}(y,x),p,X)&\leq F(y,x,(u_i(y,x),p,X)- \gamma(u_i(y,x)-u_i^{\theta}(y,x))\nonumber\\
					&=F(y,x,(u_i(y,x),p,X)-\frac{\gamma\theta}{L-y}\nonumber\\
					&<F(y,x,(u_i(y,x),p,X)
				\end{align}
			\end{gather}
			where the last inequality holds, due to the $ \frac{\gamma \theta}{L-y}>0. $
		\end{itemize}
	\end{proof}
	
	\begin{claim}
		 We consider the following modified system which depends on the positive parameter $ \theta>0 $ 
		$$\begin{cases}
			\min\biggl\{ F\bigl( y,x,q_{i}(y,x),D q_{i}(y,x),D^2 q_{i}(y,x)\bigl),  q_{i}(y,x)-\max_{j\neq i}\bigl( q_{j}(y,x)-c_{ij}(y,x)\bigl)\biggl\}=0, \left(y,x \right)\in\Omega_{L}\\ 
			q_{i}(0,x)=g_{i}(x)-\frac{\theta}{L},\ x\in\bar{\Omega}\ \
			q_{i}(y,x)=f_i(y,x)-\frac{\theta}{L-y},  (y,x)\in(0,L)\times\partial{\Omega}
		\end{cases}$$
		
		then, $ u^{\theta}:\left[0,L \right)\times\bar{\Omega}\to\R^m  $ and $  v:\left[0,L \right)\times\bar{\Omega}\to\R^m,  $ are viscosity subsolutions  and supersolutions respectively of the above problem.
	\end{claim}
	\begin{proof}[Proof]
		 At first, we will prove that $ u^{\theta}:\left[0,L \right)\times\bar{\Omega}\to\R^m  $ is a viscosity subsolution of the above system. Indeed, let  $ i\in \left\lbrace 1,2,\dots,m \right\rbrace,\  (y,x)\in(0,L) $ be random and $ (\alpha,p,X)\in \bar{J}^{2,+}u_i(y,x). $ Then equivalently, we have $ \left(\alpha+\frac{\theta}{\left( L-y\right)^2} ,p,X \right)\in \bar{J}^{2,+}u^{\theta}_i(y,x).  $ Due to the fact that function  $ u $ is a viscosity subsolution of $ (BVP) $, then the following holds: 
		$$\min\left\lbrace F(y,x,u_i(y,x),p,X),\ u_i(y,x)-\mathcal{M}_i u(y,x)\right\rbrace\leq 0.$$
		From the above, we extract the following cases:
		\begin{itemize}
			\item[(i)] Let $ \min\left\lbrace F(y,x,u_i(y,x),p,X),\ u_i(y,x)-\mathcal{M}_i u(y,x)\right\rbrace=u_i(y,x)-\mathcal{M}_i u(y,x). $\\ Then, $u_i(y,x)-\mathcal{M}_i u(y,x)\leq 0.  $ 
			From Claim \ref{Isxyrismos_epektasi} $(i)$, it is necessary that\\ $  u_i^{\theta}(y,x)-\mathcal{ M}u^{\theta}(y,x)\leq 0. $ Then, 
			$$\min\left\lbrace F(y,x,u_i^{\theta}(y,x),p,X),\ u_i^{\theta}(y,x)-\mathcal{M}_i u^{\theta}(y,x)\right\rbrace\leq u_i^{\theta}(y,x)-\mathcal{M}_i u^{\theta}(y,x)\leq 0.$$
			\item[(ii)]Let $ \min\left\lbrace F(y,x,u_i(y,x),p,X),\ u_i(y,x)-\mathcal{M}_i u(y,x)\right\rbrace=F(y,x,u_i(y,x),p,X). $\\ Then, $ F(y,x,u_i(y,x),p,X)\leq 0. $
			From  Claim \ref{Isxyrismos_epektasi} $(ii)$, we obtain that
			$$F(y,x,u_i^{\theta}(y,x),p,X) \leq F(y,x,u_i(y,x),p,X).$$
			Then
			$$\min\left\lbrace F(y,x,u_i^{\theta}(y,x),p,X),\ u_i^{\theta}(y,x)-\mathcal{M}_i u^{\theta}(y,x)\right\rbrace\leq F(y,x,u_i^{\theta}(y,x),p,X)\leq 0.$$
			Regarding the initial and boundary conditions, we obtain that
			\begin{gather}
				\text{for all\ } x\in\bar{\Omega},\ u^{\theta}_i(0,x)=u_i(0,x)-\frac{\theta}{L}\leq^{(u\ subsolution )} g_i(x)-\frac{\theta}{L}\nonumber\\
				\text{for all\ } y\in(0,L),\ \text{for all\ } x\in\partial{\Omega},\ u^{\theta}_i(y,x)=u_i(y,x)-\frac{\theta}{L-y}\leq^{(u\ subsolution )}f_i(y,x)-\frac{\theta}{L-y}\nonumber.
			\end{gather}
		\end{itemize}	
			Therefore, function $ u^{\theta} $ is a viscosity subsolution of the above problem. As for  $ v, $ it is clear that it is a viscosity supersolution of the above problem.
		\end{proof}
		 From the above remarks, we follow the same steps as we did for functions  $ u $ and $ v $, with a difference, that in place of function $ u $, we place  function $ u^{\theta} $. In this way, we receive the result $ u_i^{\theta}(y,x)\leq v_i(y,x), \text{for all\ } i\in \left\lbrace 1,2,\dots,m \right\rbrace, \text{for all\ } \left(y,x \right)\in \left[0,L \right) \times\bar{\Omega}.   $ Then, we fix a random point,  $ (y,x)\in \left[0,L \right)\times\bar{\Omega},  $ and we leave parameter  $ \theta $ to tend to 0. From property $ (\alpha) $, we obtain the desirable result.

	\end{proof}
	
	Using now the Comparison Principle, we succeed in proving the desired uniqueness property of viscosity solutions.
\begin{theorem}[Uniqueness of the Solution of the Problem  $\left(BVP \right) $]
	 If there is a viscosity solution of the problem $\left(BVP\right) $, then it is unique.
\end{theorem}
\begin{proof}[Proof]
	 We consider  $ h:=\left( h_{1},h_{2},\dots,h_{m} \right):[0,L]\times\bar{\Omega}\to\R $ and $ w:=\left( w_{1},w_{2}\dots,w_{m}  \right):[0,L]\times\bar{\Omega}\to\R$  two random viscosity solutions of problem $ (BVP). $ We will prove that $ h=w $ on $ [0,L)\times\bar{\Omega}. $ Indeed, from the fact that  $ h ,w $ are  viscosity solutions of problem $ (BVP), $ we can construct\footnote{Let $ V:[0,L]\times\bar{\Omega}\to\R^m $ with $ V=( V_1,V_2,\dots, V_m) $ be a bounded function\\ (i.e $\text{for all\ } i\in\{1,2,\dots,m\},\  \sup\{V_i(y,x)\ :\  (y,x)\in[0,L]\times\bar{\Omega}\}<+\infty $ and $ \inf\{V_i(y,x)\ :\  (y,x)\in[0,L]\times\bar{\Omega}\}>-\infty  $) and continuous on $ [0,L)\times\bar{\Omega}. $ Then, for each $ i\in\{1,2,\dots,m\},  $ the following functions 
			$$\hat{V}^1_i(y,x):=\begin{cases}
				V_i(y,x) & \ \ (y,x)\in[0,L)\times\bar{\Omega}\cr
				M & y=L,\ x\in\bar{\Omega} 
			\end{cases}
		\ \text{and} \ \hat{V}^2_i(y,x):=\begin{cases}
			V_i(y,x) &\quad (y,x)\in[0,L)\times\bar{\Omega}\\
			m & y=L,\ x\in\bar{\Omega} \\
		\end{cases}$$

where $ M:=\max_{i\in\{1,2,\dots,m\}}\sup\{V_i(y,x)\ :\ (y,x)\in [0,L]\times\bar{\Omega}\}$ and $ m:=\min_{i\in\{1,2,\dots,m\}}\inf\{V_i(y,x)\ :\ (y,x)\in [0,L]\times\bar{\Omega}\} $, are upper semicontinuous and lower semicontinuous on $[0,L]\times\bar{\Omega}$ respectively.
} appropriate modifications  $ \hat{h}^j:[0,L]\times\bar{\Omega}\to\R^m, j=1,2 $ and $ \hat{w}^j:[0,L]\times\bar{\Omega}\to\R^m, j=1,2 $ of $h$  and $ w $ respectively, such that 
	\begin{gather}\label{visco_modi}
		\hat{h}^j=h,\ \text{and}\
		\hat{w}^j=w,\   \text{on}\  [0,L)\times\bar{\Omega},\ j=1,2
	\end{gather}
	satisfying also that $ \hat{h}^{1},\ \hat{w}^{1} $ are upper semicontinuous functions on $ [0,L]\times\bar{\Omega} $ and $ \hat{h}^{2},\ \hat{w}^{2}  $ are lower semicontinuous on $ [0,L]\times\bar{\Omega} $. From the definitions of these modifications, it is extracted that $ \hat{h}^1,\ \hat{w}^1 $ are viscosity subsolutions of $ (BVP) $ and $ \hat{h}^2,\ \hat{w}^2 $ are viscosity supersolutions of $ (BVP) $. Using now, the Comparison Principle for the following couples of viscosity subsolutions and viscosity supersolutions of $ (BVP), $ in specific, $ \left(\hat{h}^1, \hat{w}^2\right)  $ and $ \left( \hat{w}^1,\hat{h}^2\right)  $, we deduce that
	\begin{gather}
		\hat{h}^1\leq\hat{w}^2,\ \text{and}\
		\hat{w}^1\leq\hat{h}^2,\ \text{on}\ [0,L)\times\bar{\Omega}	\nonumber
	\end{gather}
	  and combining with $ (\ref{visco_modi}) $, we obtain that $ h=w $ on $ [0,L)\times\bar{\Omega}. $
\end{proof}


\subsection*{Existence of Viscosity Solution }
For the rest of the section, we assume that the assumptions $ (F1)-(F3) $ and $ (O_1)-(O_5) $ hold. Furthermore, we assume the following axiom for the operator $ F $ and an axiom that controls further the relation between the obstacles functions $ c_{i,j} $, the initial data $ g_i $ and the boundary data $ f_{i} $ as well:
\begin{gather}
	(F4)\ \text{for all\ } (y,x,r,p,X)\in [0,L]\times\R^n\times\R\times\R^n\times\mathcal{S}^n,\ F(y,x,r,p,X)\geq r\nonumber\\
	\end{gather}
\begin{align}
	\left( O_6 \right)\ \min_{i,j\in\left\lbrace 1,2,\dots,m \right\rbrace }\{ c_{i,j}(y,x) \mid y\in[0,L],\ x\in\partial\Omega\}& \geq\max_{i\in\left\lbrace 1,2,\dots,m \right\rbrace } \{f_i(y,x) \mid y\in[0,L],\ x\in\partial\Omega\}-\nonumber\\
	&-\min_{i\in\left\lbrace 1,2,\dots,m \right\rbrace} \{g_i(x) \mid x\in\partial\Omega\}\nonumber
\end{align}
	\begin{gather}
			\left( O_7 \right)\ \text{for all\ } y\in(0,L),\ \text{for all\ } x,\tilde{x}\in\Omega,\ g_i(x)+c_{i,j}(y,\tilde{x})\geq 0.
	\end{gather}
\newpage
\textbf{Remark}: Assumption (F4) is needed in the proof of Claim \ref{c_2} (see equation (\ref{c_2_F_4}) therein). It seems that the necessity of this condition has to do with the fact that both quantities in the minimisation in Claim 3.12 might contain zero order terms.

\begin{theorem}[Existence of Viscosity Solution for the Problem $\left(BVP \right)$]\label{Existence Theorem}
	 The problem   $\left(BVP\right) $  has at least one  viscosity solution.
\end{theorem}
In order to prove the above theorem, it is sufficent to prove the following propositions:
\begin{prop}\label{Perronmethod}
	Assume that for each $ i\in\left\lbrace 1,2,\dots,m\right\rbrace  $ and $ \hat{x}\in\bar{\Omega} $ there exists a family of continuous viscosity sub- and supersolutions, $ \left\lbrace u^{i,\hat{x},\epsilon}\right\rbrace_{\epsilon>0}  $ and $ \left\lbrace v^{i,\hat{x},\epsilon}\right\rbrace_{\epsilon>0}  $ respectively to (BVP) such that
	$$\sup_{\epsilon>0}u_i^{\hat{x},\epsilon}(0,\hat{x})=g_i(\hat{x})=\inf_{\epsilon>0}v_i^{i,\hat{x},\epsilon}(0,\hat{x})$$
	then there exists a a viscosity solution of (BVP).
\end{prop}
With this result given, what remains is to construct appropriate barriers, i.e families of viscosity sub- and supersolutions. More specifically, we prove the following:
\begin{prop}\label{existence of barriers}
	Let $ A\in\R $, $ i\in\left\lbrace 1,2,\dots,m \right\rbrace,\ \hat{x}\in\bar{\Omega}  $  and $\phi$ be a given, nonnegative, continuous function $ \phi:\bar{\Omega}\rightarrow\R, $ on $ \bar{\Omega}, $ with the identity: $ \text{there exists}\ \delta_{\hat{x}}>0,\ \text{such that, for all\ } x\in\Omega_{\delta_{\hat{x}}}:= B_{\rho_2}\left( \hat{x},\delta_{\hat{x}}\right)\cap\Omega, \\ \phi(x)\geq\phi(\hat{x}) $. Then, there exists $ \kappa>0 $ and appropriate numbers $ B\in\R^+, C\in\R $ such that,  the functions $ U^{\hat{x},\epsilon}:=\left( U_1^{\hat{x},\epsilon},U_2^{\hat{x},\epsilon},\dots,U_m^{\hat{x},\epsilon} \right)  $ and  $\  V^{i,\hat{x},\epsilon}:=\left( V_1^{i,\hat{x},\epsilon},V_2^{i,\hat{x},\epsilon},\dots,V_m^{i,\hat{x},\epsilon} \right)  $, with
	\begin{gather}
		U_j^{\hat{x},\epsilon}(y,x):=g_j(\hat{x})-A(\phi(x)-\phi(\hat{x}))-B\exp(\kappa \phi(x))\abs{x-\hat{x}}^2-\epsilon-Cy\nonumber\\
		V_j^{i,\hat{x},\epsilon}(y,x):=g_i(\hat{x})+A(\phi(x)-\phi(\hat{x}))+B\exp(\kappa \phi(x))\abs{x-\hat{x}}^2+\epsilon+Cy+ c_{i,j}(y,x)\nonumber
	\end{gather}
	are viscosity sub- and supersolutions of the (BVP) respectively. Moreover,
	$$\sup_{\epsilon>0}U_i^{\hat{x},\epsilon}(0,\hat{x})=g_i(\hat{x})=\inf_{\epsilon>0}V_i^{i,\hat{x},\epsilon}(0,\hat{x}).$$
\end{prop}
\textbf{Remark}: The barrier functions mentioned above, are structurally identical to those mentioned in the paper \cite{LNOO4}. But there is an essential difference between our paper and \cite{LNOO4} regarding the property we require the function $ \phi $ to satisfy. Also in \cite{LNOO4} the entire proof process aims to show that the specific functions satisfy the definition of sub-supersolutions of (BVP), where in this definition of solutions there is a form of Neumann condition, in contrast to our paper which does not exist this condition. As a consequence of the latter, it follows that the techniques developed in our paper for determining the constants included in the definition of the barrier functions, are also different.\\

Combining Propositions \ref{Perronmethod} and \ref{existence of barriers} above, we extract the existence part of Theorem \ref{Existence Theorem}. We start with the proof of Proposition $ \ref{existence of barriers} $
\begin{proof}(\textbf{of Proposition $ \ref{existence of barriers} $})
	At first, we consider a random $ i\in\left\lbrace 1,2,\cdots,m \right\rbrace  $ and we prove that $ V^{i,\hat{x},\epsilon} $ is viscosity supersolution of (BVP). Let $ j\in\left\lbrace 1,2,\dots,m\right\rbrace  $ be a random index. We observe that the function $ V_j^{i,\hat{x},\epsilon}:[0,L]\times\bar{\Omega}\to\R  $ as continuous, will be lower semicontinuous on $ [0,L]\times\bar{\Omega}. $ The aim is to prove that function  $ V_j^{i,\hat{x},\epsilon} $ satisfies the definition of viscosity supersolution.\\

	The proof consists of the following claims:
	\begin{claim}\label{c_1}
		For a given $ A\in\R $ and any $ \kappa>0, $ there exists an appropriate number $ \tilde{B}\in\R^{+} $  such that\\ $ \text{for all\ } j\in\left\lbrace 1,2,\dots,m\right\rbrace  $,
		\begin{gather}
			V_j^{i,\hat{x},\epsilon}(0,x)> g_j(x),\ \text{for all\ } x\in\bar{\Omega}\nonumber\\
			g_i(\hat{x})=\inf_{\epsilon>0}V_i^{i,\hat{x},\epsilon}(0,\hat{x})\nonumber.
		\end{gather}
	\end{claim}
	\begin{proof}
		First, we observe that $ V_j^{i,\hat{x},\epsilon}(0,\hat{x})\ge g_j(\hat{x}). $ Indeed, $ \text{for all\ }\epsilon>0 $ we have that
		\begin{gather}
			V_j^{i,\hat{x},\epsilon}(0,\hat{x})=g_i(\hat{x})+c_{i,j}(0,\hat{x})+\epsilon\nonumber
		\end{gather}
		but by the use of axiom $ (O_5) $ we get that $ g_i(\hat{x})+c_{i,j}(0,\hat{x})\ge g_j(\hat{x}),\ \text{for all\ } i,j\in\left\lbrace 1,2,\dots,m \right\rbrace. $ Consequently,
		\begin{gather}
			V_j^{i,\hat{x},\epsilon}(0,\hat{x})\geq g_j(\hat{x})+\epsilon>g_j(\hat{x}),\ \text{for all\ }\epsilon>0\nonumber.
		\end{gather}
		As a result from the above we receive that $ V_j^{i,\hat{x},\epsilon}(0,\hat{x})> g_j(\hat{x}). $ Next, we proceed to the proof of non-negativity of $ V_j^{i,\hat{x},\epsilon}(0,\cdot)-g_j:\bar{\Omega}\to\R $. 
		From the last inequality and from the fact that, $ V_j^{i,\hat{x},\epsilon}(0,\cdot)-g_j$ is continuous on $\bar{\Omega} $ and especially at the point $ \hat{x}, $ we can select an appropriate radious $ r_{\hat{x}} $ such that $ \text{for all\ } x\in B(\hat{x}, r_{\hat{x}}) \cap \bar{\Omega},\ V_j^{i,\hat{x},\epsilon}(0,x)-g_j(x)>0. $ It remains to show that\\ $ \text{for all\ } x\in\bar{\Omega}\setminus B(\hat{x}, r_{\hat{x}}),\ V_j^{i,\hat{x},\epsilon}(0,x)-g_j(x)>0.  $ Clearly we see that the set $ \bar{\Omega}\setminus B(\hat{x}, r_{\hat{x}}) $ is closed in $ \R^n,  $ since $ \left( \bar{\Omega}\setminus B(\hat{x}, r_{\hat{x}})\right)^{c}=\bar{\Omega}^{c}\cup B(\hat{x}, r_{\hat{x}})   $  is open in $ \R^n. $ Moreover, the set $ \bar{\Omega}\setminus B(\hat{x}, r_{\hat{x}}) $ is bounded as a subset of the bounded set $ \bar{\Omega}. $ Consequently, $ \bar{\Omega}\setminus B(\hat{x}, r_{\hat{x}}) $ is a compact set, as closed and bounded set in $ \R^n. $ Let $ x\in\bar{\Omega}\setminus B(\hat{x}, r_{\hat{x}}) $, be a random point. We observe that
		\begin{align}
			V_j^{i,\hat{x},\epsilon}(0,x)-g_j(x)&=g_i(\hat{x})+A(\phi(x)-\phi(\hat{x}))+\tilde{B}\exp(\kappa \phi(x))\abs{x-\hat{x}}^2+\epsilon+c_{i,j}(0,x)-g_j(x)\nonumber\\
			&=K_{i,j}(x) +\epsilon
		\end{align}
		where $K_{i,j}  $ is the following function
		\begin{gather}
			K_{i,j}:\bar{\Omega}\setminus B(\hat{x}, r_{\hat{x}})\to\R,\ K_{i,j}(x):=k_1^{i,j}(x)+\tilde{B}\ k_2(x),\ \text{for all\ } x\in\bar{\Omega}\nonumber
		\end{gather}
		with 
		\begin{align}
			k_1^{i,j}(x)&:=g_i(\hat{x})-g_j(x)+c_{i,j}(0,x)+A(\phi(x)-\phi(\hat{x})),\ \text{for all\ } x\in\bar{\Omega}\setminus B(\hat{x}, r_{\hat{x}}) \nonumber\\
			k_2(x)&:=\exp(\kappa \phi(x))\abs{x-\hat{x}}^2,\ \text{for all\ } x\in\bar{\Omega}\setminus B(\hat{x}, r_{\hat{x}})\nonumber.
		\end{align}
		For our task, i.e $ \text{for all\ } x\in\bar{\Omega}\setminus B(\hat{x}, r_{\hat{x}}),\ V_j^{i,\hat{x},\epsilon}(0,x)-g_j(x)>0, $ it is sufficent to show that \\$ K_{i,j}\geq 0,\ \text{for all\ } x\in \bar{\Omega}\setminus B(\hat{x}, r_{\hat{x}}). $ Clearly, we see that $ k_2 $ is a non-negative function on $ \bar{\Omega}\setminus B(\hat{x}, r_{\hat{x}}). $
		The functions $ k_1^{i,j},\ k_2  $ as continuous on the compact set, attain a minimum value. In specific, let
		\begin{gather}
			k_2^{\min}:=\min\left\lbrace k_2(x)\ :\  x\in\bar{\Omega}\setminus B(\hat{x}, r_{\hat{x}})\right\rbrace\ \text{and}\ k_1^{\min}:=\min_{i,j\in\left\lbrace 1,2,\dots,m\right\rbrace }\left\lbrace \min\left\lbrace k_1^{i,j}(x)\ :\  x\in\bar{\Omega}\setminus B(\hat{x}, r_{\hat{x}})\right\rbrace\right\rbrace \nonumber.
		\end{gather} 
		It is obvious that $ k_2^{\min}>0, $ since $ \hat{x}\notin \bar{\Omega}\setminus B(\hat{x}, r_{\hat{x}})  $ and the non-negative expression $ \exp(\kappa\phi(x))\abs{x-\hat{x}}^2  $ is zero iff $ x=\hat{x} $. Then, $ \text{for all\ } i,j\in\left\lbrace 1,2,\dots,m\right\rbrace,\ \text{for all\ } x\in \bar{\Omega}\setminus B(\hat{x}, r_{\hat{x}})  $ and $ B\geq 0 $ we have 
		\begin{gather}
			\Delta_{\tilde{B}}:=k_1^{\min}+\tilde{B}\ k_2^{\min}\leq k_1^{i,j}(x)+\tilde{B}\ k_2(x)\equiv K_{i,j}(x)\nonumber.
		\end{gather}
		Consequently, in order to get the desired inequality, $ K_{i,j}(x)\geq 0,\ \text{for all\ } x\in \bar{\Omega}\setminus B(\hat{x}, r_{\hat{x}})  $, it is sufficent to show that $ \Delta_{\tilde{B}}\geq 0 $ for some appropriate $ \tilde{B}\geq 0. $ Indeed, the quantity $ \Delta_B:=k_1^{\min}+B\ k_2^{\min}\geq 0 $ is true  iff $$ \tilde{B}\geq \max\left\lbrace 0, -\frac{ k_1^{\min}}{k_2^{\min}}\right\rbrace.  $$
	\vspace{0.1em}	
	
		Finally, we proceed to the proof of $ g_i(\hat{x})=\inf_{\epsilon>0}V_i^{i,\hat{x},\epsilon}(0,\hat{x}) $. We have already seen, that \\$ V_i^{i,\hat{x},\epsilon}(0,\hat{x})=g_i(\hat{x})+\epsilon>g_i(\hat{x}),\ \text{for all\ }\epsilon>0. $ As a result from that, we receive that $ g_i(\hat{x}) $ is a lower bound of the non empty set $ \Delta:=\left\lbrace V_i^{i,\hat{x},\epsilon}(0,\hat{x})\ :\ \ \epsilon>0 \right\rbrace  $. Consequently, the $ \inf\Delta $ is well defined. We claim that $ g_i(\hat{x})=\inf\Delta.  $ By the characterization of infimum, it is sufficent to show that, 
		\begin{gather}
			\text{for all\ } \tilde{\epsilon}>0,\text{there exists\ } z_{\tilde{\epsilon}}\in\Delta: z_{\tilde{\epsilon}}<g_i(\hat{x})+\tilde{\epsilon}\nonumber.
		\end{gather}
		Indeed, considering a random $ \tilde{\epsilon}>0, $ we choose any $ \epsilon\in\left( 0,\tilde{\epsilon} \right),$ for example $ \epsilon=\tilde{\epsilon}/2 $ and set $ z_{\tilde{\epsilon}}:=g_i(\hat{x})+\tilde{\epsilon}/2\in\Delta $. Clearly, we see that $ z_{\tilde{\epsilon}}<g_i(\hat{x})+\tilde{\epsilon}$. As a result from the last, we get the desired equality.
	\end{proof}
	\begin{claim}\label{c_2}
	For a given $ A\in\R $ and any given $\kappa>0  $ there exists appropriate number $ B\geq\tilde{B}, $ such that, for all $(y,x)\in (0,L)\times\Omega,  $ and all $ (\alpha,p,X)\in \bar{J}^{2,-}V_j^{i,\hat{x},\epsilon}(y,x) $ the following holds:\\
		\begin{gather}
			\min\left\lbrace F(y,x,V_j^{i,\hat{x},\epsilon}(y,x),D_x V_j^{i,\hat{x},\epsilon}(y,x),D^2_{xx}V_j^{i,\hat{x},\epsilon}(y,x)),\ V_j^{i,\hat{x},\epsilon}(y,x)-\mathcal{M}_j V^{i,\hat{x},\epsilon}(y,x) \right\rbrace \geq 0\nonumber.
		\end{gather}
	\end{claim}
	\begin{proof} 
		Initially, we prove that $ \text{for all\ } (y,x)\in(0,L)\times\Omega, $
		\begin{gather}
			F(y,x,V_j^{i,\hat{x},\epsilon}(y,x),D_x V_j^{i,\hat{x},\epsilon}(y,x),D^2_{xx}V_j^{i,\hat{x},\epsilon}(y,x))\geq 0.
		\end{gather}
		From the axiom $ (F4) $, we receive that, $ \text{for all\ } (y,x)\in(0,L)\times\Omega, $
		\begin{gather}\label{c_2_F_4}
			F(y,x,V_j^{i,\hat{x},\epsilon}(y,x),D_x V_j^{i,\hat{x},\epsilon}(y,x),D^2_{xx}V_j^{i,\hat{x},\epsilon}(y,x))\geq V_j^{i,\hat{x},\epsilon}(y,x) .
		\end{gather}
	From the last, it is sufficient to prove that there exists an appropriate number $ B\geq \tilde{B} $ such that \\ $ V_j^{i,\hat{x},\epsilon}(y,x)\geq 0,\ \text{for all\ } (y,x)\in(0,L)\times\Omega. $ First, we have that 
	\begin{gather}
		V_j^{i,\hat{x},\epsilon}(y,x)=L_1(y,x)+L_2(x)\nonumber
	\end{gather}
	where 
	\begin{gather}
	L_1(y,x):=g_i(\hat{x})+\epsilon+C y+c_{i,j}(y,x)\nonumber\\
	L_2(x):= A \left( \phi(x)-\phi(\hat{x})\right) + B\exp\left( \kappa\phi(x)\right)\abs{x-\hat{x}}^2\nonumber.
	\end{gather}
	From the axiom $ \left(O_7 \right)  $, we obtain immediately that $ L_1(y,x)\geq 0,\ \text{for all\ } (y,x)\in(0,L)\times\Omega. $ Moreover, from the properties of function $ \phi, $ there exists $ \delta_{\hat{x}}>0 $ such that  for all $x\in\Omega_{\delta_{\hat{x}}}:=B_{\rho_2}\left(\hat{x},\delta \right)\cap\Omega, \\ \phi(x)-\phi(\hat{x}) \geq 0. $ Consequently, we receive that for all $(y,x)\in(0,L)\times\Omega_{\delta_{\hat{x}}}, V_j^{i,\hat{x},\epsilon}(y,x)=L_1(y,x)+L_2(x)\geq 0.  $ It remains to prove that for all $ (y,x)\in(0,L)\times\left( \Omega\setminus\Omega_{\delta_{\hat{x}}}\right) ,	V_j^{i,\hat{x},\epsilon}(y,x)\geq 0. $ Since, it holds $  L_1(y,x)\geq 0,\ \text{for all\ } (y,x)\in(0,L)\times\Omega, $  it is sufficent to show that $ \text{for all\ } x\in\Omega\setminus\Omega_{\delta_{\hat{x}}},\ L_2(x)\geq 0. $ We observe that the set $ \bar{\Omega}\setminus\Omega_{\delta_{\hat{x}}} $ is a closed and bounded subset of $ \R^n, $ i.e. a compact subset. Since the functions  $N_1(x)=A\left( \phi(x)-\phi(\hat{x}) \right) $ and $ N_2(x)=\exp\left(\kappa\phi(x) \right)\abs{x-\hat{x}}^2 $ are continuous on that compact set, they will be lower bounded on that set. Also, since $ \bar{\Omega}\setminus\Omega_{\delta_{\hat{x}}}\supset\Omega\setminus\Omega_{\delta_{\hat{x}}}, $ both functions will be lower bounded and on $ \Omega\setminus\Omega_{\delta_{\hat{x}}}. $ We set 
	\begin{gather}
		d_{N_1}:=\inf\left\lbrace N_1(x)\ :\ x\in\Omega\setminus\Omega_{\delta_{\hat{x}}}\right\rbrace \ \text{and}\ 
	    d_{N_2}:=\inf\left\lbrace N_2(x)\ :\ x\in\Omega\setminus\Omega_{\delta_{\hat{x}}}\right\rbrace\nonumber.
	\end{gather}
From the definition of function $ N_2, $ we see that its zero point is only the point $ \hat{x}, $ which is excluded from $ \bar{\Omega}\setminus\Omega_{\delta_{\hat{x}}}. $  Furthermore, 
\begin{gather}\label{d_n ineq}
	d_{N_2}\geq \tilde{d}_{N_2}:=\inf\left\lbrace N_2(x)\ :\ x\in\bar{\Omega}\setminus\Omega_{\delta_{\hat{x}}} \right\rbrace
\end{gather}
where $ \tilde{d}_{N_2}=N_2(\tilde{x}) $ for some $ \tilde{x}\in\bar{\Omega}\setminus\Omega_{\delta_{\hat{x}}}, $  since $ N_2 $ is continuous on the compact set $ \bar{\Omega}\setminus\Omega_{\delta_{\hat{x}}}.  $ Definitely, by the definition of $ N_2 $ we have $\tilde{d}_{N_2} \geq 0.$ In specific, $ \tilde{d}_{N_2} > 0 $, since if we suppose that $ N_2(\tilde{x})=0, $ we will have a zero point $ \tilde{x}\neq\hat{x} $, which is a contradiction. Finally, from $ (\ref{d_n ineq}) $ we get $d_{N_2}>0.  $
Then, we receive that $ \text{for all\ } x\in\Omega\setminus\Omega_{\delta_{\hat{x}}}, $
	\begin{gather} \label{d_n ineq_2}
		L_2(x)=N_1(x)+B N_2(x)\geq\inf\left\lbrace N_1(x)+B N_2(x)\ :\ x\in\Omega\setminus\Omega_{\delta_{\hat{x}}}\right\rbrace \geq d_{N_1}+ B d_{N_2}
	\end{gather}
If $ d_{N_1}\geq 0 $ then from  $ (\ref{d_n ineq_2}) $, we get the desired inequality  $ L_2(x)\geq 0,\ \text{for all\ } x\in\Omega\setminus\Omega_{\delta_{\hat{x}}}. $ If $ d_{N_1}<0 $, then from $ (\ref{d_n ineq_2}) $, if se set $ B\geq\max\left\lbrace \tilde{B},-\frac{d_{N_1}}{d_{N_2}} \right\rbrace  $, we receive again the desired inequality.\\
	Next, we proceed to the proof of 
		\begin{gather}
			V_j^{i,\hat{x},\epsilon}(y,x)-\mathcal{M}_j V^{i,\hat{x},\epsilon}(y,x)\geq 0,\ \text{for all\ } (y,x)\in(0,L)\times\Omega\nonumber.
		\end{gather}
		We have that $ \text{for all\ } (y,x)\in (0,L)\times\Omega, $
		\begin{flalign}
			V_j^{i,\hat{x},\epsilon}(y,x)-\mathcal{M}_j V^{i,\hat{x},\epsilon}(y,x)&= V_j^{i,\hat{x},\epsilon}(y,x)-\max_{\lambda\neq j}\left(V_{\lambda}^{i,\hat{x},\epsilon}(y,x) -c_{j,\lambda}(y,x)\right) \nonumber\\
			&= c_{i,j}(y,x)-\max_{\lambda\neq j}\left(c_{i,\lambda}(y,x)-c_{j,\lambda}(y,x)\right) \nonumber\\
			&= c_{i,j}(y,x)-\left( c_{i,\tilde{\lambda}}(y,x)-c_{j,\tilde{\lambda}}(y,x)\right)\ \left(  \text{for some}\ \tilde{\lambda}\neq j\right) \nonumber\\
			&= c_{i,j}(y,x)- c_{i,\tilde{\lambda}}(y,x)+c_{j,\tilde{\lambda}}(y,x)\label{claim_c2_eq1}
		\end{flalign}
		where the second equality is a result from the definition of function $ V_j^{i,\hat{x},\epsilon}. $ Using now the axiom $ (O_4) $, we receive that 
		\begin{gather}
			c_{i,j}(y,x)+c_{j,\tilde{\lambda}}(y,x) \geq c_{i,\tilde{\lambda}}(y,x)\nonumber\\
			\iff c_{i,j}(y,x)+c_{j,\tilde{\lambda}}(y,x) - c_{i,\tilde{\lambda}}(y,x)\geq 0\label{claim_c2_eq2}
		\end{gather}
		Finally, from $ (\ref{claim_c2_eq1}) $ and $ (\ref{claim_c2_eq2}) $ we get desired result.
	\end{proof} 
	\begin{claim}\label{c_3}
		For a given $ A\in\R,  $ there exists $ \kappa>0 $  such that
		\begin{gather}\label{c_3_2}
			\text{for all\ } j\in\left\lbrace 1,2,\dots,m\right\rbrace \text{for all\ } (y,x)\in \left( 0,L\right]\times\partial\Omega,\ V_j^{i,\hat{x},\epsilon}(y,x)\geq f_j(y,x).
		\end{gather}
	\end{claim}
	
	\begin{proof}
		From the Claim  \ref{c_1}, we have seen that for a given number $ A\in\R $, we selected appropriate number $ B\in\R^{+} $  such that the inequality $ V_j^{i,\hat{x},\epsilon}(0,x)> g_j(x),\ \text{for all\ } x\in\bar{\Omega} $ holds. In this claim, we proceed to the verification of an appropriate $ \kappa>0 $, such that $ \text{for all\ } (y,x)\in \left( 0,L\right]\times\partial\Omega, V_j^{i,\hat{x},\epsilon}(y,x)\geq f_j(y,x). $ We will notice below, that the proof of the last inequality is independent from the choice of the constant $ C. $ For this task, we are going to use the axiom $ (O_6). $ To prove the above inequality, we distinguish the following cases with respect to the position of $ \hat{x}. $
		\begin{itemize}
			\item[a)]$\hat{x}\in\partial\Omega$.\\
			We see that $ \text{for all\ } y\in\left( 0,L\right]  $
			\begin{align}
				V_j^{i,\hat{x},\epsilon}(y,\hat{x})&=g_i(\hat{x})+c_{i,j}(y,\hat{x})+Cy+\epsilon\nonumber\\
				&>^{(C\geq0,\epsilon>0)}\min_{\lambda\in\left\lbrace 1,2,\dots,m \right\rbrace }\{g_{\lambda}(x)\mid\ x\in\partial\Omega\}+ \min_{d,k\in\left\lbrace 1,2,\dots,m \right\rbrace }\{c_{d,k}(y,x)\mid\ y\in[0,L] x\in\partial\Omega\}\nonumber\\
				&\geq^{(O_6)} \max_{\lambda\in\left\lbrace 1,2,\dots,m \right\rbrace } \{f_{\lambda}(y,x) \mid y\in[0,L],\ x\in\partial\Omega\}\geq f_j(y,\hat{x})\nonumber
			\end{align}
			\item [b)] $ \hat{x}\in\Omega $.\\
			Due to the fact that $ \hat{x}\in\Omega  $ we receive that $ \hat{x}\notin\partial\Omega,\ i.e \left\lbrace \hat{x} \right\rbrace \cap\partial\Omega=\emptyset  $ (since $ \Omega $ is open). Also, from the compactness of $ \partial\Omega $ and the closed set $ \left\lbrace \hat{x}\right\rbrace  $, we receive that
			\begin{gather}
				d\equiv dist(\hat{x},\partial\Omega)=dist(\left\lbrace\hat{x} \right\rbrace ,\partial\Omega)>0.
			\end{gather} 
			Moreover, due to the continuity of $ \phi $ on the compact set $ \partial\Omega, $ the function $ \phi $ attains a minimum value. We set $ \phi_{\min}:=\min\{\phi(x)\mid x\in\partial\Omega\}=\phi(\tilde{x}) $, for some $ \tilde{x}\in\partial\Omega. $ Also, $ \phi_{\min}>0, $ since $ \phi(x)>0 $ on $ \bar{\Omega}. $ By the same argument, the function $ \Gamma_A:=A (\phi(x)-\phi(\hat{x})) $ attains a minimum value, which is denoted as $ \Gamma_A^{\min}. $
			Then, $ \text{for all\ } y\in(0,L],\ \text{for all\ } x\in\partial\Omega $ we observe that
			\begin{align}
				V_j^{i,\hat{x},\epsilon}(y,x)&=g_i(\hat{x})+ A\left(\phi(x)-\phi(\hat{x}) \right)+ B \exp(\kappa\phi(x))\abs{x-\hat{x}}^2+C y+ c_{i,j}(y,x)+\epsilon\nonumber\\
				&>^{(\exp\text{is increasing func.} )}_{(\epsilon>0)}\min_{\lambda\in\{1,2,\dots,m\}}\{g_{\lambda}(x)\mid x\in\partial\Omega\}+\nonumber\\
				&+\min_{d,k\in\left\lbrace 1,2,\dots,m \right\rbrace }\{c_{d,k}(y,x)\mid\ y\in[0,L] x\in\partial\Omega\}+ \Gamma_A^{\min}+B\ d^2\exp(k\phi_{\min})\nonumber\\
				&\geq^{(O_6)}\max_{\lambda\in\left\lbrace 1,2,\dots,m \right\rbrace } \{f_{\lambda}(y,x) \mid y\in[0,L],\ x\in\partial\Omega\}+\Gamma_A^{\min}+B\ d^2\exp(k\phi_{\min})\nonumber\\
				&\geq f_j(y,x)+\Gamma_A^{\min}+B\ d^2\exp(k\phi_{\min})\nonumber.
			\end{align}
			From the last inequality, in order to show that $ V_j^{i,\hat{x},\epsilon}(y,x)>f_j(y,x),  $ it is sufficent to show that the quantity $ J:= \Gamma_A^{\min}+B\ d^2\exp(k\phi_{\min})$ is non-negative. Indeed,\\
			$$  J\geq 0 \iff \kappa\geq\kappa_0:=\frac{-\Gamma_A^{\min}}{B\ d^2\ \phi_{\min}}$$
			As a result from the last inequality, we can select for $ \kappa $, to be every value in the interval\\ $ \left[ \max\{1,\kappa_0\},\infty\right)  $, since we want  $ \kappa $ to be positive.
		\end{itemize} 
		In both cases, we observe that, for any given number A, the determination of the constant $ C, $ does not play any role for the satisfaction of our defined target: $ V_j^{i,\hat{x},\epsilon}(y,x)>f_j(y,x),\ \text{for all\ } y\in(0,L],\\ \text{for all\ } x\in\partial\Omega. $
	\end{proof} 
	Following an analogous procedure, as we did for the function $ V^{i,\hat{x},\epsilon}:[0,L]\times\bar{\Omega}\to\R^m $, for any given number $ A\in\R, $ we prove that there exists constants $ \tilde{B}\in\R^{+} $ and $ \tilde{\kappa}>0 $ such that the function $  U^{\hat{x},\epsilon}:[0,L]\times\bar{\Omega}\to\R^m $ is a viscosity subsolution of the problem $ (BVP) $ and moreover satisfies $ \sup_{\epsilon>0}U_i^{\hat{x},\epsilon}(0,\hat{x})=g_i(\hat{x}),\ \text{for all\ } i\in\{1,2,,\dots,m\}. $ Additionally, for any given number $ A\in\R $, we can choose appropriate constants $ B\in\R^{+} $ and $ \kappa>0 $ such that, both functions\\ $ V^{i,\hat{x},\epsilon}:[0,L]\times\bar{\Omega}\to\R^m $ and $ U^{\hat{x},\epsilon}:[0,L]\times\bar{\Omega}\to\R^m  $ satisfy respectively the definition of viscosity supersolution and subsolution for the problem $ (BVP). $
\end{proof}
\begin{proof}(\textbf{of Proposition $ \ref{Perronmethod} $}) 
	Throughout the proof, we fix a random index $ i\in\{1,2,\dots,m\}. $ Let $ \epsilon_o>0 $ be a random point. Then from Proposition $ \ref{existence of barriers} $ we receive that the function\\ $ V_j^{i,\hat{x},\epsilon_o}:[0,L]\times\bar{\Omega}\to\R,\ j\in\{1,2,\dots,m\}, $ attains a maximum value as a continuous on a compact set. Let $ M_{\epsilon_o}^j:=\max_{[0,L]\times\bar{\Omega}}V_j^{i,\hat{x},\epsilon_o} $ and $ M_{\epsilon_o}:=\max\{M_{\epsilon_o}^j\ :\ \ j=1,2,\dots,m \} $. Next, we define the following function, $ w:[0,L]\times\bar{\Omega}\to\R^m $, where $ w=(w_1,w_2,\dots,w_m) $ with $ w_i:[0,L]\times\bar{\Omega}\to\R,\\ i\in\{1,2,\dots,m\}, $ defined as 
	\begin{gather}
		w_i(y,x):=\sup\{u_i(y,x)\ :\ \ u\in\mathcal{F}\}, \text{for all\ }(y,x)\in[0,L]\times\bar{\Omega}
	\end{gather}
where
\begin{gather}
	\mathcal{F}:=\{u=(u_1,u_2,\dots,u_m):[0,L]\times\bar{\Omega}\to\R \ : \ u\ \text{subsolution of (BVP)}, u_i\leq M_{\epsilon_o}\ \text{on}\ [0,L]\times\bar{\Omega}\}\nonumber.
\end{gather}
	Initially, we prove that the above function  is well defined and bounded as well. Later, it is proved that the function $ w_i $ satisfies the definition of viscosity solution.   
	\begin{claim}\label{well defined}
		The function $ w_i:[0,L]\times\bar{\Omega}\to\R  $ is well defined and bounded. 
	\end{claim}
	\begin{proof}
		By saying well defined, we mean that, the function $ w_i:[0,L]\times\bar{\Omega}\to\R$ attains a real value for each $ (y,x)\in[0,L]\times\bar{\Omega}.$ We define for each $ (y,x)\in[0,L]\times\bar{\Omega},\ i\in\{1,2,\dots,m\} $ the following set:
		$$A^i_{y,x}:=\{u_i(y,x)\ :\ \ u\in\mathcal{F}\}.$$
		We observe that the above set is non empty. Indeed, we know from the Proposition $ \ref{existence of barriers} $ that there exists a family of continuous subsolutions $ (U_i^{\hat{x},\epsilon})_{\epsilon>0} $ of the problem (BVP). Since the function\\ $ V_i^{i,\hat{x},\epsilon_o} $ is a supersolution of (BVP), then from the Comparison Principle, we receive that \\ $ U_i^{\hat{x},\epsilon}(t,x)\leq V_i^{i,\hat{x},\epsilon_o}(y,x),   $ on $ [0,L)\times\bar{\Omega} $. Furthermore, from the last inequality and the fact that $ U_i^{\hat{x},\epsilon}, V_i^{i,\hat{x},\epsilon_o}  $ are continuous at the point $ (L,x)\in[0,L]\times\bar{\Omega} $, we receive that $ U_i^{\hat{x},\epsilon} (L,x)\leq V_i^{i,\hat{x},\epsilon_o} (L,x) $. In conclution, we have that $ U_i^{\hat{x},\epsilon} (y,x)\leq V_i^{i,\hat{x},\epsilon_o} (y,x)\leq M_{\epsilon_o}. $ As a result the family $ (U_i^{\hat{x},\epsilon})_{\epsilon>0} $, is contained in $ A^i_{y,x}:=\{u_i(y,x)\ :\  u\in\mathcal{F} \} $. It remains to prove that the set $ A^i_{y,x}  $ is upper bounded for each $(y,x)\in[0,L]\times\bar{\Omega}.$ Then, by completeness we receive that $ w_i(y,x):=\sup A^i_{y,x} $ is located in $ \R. $ Let $ (y_o,x_o) $ be a random point in $ [0,L]\times\bar{\Omega}. $ Then, $ u_i(y_o,x_o)\leq M_{\epsilon_o}, \text{for all\ } u\in\mathcal{F}. $ From the last we see that $ A^i_{y_o,x_o}  $ is upper bounded from $ M_{\epsilon_o} $. Moreover, we have that $ w_i(y_o,x_o)=\sup A^i_{y_o,x_o}\leq M_{\epsilon_o}  $. I.e $ w_i $ is upper bounded by $ M_{\epsilon_o}. $
		We proceed to the proof that function $ w_i $ is lower bounded. Indeed, we choose a random $ \epsilon^{*}>0. $ Then, from the fact that $ U_i^{\hat{x},\epsilon^*}\in\mathcal{F}, $ is continuous on the compact set $ [0,L]\times\bar{\Omega}, $ will attain a minimum value $ m:=\min_{(y,x)\in[0,L]\times\bar{\Omega}}U_i^{\hat{x},\epsilon^*}(y,x) $. Then, we have that $ \text{for all\ } (y,x)\in[0,L]\times\bar{\Omega}, $
		\begin{gather}
			m\leq U_i^{\hat{x},\epsilon^*}(y,x)\leq w_i(y,x):=\sup\{u_i(y,x)\ :\ u\in\mathcal{F} \},
		\end{gather}
		consequently $ w_i $ is lower bounded. 
	\end{proof}

	We proceed to the proof that function $ w_i:[0,L]\times\bar{\Omega}\to\R $ is a viscosity solution of (IBVP).
	For this task, we will need the notions of lower and upper semicontinuous envelopes\footnote{Let $ (X,\rho) $ be a metric space, $ A\subset X $ and $ f:A\to[-\infty,\infty]. $ We define as lower semicontinuous envelope of the function $ f $, the function
		\begin{gather}
			f_{*}:A\to[-\infty,\infty],\ f_{*}(x):=\lim_{\delta\to 0}\inf\{f(y)\ :\ y\in A\cap B_{\rho}(x,\delta)\},\ \text{for all\ } x\in A\nonumber
		\end{gather}
		Analogously, we define as upper semicontinuous envelope of the function $ f $, the function
		\begin{gather}
			f^{*}:A\to[-\infty,\infty],\ f^{*}(x):=\lim_{\delta\to 0}\sup\{f(y)\ :\ y\in A\cap B_{\rho}(x,\delta) \},\ \text{for all\ } x\in A\nonumber
		\end{gather}
	} of $ w. $ In specific, we denote $ \text{for all\ } i\in\{1,2,\dots,m\} $
	\begin{gather}
		w_{*,i}:[0,L]\times\bar{\Omega}\to\bar{\R}\nonumber\ \text{and}\ w_{i}^{*}:[0,L]\times\bar{\Omega}\to\bar{\R}\nonumber
	\end{gather}
	to be the lower and the upper semicontinuous envelope of function $ w $ respectively. By their definition, it is extrtacted that the lower semicontinuous envelope $ w_{*,i} $ of $ w_i $ is the largest lower semicontinuous function that it is dominated by function $ w_i $. Moreover, the upper semicontinuous function $ w_{i}^{*} $ is the smallest upper semicontinuous function that dominates the function $ w_i $. So we have
	\begin{gather}
		w_{*,i}\leq w_i,\ \text{on}\ [0,L]\times\bar{\Omega}\ \text{and}\ \text{for all\ } h\in LSC([0,L]\times\bar{\Omega}),\ h\leq w_i\Rightarrow h\leq w_{*,i}\label{perron_1} \\
		w_{i}^{*}\geq w_i,\ \text{on}\ [0,L]\times\bar{\Omega}\ \text{and}\ \text{for all\ } h\in USC([0,L]\times\bar{\Omega}),\ h\geq w_i\Rightarrow h\geq 	w_{i}^{*}\label{perron_2}.
	\end{gather}
	Also, we observe that $ w_{*,i} $ and $ w_i^{*} $ take real values and in particular, both envelopes are bounded functions by the numbers $ m $ and $ M_{\epsilon_o}. $ I.e $ Im( w_{i,*})\subset [m,M_{\epsilon_o}] $ and $ Im(w_i^{*})\subset[m,M_{\epsilon_o}]. $ The last is a consequence from the definition of  $ w_{*,i} $ and $ w_i^{*} $ and the fact that $ w_i $ is bounded function.
	\begin{claim}\label{bound_wi} 
		For the function $ w_i:[0,L]\times\bar{\Omega}\to\R, $ the folowing holds:
		$$\text{for all\ }\epsilon>0,\ \text{for all\ }(t,x)\in[0,L)\times\bar{\Omega},\ w_i(y,x)\leq V_i^{i,\hat{x},\epsilon}(y,x).$$
	\end{claim}
	\begin{proof}
		Let $ \epsilon>0 $ be a random point, that we fix and let $ u\in\ F $ be random function. Then, from the comparison principle, since $ V_i^{i,\hat{x},\epsilon} $ is a supersolution of $ (BVP) $, it is extracted that $ u_i(y,x)\leq V_i^{i,\hat{x},\epsilon}(y,x),\ \text{for all\ } (y,x)\in[0,L)\times\bar{\Omega}. $ By setting  a fixed point $ (y_o,x_o)\in[0,L)\times\bar{\Omega},$ we receive $ u_i(y_o,x_o)\leq V_i^{i,\hat{x},\epsilon}(y_o,x_o),\ \text{for all\ } u\in\mathcal{F}. $ From the last, we see that the real number $ V_i^{i,\hat{x},\epsilon}(y_o,x_o) $ is an upper bound of the set $ A^i_{(y_o,x_o)} $. By the definition of $ w_i  $, we conclude that\\ $ w_i(y_o,x_o)=\sup A^i_{(y_o,x_o)}\leq V_i^{i,\hat{x},\epsilon}(y_o,x_o).$ Because $ (y_o,x_o)\in[0,L)\times\bar{\Omega} $ is a random point, we have the desired result $  w_i(y,x)\leq V_i^{i,\hat{x},\epsilon}(y,x),\text{for all\ } (y,x)\in [0,L)\times\bar{\Omega}. $
	\end{proof}
	
	\begin{claim}\label{viscosity_property}
		The function $ w:[0,L]\times\bar{\Omega}\to\R^m $ satisfies the axioms of viscosity solution for (IBVP).
	\end{claim}
	\begin{proof}
		Let $ i\in\{1,2,\dots,m\} $ be a random index. From $ (\ref{perron_1}) $ and $ (\ref{perron_2}) $ we receive that
		\begin{gather}\label{visco_1}
			w_{*,i}\leq w_i \leq w^*_{i},\ \text{on}\ [0,L]\times\bar{\Omega}.
		\end{gather} 
		The main idea, is to prove that  $ w_{*,i}  $ is a \textit{supersolution} of (BVP) and $ w^*_{i} $ is a \textit{subsolution} of (BVP). Then, from the application of Comparison Principle, we receive that 
		\begin{gather}\label{visco_2}
			w^*_{i}\leq w_{*,i}\ \text{on}\ [0,L)\times\bar{\Omega}.
		\end{gather}
		Combining $ (\ref{visco_1}) $ and $ (\ref{visco_2}) $, we obtain that $ w_i=w^*_{i}=w_{*,i},\ \text{on}\ [0,L)\times\bar{\Omega}. $ From the last equality, we conclude that the function $ w_i:[0,L]\times\bar{\Omega}\to\R $  satisfies the axioms of viscosity solution. \\
		
		First, we establish the subsolution property of $ w^*_{i}. $
		\begin{itemize}
			\item We proceed to the proof that $ w^{*}_i(0,x) \leq g_i(x),\ \text{for all\ } x\in\bar{\Omega}.$\\
			From the Claim $ (\ref{bound_wi}) $, we receive that $ \text{for all\ }\epsilon\in(0,\epsilon_o],\ \text{for all\ }(y,x)\in[0,L)\times\bar{\Omega},\ w_i(y,x)\leq V_i^{i,\hat{x},\epsilon}(y,x). $ The function $ V_i^{i,\hat{x},\epsilon},\ \epsilon\in(0,\epsilon_o], $ as continuous on $ [0,L]\times\bar{\Omega}, $ is upper semicontinuous on that set. In the case that, there exists at least one $ \epsilon^*\in(0,\epsilon_o] $ and a point $ \tilde{x}\in\bar{\Omega}$ such that $ w_i(L,\tilde{x})>V_i^{i,\hat{x},\epsilon^*}(L,\tilde{x}) $, then we can do an appropriate modification\footnote{We set as a modification of function $ V_i^{i,\hat{x},\epsilon^*} $ the function $ \tilde{V}_i^{i,\hat{x},\epsilon^*}:[0,L]\times\bar{\Omega}\to\R $ defined as\\ $ \tilde{V}_i^{i,\hat{x},\epsilon^*}(y,x) = 
				\begin{cases}
					V_i^{i,\hat{x},\epsilon^*}(t,x) &\quad (y,x)\in[0,T)\times\bar{\Omega}\\
					M_{\epsilon_o} & y=L,\ x\in\bar{\Omega} \\
				\end{cases} $, where $ M_{\epsilon_o}=\max_{(y,x)\in [0,L]\times\bar{\Omega}} V_i^{i,\hat{x},\epsilon_o}(y,x), $ which also satisfies \\ $  M_{\epsilon_o}>=\max_{(y,x)\in [0,L]\times\bar{\Omega}} V_i^{i,\hat{x},\epsilon}(y,x),\ \text{for all\ }\epsilon\in(0,\epsilon_o].  $ Then, the function $ \tilde{V}_i^{i,\hat{x},\epsilon^*} $ remains upper semicontinuous on its domain and satisfies at the same time $ \tilde{V}_i^{i,\hat{x},\epsilon^*}\geq V_i^{i,\hat{x},\epsilon*} $ on $ [0,L]\times\bar{\Omega},\ \text{for all\ }\epsilon>0 $.} of function $ V_i^{i,\hat{x},\epsilon^*} $, let this modification denoted as $\tilde{V}_i^{i,\hat{x},\epsilon^*} $ which remains upper semicontinuous on $ [0,L]\times\bar{\Omega} $ and satisfies at the same time $ w_i(y,x)\leq \tilde{V}_i^{i,\hat{x},\epsilon^*}(y,x),\ \text{for all\ } (y,x) \in [0,L]\times\bar{\Omega}. $ Then, from the relation $ (\ref{perron_2}) $, the following holds: 
			\begin{gather}
				w_i^{*}(y,x)\leq \left( \tilde{V}_i^{i,\hat{x},\epsilon^*}\right) ^*(y,x),\ \text{for all\ }(y,x)\in[0,L]\times\bar{\Omega}.
			\end{gather}
			Specifically, for $ \hat{x}\in\bar{\Omega} $ randomly selected, we receive that
			$$w_i^{*}(0,\hat{x})\leq \left(\tilde{V}_i^{i,\hat{x},\epsilon^*}\right)^* (0,\hat{x})=\left( V_i^{i,\hat{x},\epsilon^*}\right)^* (0,\hat{x})\stackrel{V_i^{i,\hat{x},\epsilon^*}\text{continuous} }{=}V_i^{i,\hat{x},\epsilon^*}(0,\hat{x}),$$
			where the second inequality is a known result\footnote{Let $ (X,\rho)  $ be a metric space, $ A\subset X,\ x_o\in A $ and $ f,g:A\to\R $ mappings. If there exists $ \delta_0>0, $ such that\\ $ \text{for all\ } x\in B_{\rho}\left(x_o,\delta_0 \right)\cap A,\ f(x)=g(x)  $ then, $ f^*(x_o)=g^*(x_o) $ and $ f_*(x_o)=g_*(x_o) $}. In conclusion the above pathological case, is solved. For the rest $ \epsilon\in(0,\epsilon_o],  $ which satisfy the inequality $ w_i(y,x)\leq V_i^{i,\hat{x},\epsilon}(y,x), \\ \text{for all\ } (y,x)\in[0,L]\times\bar{\Omega} $, we follow a similar procedure to prove that $ w^*_i(0,\hat{x})\leq V_i^{i,\hat{x},\epsilon}(0,\hat{x}) $. In specific, from the last inequality, from the relation $ (\ref{perron_2}) $, the following holds
			\begin{gather}
			w_i^{*}(y,x)\leq \left(V_i^{i,\hat{x},\epsilon}\right) ^*(y,x),\ \text{for all\ }(y,x)\in[0,L]\times\bar{\Omega}.
			\end{gather} 
		Then for the previous selected $ \hat{x}\in\bar{\Omega} $, we receive that 
		$$w_i^{*}(0,\hat{x})\leq \left(V_i^{i,\hat{x},\epsilon}\right)^* (0,\hat{x})=\stackrel{V_i^{i,\hat{x},\epsilon}\text{continuous} }{=}V_i^{i,\hat{x},\epsilon}(0,\hat{x}).$$
		Consequently, from the above analysis, we receive that 
		$$w_i^{*}(0,\hat{x})\leq V_i^{i,\hat{x},\epsilon}(0,\hat{x}),\ \text{for all\ }\epsilon\in(0,\epsilon_o].$$
		From the above inequality, we conclude that the real number $ w_i^{*}(0,\hat{x})  $ is a lower bound of the set $ \{V_i^{i,\hat{x},\epsilon}(0,\hat{x})\ : \ \epsilon\in(0,\epsilon_o] \} $. Consequently, from the definition of infimum, we obtain that 
			\begin{gather}
				w_i^{*}(0,\hat{x})\leq\inf_{\epsilon\in(0,\epsilon_o]}V_i^{i,\hat{x},\epsilon}(0,\hat{x})=\inf_{\epsilon>0}V_i^{i,\hat{x},\epsilon}(0,\hat{x})\stackrel{\text{Proposition}\ \ref{existence of barriers}}{=}g_i(\hat{x})
			\end{gather}
		where the $ \inf_{\epsilon\in(0,\epsilon_o]}V_i^{i,\hat{x},\epsilon}(0,\hat{x})=\inf_{\epsilon>0}V_i^{i,\hat{x},\epsilon}(0,\hat{x}) $ holds since the function $ V_i^{i,\hat{x},\epsilon}(0,\hat{x}) $ is of the form $ V_i^{i,\hat{x},\epsilon}(0,\hat{x}) =h+\epsilon $, where $ h $ is a constant number. 
		Because $ \hat{x}\in\bar{\Omega} $ was a random point, the desired inequality is established.
			\item We proceed to the proof that $\text{for all\ } \left(y,x \right)\in(0,L)\times\Omega,\text{for all\ }\left(\alpha, p,X \right)\in \bar{J}^{2,+}w^{*}_i(y,x),$
			$$\min\{F\left(y,x,w_i^{*}(y,x),p,X\right),w^{*}_i(y,x)-\mathcal{M}_i\textbf{w}^{*}(y,x)\} \leq 0.$$
			Let $ (\hat{y},\hat{x})\in(0,L)\times\Omega $ be a random point, and $ (p,X)\in \bar{J}^{2,+}w^*_i(\hat{y},\hat{x}). $ By the definition of $ w_i^*, $ there exists a sequence
			\begin{gather}\label{convergence of envelope}
				(y_n,x_n,u_i^n\left(y_n,x_n \right) ) \xrightarrow{n\to\infty} \left(\hat{y},\hat{x},w^*_i(\hat{y},\hat{x})\right)
			\end{gather}
			where $ (y_n,x_n)\in[0,L]\times\bar{\Omega}  $ and $ \left( u_i^n \right)_{n\in\N}  $ an appropriate sequence of subsolutions of $ (BVP) $. Furthermore, if $ z_n=(\tilde{y}_n,\tilde{x}_n) $ is a sequence of points in $ [0,L]\times\bar{\Omega} $ such that\\ $ z_n\xrightarrow{n\to\infty} (\tilde{y}_o,\tilde{x}_o)\in[0,L]\times\bar{\Omega}, $ then $ \limsup_{n\to\infty}u_i^n(z_n)\leq w^*_i(\tilde{y}_o,\tilde{x}_o). $ Indeed, due to the fact that $ w^*_i $ is upper semicontinuous function, then by the sequential criterion, we receive that 
			\begin{gather}
				\limsup_{n\to\infty} w^*_i(z_n)\leq w^*_i(\tilde{y}_o,\tilde{x}_o)\label{sequential_criterion_w^*_i}.
			\end{gather}
			Moreover, by the definition of $ w^*_i, $ and $ w_i $ it follows that 
			\begin{gather}
				\text{for all\ } n\in\N,\ w^*_i(z_n)\geq w_i(z_n),\  w_i(z_n)\geq u_i^n(z_n).
			\end{gather}
			 From the last inequalities, we receive
			\begin{gather}
				\limsup_{n\to\infty} w^*_i(z_n)\geq\limsup_{n\to\infty} w_i(z_n)\label{perron_condition_1}\\
				\limsup_{n\to\infty} w_i(z_n)\geq\limsup_{n\to\infty} u_i^n(z_n)\label{perron_condition_2}.
			\end{gather}
			By the combination of $ (\ref{sequential_criterion_w^*_i}),\ (\ref{perron_condition_1}) $ and $ (\ref{perron_condition_2}) $ it is extracted, that $ \limsup_{n\to\infty}u^n_i(z_n)\leq w^*_i(\tilde{y}_o,\tilde{x}_o). $ Then, we get from Proposition 4.3 of $ \cite{CIL} $ that there exists a sequence $ (\hat{y}_n,\hat{x}_n)\in [0,L]\times\bar{\Omega} $ and \\$ (p_n,X_n)\in \bar{J}^{2,+}u^n_i(\hat{y}_n,\hat{x}_n) $ such that,
			\begin{gather}
				\left(\hat{y}_n,\hat{x}_n,u^n_i(\hat{y}_n,\hat{x}_n),p_n,X_n \right) \xrightarrow{n\to\infty}\left( \hat{y},\hat{x},w^*_i(\hat{y},\hat{x}),p,X\right). 
			\end{gather}
		Because the point $ (\hat{y},\hat{x}) $ is an internal point of $ (0,L)\times\Omega, $ thus there exists a $ m_o\in\N $ such that,\\ $ \text{for all\ } n\geq m_o, (\hat{y}_n,\hat{x}_n)\in (0,L)\times\Omega $. Also, from the continuity of $ F $, we receive that 
			\begin{gather}\label{continuity_of_F}
				F\left(\hat{y}_n,\hat{x}_n,u^n_i(\hat{y}_n,\hat{x}_n),p_n,X_n \right) \xrightarrow{n\to\infty} F\left( \hat{y},\hat{x},w^*_i(\hat{y},\hat{x}),p,X\right).
			\end{gather}
		From the fact that $ w_j^* $ is an upper semicontinuous function, we receive that 
		\begin{gather}\label{w_j_estimation}
			w_j^*(\hat{y},\hat{x})\geq\limsup_{n\to\infty} w_j^*(\hat{y}_n,\hat{x}_n) =\limsup_{n\to\infty} w_j^*(\hat{y}_{n+m_o},\hat{x}_{n+m_o}).
		\end{gather}
		Then, we have that
		\begin{align}
			\max_{j\neq i}\left( w_j^*(\hat{y},\hat{x})-c_{i,j}(\hat{y},\hat{x})\right) &\geq^{(\ref{w_j_estimation})}\max_{j\neq i}\left(\limsup_{n\to\infty} w_j^*(\hat{y}_{n+m_o},\hat{x}_{n+m_o})-c_{i,j}(\hat{y},\hat{x})\right)\nonumber\\
			&\stackrel{(c_{i,j}\ contin.)}{=}\max_{j\neq i}\left(\limsup_{n\to\infty}\left(  w_j^*(\hat{y}_{n+m_o},\hat{x}_{n+m_o})-c_{i,j}(\hat{y}_{n+m_o},\hat{x}_{n+m_o})\right)\right) \nonumber\\
			&\geq\limsup_{n\to\infty}\max_{j\neq i}\left( w_j^*(\hat{y}_{n+m_o},\hat{x}_{n+m_o})-c_{i,j}(\hat{y}_{n+m_o},\hat{x}_{n+m_o})\right) \nonumber\\
			&\geq\limsup_{n\to\infty}\left( \max_{j\neq i}\left( u_j^{n+m_o}(\hat{y}_{n+m_o},\hat{x}_{n+m_o})-c_{i,j}(\hat{y}_{n+m_o},\hat{x}_{k_n+m_o})\right) \right). 
		\end{align}
		From the last inequality, it follows that $ \text{for all\ } n\in\N, $
		\begin{align}
		a_n:&=u^{n+m_o}_i(\hat{y}_{n+m_o},\hat{x}_{n+m_o})-\max_{j\neq i}\left( w_j^*(\hat{y},\hat{x})-c_{i,j}(\hat{y},\hat{x})\right)\nonumber\\
		& \leq u^{n+m_o}_i(\hat{y}_{n+m_o},\hat{x}_{n+m_o})-\nonumber\\
		&-\limsup_{n\to\infty}\left( \max_{j\neq i}\left( u_j^{n+m_o}(\hat{y}_{n+m_o},\hat{x}_{n+m_o})-c_{i,j}(\hat{y}_{n+m_o},\hat{x}_{n+m_o})\right)\right) :=b_n\nonumber.
		\end{align}
	Thus, $ \text{for all\ } n\in\N,\ a_n\leq b_n $ and as a result from the last, we conclude that $ \lim_{n\to\infty}a_n\leq\lim_{n\to\infty}b_n $. The last inequality, it is true, since from the expressions of the sequences $ (a_n)_{n\in\N} $ and $ (b_n)_{n\in\N}, $ it is clear that their limits exists. If we translate the last limit inequality, we observe that
	\begin{align}
		& w_i^*(\hat{y},\hat{x})-\max_{j\neq i}\left( w^*_j(\hat{y},\hat{x})-c_{i,j}(\hat{y},\hat{x})\right)\nonumber\\
		& \leq w_i^*(\hat{y},\hat{x})-\limsup_{n\to\infty}\left(\max_{j\neq i}\left( u^{n+m_o}_j(\hat{y}_{n+m_o},\hat{x}_{n+m_o}) -c_{i,j}(\hat{y}_{n+m_o},\hat{x}_{n+m_o})\right)  \right) \nonumber.
	\end{align}
From the definition of $ \limsup_{n\to\infty}, $ there exists a strictly increasing sequence of natural numbers $ (k_n)_{n\in\N} $, such that
\begin{align}
	&\limsup_{n\to\infty}\left(\max_{j\neq i}\left( u^{n+m_o}_j(\hat{t}_{n+m_o},\hat{x}_{n+m_o}) -c_{i,j}(\hat{y}_{n+m_o},\hat{x}_{n+m_o})\right)  \right)\nonumber\\ &=\lim_{n\to\infty}\max_{j\neq i}\left(u_j^{k_n+m_o} \left( \hat{y}_{k_n+m_o},\hat{x}_{k_n+m_o}\right) -c_{i,j}\left( \hat{y}_{k_n+m_o},\hat{x}_{k_n+m_o}\right) \right) \nonumber.
\end{align}
	At the same time, we observe that
	\begin{gather}
		w^*_i(\hat{y},\hat{x})=\lim_{n\to\infty}u^n_i(\hat{y}_n,\hat{x}_n)=\lim_{n\to\infty} u^{k_n+m_o}_i(\hat{y}_{k_n+m_o},\hat{x}_{k_n+m_o}).
	\end{gather}
Combining the last two limit equalities, we receive that
\begin{align}
	& w_i^*(\hat{y},\hat{x})-\max_{j\neq i}\left( w^*_j(\hat{y},\hat{x})-c_{i,j}(\hat{y},\hat{x})\right)\nonumber\\ 
	&\leq \lim_{n\to\infty}\left(u^{k_n+m_o}_i(\hat{y}_{k_n+m_o},\hat{x}_{k_n+m_o})- \max_{j\neq i}\left(u_j^{k_n+m_o} \left( \hat{y}_{k_n+m_o},\hat{x}_{k_n+m_o}\right) -c_{i,j}\left( \hat{y}_{k_n+m_o},\hat{x}_{k_n+m_o}\right) \right) \right) \nonumber
\end{align}
thus,
	\begin{align}
		&\min\left\lbrace F\left( \hat{y},\hat{x},w^*_i(\hat{y},\hat{x}),p,X\right) ,w_i^*(\hat{y},\hat{x})-\max_{j\neq i}\left( w^*_j(\hat{y},\hat{x})-c_{i,j}(\hat{y},\hat{x})\right)\right\rbrace \nonumber\\
		&\leq\min\biggl\{ F\left( \hat{y},\hat{x},w^*_i(\hat{y},\hat{x}),p,X\right) ,\lim_{n\to\infty}\biggl(u^{k_n+m_o}_i(\hat{y}_{k_n+m_o},\hat{x}_{k_n+m_o})-\nonumber\\ 
		&\qquad\qquad-\max_{j\neq i}\left(u_j^{k_n+m_o} \left( \hat{y}_{k_n+m_o},\hat{x}_{k_n+m_o}\right) - c_{i,j}\left( \hat{y}_{k_n+m_o},\hat{x}_{k_n+m_o}\right)\right)\biggl)\biggl\} \nonumber\\
		&=^{(\ref{continuity_of_F})}\min\biggl\{ \lim_{n\to\infty}F(\hat{y}_{k_n+m_o},\hat{x}_{k_n+m_o},u^{k_n+m_o}_i(\hat{t}_{k_n+m_o},\hat{x}_{k_n+m_o}),p_{k_n+m_o},X_{k_n+m_o}),\nonumber\\ &\quad\lim_{n\to\infty}\biggl(u^{k_n+m_o}_i(\hat{y}_{k_n+m_o},\hat{x}_{k_n+m_o})-\max_{j\neq i}\left(u_j^{k_n+m_o} \left( \hat{y}_{k_n+m_o},\hat{x}_{k_n+m_o}\right) - c_{i,j}\left( \hat{y}_{k_n+m_o},\hat{x}_{k_n+m_o}\right) \right) \biggl)  \biggl\}\nonumber\\
		&=\lim_{n\to\infty} \min\biggl\{ F\left( \hat{y}_{k_n+m_o},\hat{x}_{k_n+m_o},u^{k_n+m_o}_i(\hat{y}_{k_n+m_o},\hat{x}_{k_n+m_o}),p_{k_n+m_o},X_{k_n+m_o}\right) ,\nonumber\\
		&\qquad u^{k_n+m_o}_i(\hat{t}_{k_n+m_o},\hat{x}_{k_n+m_o}) 
		-\max_{j\neq i}\left(u_j^{k_n+m_o} \left( \hat{y}_{k_n+m_o},\hat{x}_{k_n+m_o}\right) - c_{i,j}\left( \hat{y}_{k_n+m_o},\hat{x}_{k_n+m_o}\right) \right)\biggl\}\leq 0
	\end{align}
where the last inequality is true, since the sequence of functions $ \left( u^n_i\right)_{n\in\N}  $ are subsolutions of $ (BVP). $ Consequently, the desired inequality is established.

		\item We proceed to the proof that $ \text{for all\ } (y,x)\in(0,L)\times\partial{\Omega},\ w_i^*(y,x)\leq f_i(y,x) $\\
		Let $ (\hat{y},\hat{x})\in(0,L)\times\partial{\Omega} $ a random point. Similarly, as before, by the definition of $ w^*_i $ there exists a sequence 
		\begin{gather}
			\left( y_n,x_n,u^n_i(y_n,x_n)\right)\xrightarrow{n\to\infty}\left( \hat{y},\hat{x},w^*_i(\hat{y},\hat{x}) \right)  
		\end{gather}
		 where $ \left( y_n,x_n\right)\in(0,L)\times\partial{\Omega}  $ and $ u^n_i $ appropriate sequence of subsolution of $ (BVP) $. Since $  u^n_i $ are subsolutions, we receive that $ \text{for all\ } n\in\N $ 
		 \begin{gather}
		 	u^n_i(y_n,x_n) \leq f_i(y_n,x_n)
		 \end{gather}
		 consequently, by taking $ n\to\infty $, we receive the desired inequality.
		\end{itemize}
		In particular, since $ w^{*} $ is a subsolution of $ (BVP), $ then by the definition of $ w $, it is extracted that $ w^{*}\leq w, $ on $ [0,L]\times\bar{\Omega}.  $ Also from $ (\ref{visco_1}), $ we have $ w\leq w^{*} $ on $ [0,L]\times\bar{\Omega}. $ From the last, we conclude that $ w=w^{*} $ on $ [0,L]\times\bar{\Omega}. $\\
		
		Next we prove that $ w_{*} $ is a supersolution of $ (BVP), $ following a classical argument by contradiction. In specific, we show that, if $ w_{*} $ is not a supersolution, then there exists a subsolution strictly greater than $ w^{*}, $ where the last function, is equal with $ w. $ This is a contradiction, by the definition of $ w. $\\
		
		First, we notice that, $ w_{*}  $ satisfies $ w_{*,i}(0,x)\geq g_i(x),\ \text{for all\ } x\in\bar{\Omega}. $ Indeed, by the definition of $ w_i $, we have that $ w_i(y,x)\geq U_i^{\hat{x},\epsilon}(y,x),\ \text{for all\ } (y,x)\in[0,L]\times\bar{\Omega},\ \text{for all\ }\epsilon>0. $ Due to the fact that $ U_i^{\hat{x},\epsilon} $ is continuous on $ [0,L]\times\bar{\Omega}, $ it is extracted that the specific function lower semicontinuous on $  [0,L]\times\bar{\Omega}. $ From $ (\ref{perron_1}) $, we receive that $ w_{*,i}(y,x)\geq U_i^{\hat{x},\epsilon}(y,x) $ on $ [0,L]\times\bar{\Omega},\ \text{for all\ }\epsilon>0. $ In particular, we get that $ w_{*,i}(0,\hat{x})\geq U_i^{\hat{x},\epsilon}(0,\hat{x}),\ \text{for all\ }\epsilon>0.  $ From the last we conclude that the real number $ w_{*,i}(0,\hat{x})(0,\hat{x}) $ is an upper bound of the set $ \{U_i^{\hat{x},\epsilon}(0,\hat{x})\ :\  \epsilon>0\}. $ Consequently,
		\begin{gather}
			w_{*,i}(0,\hat{x})(0,\hat{x})\geq\sup_{\epsilon>0}U_i^{\hat{x},\epsilon}(0,\hat{x})\stackrel{   \text{Proposition}\ \ref{existence of barriers}}{=}g_i(\hat{x}).
		\end{gather}
	Assume now that  $ w_{*,i} $ is not a supersolution $ (BVP). $ Then, there exists at least one index\\ $ i_o\in\{1,2,\dots,m\} $ such that the function $ w_{*,i_o} $ does not satisfy the supersolution property. Then there are two main possible cases. 
	\begin{itemize}
		\item $ \text{There exists\ } (\hat{y},\hat{x})\in(0,L)\times\Omega $ and $ (\alpha_o,p_o,X_o) \in \bar{J}^{2,-}w_{*,i_o}(\hat{y},\hat{x})$, such that
		\begin{gather}\label{main_case_w_*}
			\min\{F\left(\hat{y},\hat{x},w_{*,i_o}(\hat{y},\hat{x}),p_o,X_o \right),w_{*,i_o}(\hat{y},\hat{x}) -\mathcal{M}_{i_o}w_{*}(\hat{y},\hat{x})\}<0
		\end{gather}
	For $ \delta\in(0,\delta_o), $ let $ \beta:(0,\delta_o)\to\R $ be a random real positive function with the identity $ \lim_{\delta\to 0}\beta(\delta)=0. $ Furthermore, let $ h:[0,L]\times\bar{\Omega}\times(0,\delta_o)\to\R $ be a random function, such that the following identities hold:
	\begin{itemize}
		\item $ \lim_{(\delta,y,x)\to(0^+,\hat{y},\hat{x})}D_x h(y,x,\delta)=0 $ and  $ \lim_{(\delta,y,x)\to(0^+,\hat{y},\hat{x})}D^2_{xx} h(y,x,\delta)=0 $ \nonumber
		\item $h$ is continuous and bounded on its domain
		\item  $ \text{for all\ }\delta\in(0,\delta_o),\ \text{for all\ } V\subset [0,L]\times\bar{\Omega}\ \text{bounded}, \sup_{(y,x)\in V}h(y,x,\delta)>0$\nonumber.
	\end{itemize}
	 Gradually, during the following proof, we will demand extra conditions for functions $ \beta $ and $ h. $ 
	 Now, we construct the following family of functions \\ $ \tilde{w}^{\delta}:[0,L]\times\bar{\Omega}\to\R^m $, where $ \tilde{w}^{\delta}=\left( \tilde{w}^{\delta}_1,\tilde{w}^{\delta}_2,\dots, \tilde{w}^{\delta}_m\right)  $ with
	\begin{align}
		\tilde{w}^{\delta}_{i}(y,x):=&w_{*,i}(\hat{y},\hat{x})+ h(y,x,\delta)\beta(\delta)+\alpha_o (y-\hat{t})+\left\langle p_o,(x-\hat{x}) \right\rangle+\nonumber\\ 
		&+\frac{1}{2}\left\langle X_o\ (x-\hat{x}), (x-\hat{x})\right\rangle,\ i\in\left\lbrace1,2,\dots,m \right\rbrace\nonumber.
	\end{align}
For each $ i\in\left\lbrace 1,2,\dots,m\right\rbrace  $ fixed, we define the function $ S_i(\delta,y,x):= \tilde{w}^{\delta}_{i}(y,x).  $ We observe that 
\begin{gather}
	\lim_{(\delta,y,x)\to(0^+,\hat{y},\hat{x})} S_i(\delta,y,x)=w_{*,i}(\hat{y},\hat{x})\nonumber\\
	\text{and}\ \lim_{(\delta,y,x)\to(0^+,\hat{y},\hat{x})} D_x S_i(\delta,y,x)=\lim_{(\delta,y,x)\to(0^+,\hat{y},\hat{x})} D_x\tilde{w}^{\delta}_i(y,x)=p_o\nonumber\\
	\lim_{(\delta,y,x)\to(0^+,\hat{y},\hat{x})} D_{xx}^2 S_i(\delta,y,x)=\lim_{(\delta,y,x)\to(0^+,\hat{y},\hat{x})} D_{xx}^2\tilde{w}^{\delta}_i(y,x)=X_o.
\end{gather}

\begin{claim}\label{claim_subsolution property}
	 There exists sufficiently small positive parameters,  $ \delta^*\in(0,\delta_o),$ and $ R\in(0,R_o], $  such that the function $ \tilde{w}_{i_o}\equiv\tilde{w}^{\delta^*}_{i_o}:[0,L]\times\bar{\Omega}\to\R, $ satisfies the subsolution property of
	\begin{gather}
		\min\{F(y,x,\tilde{w}_{i_o}(y,x),D_x \tilde{w}_{i_o}(y,x),D^2_{xx} \tilde{w}_{i_o}(y,x)), \tilde{w}_{i_o}(y,x)-\mathcal{M}_{i_o} w^*(y,x)\}=0
	\end{gather}
 	in $ Q_R:=\left\lbrace (y,x)\in[0,L]\times\bar{\Omega}\ :\  \abs{y-\hat{y}}+\abs{x-\hat{x}}^2\leq R\right\rbrace $ and also $ \tilde{w}_{i_o}\leq M $ on $ Q_R. $ 
\end{claim}

\begin{proof}
	Initially, we observe that
	\begin{align}
		\mathcal{M}_{i_o}w_*(\hat{t},\hat{x})&=\max_{\lambda\neq i_o}\left( w_{*,\lambda}(\hat{y},\hat{x}) -c_{i_o,\lambda}(\hat{y},\hat{x})\right)\nonumber\\
		&= \max_{\lambda\neq i_o}\left(\lim_{(\delta,y,x)\to(0^+,\hat{y},\hat{x})}\left( \tilde{w}^{\delta}_{\lambda}(y,x)-c_{i_o,\lambda}(y,x) \right)\right) \nonumber\\
		&=\lim_{(\delta,y,x)\to(0^+,\hat{y},\hat{x})}\max_{\lambda\neq i_o}\left(\tilde{w}^{\delta}_{\lambda}(y,x)-c_{i_o,\lambda}(y,x) \right)\nonumber\\
		&= \lim_{(\delta,y,x)\to(0^+,\hat{y},\hat{x})}\mathcal{M}_{i_o}\tilde{w}^{\delta}(y,x)
	\end{align} 
	thus, from the above, we receive
	\begin{gather}
	w_{*,i_o}\left(\hat{y},\hat{x}\right)-\mathcal{M}_{i_o}w_*(\hat{y},\hat{x})=\lim_{(\delta,y,x)\to(0^+,\hat{y},\hat{x})}\left(\tilde{w}^{\delta}_{i_o}(y,x)-\mathcal{M}_{i_o}\tilde{w}^{\delta}(y,x)\right). 
	\end{gather}
	Moreover, from continuity of $ F $, we have that
	\begin{gather}
		F\left(\hat{y},\hat{x},w_{*,i_o}(\hat{y},\hat{x}),p_o,X_o \right)=\lim_{(\delta,y,x)\to(0^+,\hat{y},\hat{x})}F\left(y,x,\tilde{w}^{\delta}_{i_o}(y,x),D_x\tilde{w}^{\delta}_{i_o}(y,x),D^2_{xx}\tilde{w}^{\delta}_{i_o}(y,x) \right).
	\end{gather}
	Considering the above limits, we have that
	\begin{align}
		0&>^{(\ref{main_case_w_*})}\min\left\lbrace 	F\left(\hat{y},\hat{x},w_{*,i_o}(\hat{y},\hat{x}),p_o,X_o \right), w_{*,i_o}(\hat{y},\hat{x})-\mathcal{M}_{i_o}w_*(\hat{y},\hat{x})\right\rbrace \nonumber\\
		&=\min\Biggl\{\lim_{(\delta,y,x)\to(0^+,\hat{y},\hat{x})}F\left(y,x,\tilde{w}^{\delta}_{i_o}(y,x),D_x\tilde{w}^{\delta}_{i_o}(y,x),D^2_{xx}\tilde{w}^{\delta}_{i_o}(y,x) \right),\nonumber\\ &\qquad\qquad\lim_{(\delta,y,x)\to(0^+,\hat{y},\hat{x})}\left( \tilde{w}^{\delta}_{i_o}(y,x)-\mathcal{M}_{i_o}\tilde{w}^{\delta}(y,x)\right) \Biggl\} \nonumber\\
		&=\lim_{(\delta,y,x)\to(0^+,\hat{y},\hat{x})}\min\Biggl\{ F\left(y,x,\tilde{w}^{\delta}_{i_o}(y,x),D_x\tilde{w}^{\delta}_{i_o}(y,x),D^2_{xx}\tilde{w}^{\delta}_{i_o}(y,x) \right),\nonumber\\ &\qquad\qquad\qquad\qquad\qquad\qquad\tilde{w}^{\delta}_{i_o}(y,x)-\mathcal{M}_{i_o}\tilde{w}^{\delta}(y,x)\Biggl\} 
	\end{align}
	consequently, there exists $ \tilde{\delta}\in(0,\delta_o)>0 $ and $ \tilde{R}\in(0,R_o], $ such that $ \text{for all\ }\delta\in(0,\tilde{\delta}) $ and \\$ \text{for all\ } (y,x)\in Q_{\tilde{R}}, $
	\begin{gather}
		\min\left\lbrace F\left(y,x,\tilde{w}^{\delta}_{i_o}(y,x),D_x\tilde{w}^{\delta}_{i_o}(y,x),D^2_{xx}\tilde{w}^{\delta}_{i_o}(y,x) \right), \tilde{w}^{\delta}_{i_o}(y,x)-\mathcal{M}_{i_o}\tilde{w}^{\delta}(y,x)\right\rbrace <0.
	\end{gather}
We now fix $ \delta^*\in(0,\tilde{\delta}) $, and from the above inquality it holds that, $ \text{for all\ } (y,x)\in Q_{\tilde{R}} $
\begin{gather}\label{important_relation_b}
	\min\left\lbrace F\left(y,x,\tilde{w}^{\delta^*}_{i_o}(y,x),D_x\tilde{w}^{\delta^*}_{i_o}(y,x),D^2_{xx}\tilde{w}^{\delta^*}_{i_o}(y,x) \right), \tilde{w}^{\delta^*}_{i_o}(y,x)-\mathcal{M}_{i_o}\tilde{w}^{\delta^*}(y,x)\right\rbrace <0.
\end{gather}
We claim now, that there exists a sufficently small  $ R\in(0,\tilde{R}) $ such that, $ \text{for all\ } (y,x)\in Q_R,$
	\begin{gather}
		\mathcal{M}_{i_o}\tilde{w}^{\delta^*}(y,x) \leq\mathcal{M}_{i_o}w^*(y,x).
	\end{gather}
Indeed, we first notice that
\begin{align}
\mathcal{M}_{i_o}\tilde{w}^{\delta^*}(y,x)&=\max_{\lambda\neq i_o}\left\lbrace\tilde{w}^{\delta^*}_{\lambda}(y,x)-c_{i_o,\lambda}(y,x) \right\rbrace\nonumber\\
&=\max_{\lambda\neq i_o}\{ w_{*,\lambda}(\hat{y},\hat{x})+ h(y,x,\delta^*)\beta(\delta^*)+\alpha_o (y-\hat{y})+\left\langle p_0, x-\hat{x}\right\rangle+\nonumber\\
&+\frac{1}{2}\left\langle X_o (x-\hat{x}),(x-\hat{x}) \right\rangle-c_{i_o,\lambda}(y,x) \}\nonumber\\
&\equiv\max_{\lambda\neq i_o}\left\lbrace K(y,x)-c_{i_o,\lambda}(y,x) \right\rbrace\label{important_relation_o}
\end{align}
where 
\begin{gather}
	K(y,x):=w_{*,\lambda}(\hat{y},\hat{x})+h(y,x,\delta^*)\beta(\delta^*)+\alpha_o (y-\hat{y})+\left\langle p_0, x-\hat{x}\right\rangle+\frac{1}{2}\left\langle X_o (x-\hat{x}),(x-\hat{x}) \right\rangle\nonumber.
\end{gather}
Due to the fact that, $ w_{*,\lambda} $ is lower semicontinuous at point $ (\hat{y},\hat{x}) $, for fixed $ \epsilon^*>0 $, there exists $ r_{\lambda}>0 $ such that, $ \text{for all\ }(y,x)\in B_{\rho_2}\left( (\hat{y},\hat{x}),r_{\lambda} \right)\cap \left( [0,L]\times\bar{\Omega}\right),   $
\begin{gather}
	w_{*,\lambda}(\hat{y},\hat{x})<w_{*,\lambda}(y,x)+\epsilon^*\leq w^*_{\lambda}(y,x)+\epsilon^*.
\end{gather}
	We define $ r_o:=\min\left\lbrace r_{\lambda}\ :\  \lambda\in\left\lbrace 1,2,\dots,m\right\rbrace\setminus\left\lbrace i_o \right\rbrace   \right\rbrace  $. Then $ \text{for all\ } (y,x)\in B_{\rho_2}\left( (\hat{y},\hat{x}),r_o\right) \cap \left( [0,L]\times\bar{\Omega}\right),  $
	\begin{gather}
		w_{*,\lambda}(\hat{y},\hat{x})\leq w^*_{\lambda}(y,x)+\epsilon^*.
	\end{gather}
	We select $ R\in (0,\tilde{R}),  $ such that $ Q_R\subset B_{\rho_2}\left( (\hat{y},\hat{x}),r_o\right) \cap \left( [0,L]\times\bar{\Omega}\right).  $
	From the above estimations, $\text{for all\ } \lambda\in\left\lbrace 1,2,\dots,m \right\rbrace\setminus \left\lbrace i_o \right\rbrace\ \text{and}\   \text{for all\ } (y,x)\in Q_R $
	\begin{align}
		K(y,x)&\leq w^*_{\lambda}(y,x)+\epsilon^*+ h(y,x,\delta^*)\beta(\delta^*)+\alpha_o (y-\hat{y})+\left\langle p_0, x-\hat{x}\right\rangle+\nonumber\\
		&+\frac{1}{2}\left\langle X_o (x-\hat{x}),(x-\hat{x}) \right\rangle\nonumber\\
		&\leq w^*_{\lambda}(y,x)+\epsilon^*+\sup\left\lbrace h(y,x,\delta^*)\ : \ \ (y,x)\in Q_R\right\rbrace  \beta(\delta^*)+\sup\big{\lbrace}\alpha_o (y-\hat{y})+\left\langle p_0, x-\hat{x}\right\rangle+\nonumber\\
		&+\frac{1}{2}\left\langle X_o (x-\hat{x}),(x-\hat{x}) \right\rangle\ : \ \ (y,x)\in Q_R \big{\rbrace}
	\end{align}
and we set 
\begin{gather}
	\Lambda:=\sup\big{\lbrace}\alpha_o (y-\hat{y})+\left\langle p_0, x-\hat{x}\right\rangle
	+\frac{1}{2}\left\langle X_o (x-\hat{x}),(x-\hat{x}) \right\rangle\ \ :\ \ (y,x)\in Q_R \big{\rbrace}\nonumber.
\end{gather}
The main target at this point, is to guarantee the following
\begin{gather}\label{negativity_rel}
	\epsilon^*+\sup\{h(y,x,\delta^*)\ : \ (y,x)\in Q_R\} +\Lambda\leq 0\
	\text{and}\ \beta(\delta^*) h(\hat{y},\hat{x},\delta^*)>0
\end{gather}
 setting appropriate extra conditions for functions $ h $ and $ \beta. $ The second condition will be needed in the proof of the following claim.\\
For this task, we have the following cases
\begin{itemize}
	\item Let $ \Lambda\leq 0 $. In this case, we can choose $ \epsilon^*>0 $ such that $ \epsilon^*+\Lambda<0 $ holds. Equivalently, $ 0<-\epsilon^*-\Lambda. $ In order to hold $ (\ref{negativity_rel}) $, it is sufficent to define the value of function $ \beta $ at point $ \delta
	^* $ in a way that the following holds
	\begin{gather}
		0<\beta(\delta^*)\leq\frac{-\epsilon^*-\Lambda}{\sup\{h(y,x,\delta^*)\ | \ (y,x)\in Q_R\}}\equiv k_o\ \text{and}\ h(\hat{y},\hat{x},\delta^*)>0.
	\end{gather}

	\item Let $ \Lambda> 0 $. In this case, we receive automatically that $ -\epsilon^*-\Lambda<0. $ Then again, we observe that, if we restrict the value of $ \beta $ at point $ \delta^* $ to be inside the interval $ (-\infty,k_o) $ and moreover demand $ h(\hat{y},\hat{x},\delta^*)<0 $, we conclude that $ (\ref{negativity_rel}) $ holds.
\end{itemize}
	From, both cases, we obtain that, $ \text{for all\ } (y,x)\in Q_R,\ K(y,x)\leq w^*_{\lambda}(y,x) $. As a result, we receive that
	\begin{align}
		\mathcal{M}_{i_o}\tilde{w}^{\delta^*}(y,x)&\leq^{(\ref{important_relation_o})}\max_{\lambda\neq i_o}\left\lbrace K(y,x)-c_{i_o,\lambda}(y,x) \right\rbrace\nonumber\\
		&\leq\max_{\lambda\neq i_o}\left\lbrace w^*_{\lambda}(y,x)-c_{i_o,\lambda}(y,x)\right\rbrace=\mathcal{M}_{i_o}w^*(y,x). 
	\end{align}
	Equivalently, we obtain $ \text{for all\ } (y,x)\in Q_R\subset Q_{\tilde{R}}, $ 
	\begin{gather}
		\tilde{w}^{\delta^*}_{i_o}(y,x)-\mathcal{M}_{i_o}w^*(y,x) \leq \tilde{w}^{\delta^*}_{i_o}(y,x)-\mathcal{M}_{i_o}\tilde{w}^{\delta^*}(y,x).
	\end{gather}
	From the last inequality combined with $ (\ref{important_relation_b}), $ we receive that, $ \text{for all\ } (y,x)\in Q_R $ that
	\begin{gather}
		\min\left\lbrace F\left(y,x,\tilde{w}^{\delta^*}_{i_o}(y,x),D_x\tilde{w}^{\delta^*}_{i_o}(y,x),D^2_{xx}\tilde{w}^{\delta^*}_{i_o}(y,x) \right), \tilde{w}^{\delta^*}_{i_o}(y,x)-\mathcal{M}_{i_o}w^{*}(y,x)\right\rbrace <0.
	\end{gather}
	In conclusion, the function $ \tilde{w}_{i_o}:=\tilde{w}^{\delta^*}_{i_o} $ satisfies in the classical sense the subsolution property.

\end{proof}

Next we define the following function $ \hat{u}:=(\hat{u}_1,\hat{u}_2,\dots,\hat{u}_m):[0,L]\times\bar{\Omega}\to\R^m $, where 
\begin{gather}
	\hat{u}_{i_o}(y,x):=\begin{cases}
		\max\{w^*_{i_o}(y,x),\tilde{w}_{i_o}(y,x)\},\ & \text{ if }\ (y,x)\in Q_R\nonumber\\
		w^*_{i_o}(y,x),\ & \text{otherwise} 
	\end{cases}
 \ \text{and}\ \hat{u}_{j}(y,x):=w^*_{j}(y,x),\ \text{if}\ j\neq i_o.
\end{gather}
\begin{claim}
The function $ \hat{u}_{i_o}:[0,L]\times\bar{\Omega}\to\R $ is upper semicontinuous.
\end{claim}
\begin{proof}
	 First, we notice that $ Q_R $ is a closed subset of $ \R\times\R^n $ and as a result, of the closedness of $ [0,L]\times\bar{\Omega} $, we receive that $ Q_R $ is closed in the metric subspace $ [0,L]\times\bar{\Omega}. $ Let $ (\tilde{y},\tilde{x})\in[0,L]\times\bar{\Omega} $ be a random point. There are three cases: $ (\tilde{y},\tilde{x})\in int\left( Q_R \right) $ or $ (\tilde{y},\tilde{x})\in \partial Q_R   $ or \\ $ (\tilde{y},\tilde{x})\in K:=\left(  [0,L]\times\bar{\Omega}\right) \setminus Q_R. $ 
	 \begin{itemize}
	 	\item Let $  (\tilde{y},\tilde{x})\in int\left( Q_R\right). $ Since $ int\left( Q_R\right) $ is an open set, then there exists $ \tilde{\delta}>0 $ such that, $ B_{\rho_2}\left((\tilde{y},\tilde{x}),\tilde{\delta} \right)\subset int(Q_R).  $ For a random $ \epsilon>0, $ due to the fact that $ w^*_{i_o} $ is upper semicontinuous, there exists $\delta_1>0,  $ such that $ \text{for all\ } (y,x)\in B_{\rho_2}\left( (\tilde{y},\tilde{x}),\delta_1\right)\cap\left([0,L]\times\bar{\Omega} \right), $
	 	\begin{gather}
	 		w_{i_o}^*(y,x)<w^*_{i_o}(\tilde{y},\tilde{x})+\epsilon \leq\hat{u}_{i_o} (\tilde{y},\tilde{x})+\epsilon.
	 	\end{gather}
	 	Moreover, since $ \tilde{w}_{i_o} $ is continuous at $ (\tilde{y},\tilde{x}) $, it will be also an upper semicontinuous on that point. As a result, there exists $ \delta_2>0 $ such that, $ \text{for all\ } (y,x)\in B_{\rho_2}\left( (\tilde{y},\tilde{x}),\delta_2\right)\cap\left([0,L]\times\bar{\Omega} \right), $
	 	\begin{gather}
	 		\tilde{w}_{i_o}(y,x)<\tilde{w}_{i_o}(\tilde{y},\tilde{x})+\epsilon\leq\hat{u}_{i_o} (\tilde{y},\tilde{x})+\epsilon.
	 	\end{gather}
 		Then, setting $ \delta_o:=\min\{\tilde{\delta},\delta_1,\delta_2\}>0 $ we see that, \\$ \text{for all\ }(y,x)\in B_{\rho_2}\left( (\tilde{y},\tilde{x}),\delta_o\right)=B_{\rho_2}\left( (\tilde{y},\tilde{x}),\delta_o\right)\cap \left( [0,L]\times\bar{\Omega}\right),   $
 		\begin{gather}
 			w_{i_o}^*(y,x)<\hat{u}_{i_o}(\tilde{y},\tilde{x})+\epsilon\ \text{and}\ \tilde{w}_{i_o}(y,x)<\hat{u}_{i_o}(\tilde{y},\tilde{x})+\epsilon
 		\end{gather}
 	and from the last, we obtain $ \hat{u}_{i_o}(y,x):=\max\{w^*_{i_o}(y,x),\tilde{w}_{i_o}(y,x)\}<\hat{u}_{i_o}(\tilde{y},\tilde{x}) +\epsilon$
	 	\item Let $ (\tilde{y},\tilde{x})\in\partial Q_R $ As before, due to the fact that $ w^*_{i_o} $ and $ \tilde{w}_{i_o} $ are upper semicontinuous at point $ (\tilde{y},\tilde{x}) $, for $ \epsilon>0 $ randrom the following holds:
	 	\begin{gather}
	 		\text{there exists\ } \delta_1>0:\ \text{for all\ } (y,x)\in B_{\rho_2}\left((\tilde{y},\tilde{x}),\delta_1 \right)\cap \left([0,L]\times\bar{\Omega} \right)\nonumber\\
	 		w^*_{i_o}(y,x)<w^*_{i_o}(\tilde{y},\tilde{x})+\epsilon\leq\max\{w^*_{i_o}(\tilde{y},\tilde{x}),\tilde{w}_{i_o}(\tilde{y},\tilde{x})\}+\epsilon=\hat{u}_{i_o}(\tilde{y},\tilde{x})+\epsilon\label{relation_aghjk_1}\\
	 		\text{there exists\ } \delta_2>0:\ \text{for all\ } (y,x)\in B_{\rho_2}\left((\tilde{y},\tilde{x}),\delta_2 \right)\cap \left([0,L]\times\bar{\Omega} \right)\nonumber\\
	 		\tilde{w}_{i_o}(y,x)<\tilde{w}_{i_o}(\tilde{y},\tilde{x})+\epsilon\leq \hat{u}_{i_o}(\tilde{y},\tilde{x})+\epsilon\label{relation_aghjk_2}.
	 	\end{gather}
 		We set $ \delta_o:=\{\delta_1,\delta_2\}>0. $ Then $ \text{for all\ } (y,x)\in B_{\rho_2}\left( (\tilde{y},\tilde{x}),\delta_o\right)\cap\left([0,L]\times\bar{\Omega} \right)   $ from $ (\ref{relation_aghjk_1}) $ and $ (\ref{relation_aghjk_2}) $ we obtain that 
 		\begin{gather}
 			\hat{u}_{i_o}(y,x)=\max\{w^*_{i_o}(y,x),\tilde{w}_{i_o}(y,x)\}<\hat{u}_{i_o}(\tilde{y},\tilde{x})+\epsilon.
 		\end{gather}
	 	\item Let $ (\tilde{y},\tilde{x})\in K:=\left(  [0,L]\times\bar{\Omega}\right) \setminus Q_R, $ by the definition of $ \hat{u}_{i_o} $ we obtain that $ \hat{u}_{i_o}(\tilde{y},\tilde{x})=w^*_{i_o}(\tilde{y},\tilde{x}). $ Because $ w^*_{i_o}  $ is upper semicontinuous at point $ (\tilde{y},\tilde{x}), $ for a random $ \epsilon>0, $ there exists $\delta>0$ such that $ \text{for all\ } (y,x)\in B_{\rho_2}\left( (\tilde{y},\tilde{x}),\delta \right)\cap\left( [0,L]\times\bar{\Omega}\right) \subset K, $
	 	\begin{gather}
	 		 \hat{u}_{i_o}(y,x)=w^*_{i_o}(y,x)<w^*_{i_o}(\tilde{y},\tilde{x})+\epsilon=\hat{u}_{i_o}(\tilde{y},\tilde{x})+\epsilon\nonumber.
	 	\end{gather} 
	 \end{itemize}
From the three cases, we conclude that $ \hat{u}_{i_o} $ is upper semicontinuous on its domain.
\end{proof}
Next, we proceed to the proof that $ \hat{u}:[0,L]\times\bar{\Omega}\to\R $ is a subsolution of $ (BVP). $\\
\begin{claim}
The function $ \hat{u}:[0,L]\times\bar{\Omega}\to\R^m $ with $ \hat{u}=(\hat{u}_1,\hat{u}_2,\dots,\hat{u}_{i_o},\dots\hat{u}_m)  $ is a subsolution of $ (BVP)$.
\end{claim}
\begin{proof}
Initially, we prove that $ \text{for all\ } j\neq i_o, \hat{u}_j $ is a subsolution of $ (BVP). $ We notice that 
\begin{gather}\label{main relation_1*}
	\text{for all\ } (y,x)\in(0,L)\times\Omega,\ \hat{u}_j(y,x)-\mathcal{M}_j\hat{u}(y,x)\leq w^*_j(y,x)-\mathcal{M}_j w^*(y,x)
\end{gather}
Indeed, for $ (y,x)\in (0,L)\times\Omega $ 
\begin{equation*}
	\hat{u}_j(y,x)-\mathcal{M}_j\hat{u}(y,x)\stackrel{\text{Def of}\  \hat{u}_j}{=}w^*_j(y,x)-\max\Bigl\{ \hat{u}_{\lambda}(y,x)-c_{j,\lambda}(y,x)\ :\  \lambda\in\{1,2,\dots,i_o,\dots,m\}\setminus\{j\} \Bigl\}\nonumber
\end{equation*}
but for each $ \lambda\in\{1,2,\dots,m\}\setminus\{i_o\},\ \hat{u}_{\lambda}=w^*_{\lambda}$ and $ \hat{u}_{i_o}\geq w^*_{i_o} $ on $ [0,L]\times\bar{\Omega}. $
Thus, we obtain
\begin{gather}
	\max\Bigl\{ \hat{u}_{\lambda}(y,x)-c_{j,\lambda}(y,x)\ : \ \lambda\in\{1,2,\dots,i_o,\dots,m\}\setminus\{j\} \Bigl\}\nonumber\\
	\geq\max\Bigl\{ w^*_{\lambda}(y,x)-c_{j,\lambda}(y,x)\ :\  \lambda\in\{1,2,\dots,i_o,\dots,m\}\setminus\{j\} \Bigl\}\nonumber
\end{gather}
and equivalently,
\begin{align}
	\hat{u}_j(y,x)-\mathcal{M}_j\hat{u}(y,x)&=w^*_j (y,x)-\max\Bigl\{ \hat{u}_{\lambda}(y,x)-c_{j,\lambda}(y,x)\ :\ \lambda\in\{1,2,\dots,i_o,\dots,m\}\setminus\{j\} \Bigl\}\nonumber\\
	&\leq w^*_j (y,x)-\max\Bigl\{ w^*_{\lambda}(y,x)-c_{j,\lambda}(y,x)\ :\ \lambda\in\{1,2,\dots,i_o,\dots,m\}\setminus\{j\} \Bigl\}\nonumber\\
	&=w^*_j (y,x)-\mathcal{M}_j w^*(y,x)\nonumber.
\end{align}
Then, for $ (y,x)\in(0,L)\times\Omega, $ and $ (a,p,X) \in \bar{J}^{2,+}\hat{u}_j(y,x)\stackrel{j\neq i_o}{=}\bar{J}^{2,+}w^*_j(y,x)$, from the definition of $ \hat{u}_j $ it follows, that
\begin{gather}\label{main relation_2*}
	F\left( y,x,\hat{u}_j(y,x),p,X\right) =F\left( y,x,w^*_j(y,x),p,X\right) 
\end{gather}
thus, we obtain
\begin{align}
	&\min\Big\{F\left( y,x,\hat{u}_j(y,x),p,X\right) ,\hat{u}_j(y,x)-\mathcal{M}_j\hat{u}(y,x)\Bigl\}\nonumber\\
	&=^{(\ref{main relation_2*})}\min\Big\{F\left( y,x,w^*_j(y,x),p,X\right) ,\hat{u}_j(y,x)-\mathcal{M}_j\hat{u}(y,x)\Bigl\}\nonumber\\
	&\leq^{(\ref{main relation_1*})} \min\Big\{F\left( y,x,w^*_j(y,x),p,X\right) ,w^*_j(y,x)-\mathcal{M}_j w^*(y,x)\Bigl\}\nonumber\\
	&\leq 0\ \text{(since}\ w^*_j\ \text{is a subsolution of (BVP)}).
\end{align}
Moreover, since $ w^*_j $ is a subsolution of (IBVP), we have for $ y=0, \text{for all\ } x\in\bar{\Omega},\\ \hat{u}_j(0,x)=w^*_j(0,x)\leq g_j(x)$ 
and $ \text{for all\ } (y,x)\in (0,L) \times\partial\Omega,\ \hat{u}_j(y,x)=w^*_j(y,x)\leq f_j(y,x)$. \\
Next, we prove that $ \hat{u}_{i_o}  $ is a subsolution of $ (BVP) $.
First, we remind that due to the fact that $ (\hat{y},\hat{x}) \in (0,L)\times\Omega $, and $ (\hat{y},\hat{x}) $ is an internal point of $ Q_R, $ we are able to choose an appropriate radious $ R $ such that $ Q_R\subset(0,L)\times\Omega. $ Then by the definition of $ \hat{u}_{i_o}, $ and the fact that $ w^*_{i_o} $ is a subsolution of $ (BVP) $ we obtain that
\begin{itemize}
	\item $\text{for all\ } x\in\bar{\Omega},\ \hat{u}_{i_o}(0,x)=w^*_{i_o}(0,x)\leq g_{i_o}(x)$
	\item $\text{for all\ } (y,x)\in(0,L)\times\partial\Omega,\ \hat{u}_{i_o}(y,x)=w^*_{i_o}(y,x)\leq f_{i_o}(y,x)$.
\end{itemize}
Now, let $ (y,x)\in(0,L)\times\Omega, $ and $ (a,p,X)\in \bar{J}^{2,+}\hat{u}_{i_o}(y,x) $. There are two main cases: \\$ (y,x)\in(0,L)\times\Omega\setminus Q_R $ or $ (y,x)\in Q_R $. In the first case, by the definition of $ \hat{u}_{i_o} $, we have that $ \hat{u}_{i_o}(y,x)=w^*_{i_o}(y,x) $ and, since $ w^*_{i_o} $ is a subsolution we have the desired inequality
\begin{gather}
	\min\Bigl\{ F\left( y,x,\hat{u}_{i_o}(y,x),p,X\right),\hat{u}_{i_o}(y,x)-\mathcal{M}_{i_o}\hat{u}(y,x) \Bigl\}\leq 0\nonumber.
\end{gather}
For the second  case, we have
\begin{gather}
	\hat{u}_{i_o}(y,x)=\max\{w^*_{i_o}(y,x),\tilde{w}_{i_o}(y,x)\}\nonumber.
\end{gather}
In this situation, there are two possible subcases: $ \hat{u}_{i_o}(y,x)=w^*_{i_o}(y,x) $ or $ \hat{u}_{i_o}(y,x)=\tilde{w}_{i_o}(y,x) $. We notice that 
\begin{align}
	\mathcal{M}_{i_o}\hat{u}(y,x)=\max_{\lambda\neq i_o}\{\hat{u}_{\lambda}(y,x)-c_{i_o,\lambda}(y,x)\}&=\max_{\lambda\neq i_o}\{w^*_{\lambda}(y,x)-c_{i_o,\lambda}(y,x)\}\nonumber\\
	&=\mathcal{M}_{i_o}w^*(y,x).	\nonumber
\end{align}
 Also, from $ (\ref{claim_subsolution property}) $, we receive that 
\begin{gather}
	\min\{F(y,x,\tilde{w}_{i_o}(y,x),D_x \tilde{w}_{i_o}(y,x),D^2_{xx} \tilde{w}_{i_o}(y,x)), \tilde{w}_{i_o}(y,x)-\mathcal{M}_{i_o} w^*(y,x)\}\leq 0\nonumber.
\end{gather}
Using the above we obtain that
\begin{gather}
	\min\{F(y,x,\hat{u}_{i_o}(y,x),p,X),\hat{u}_{i_o}(y,x)-\mathcal{M}_{i_o} \hat{u}(y,x)\}\leq 0\nonumber.
\end{gather}

\end{proof}

In the next part, we reach to a contradiction. In specific, we prove that function $ \hat{u}_{i_o} $ is strictly greater than $ w_{i_o}, $ which is a contradiction by the definition of $ w_{i_o}.  $ Previously, we saw that $ w^*_{i_o}=w_{i_o}.  $ So it follows that $ (w^*_{i_o})_*=w_{*,i_o} $. From the definition of $ (w^*_{i_o})_* $ there exists a sequence
\begin{gather}\label{approaching sequence}
	\left( y_n,x_n,w^*_{i_o}(y_n,x_n)\right) \xrightarrow{n\to\infty}\left( \hat{y},\hat{x},(w^*_{i_o})_{*}(\hat{y},\hat{x})\right)=\left( \hat{y},\hat{x},(w_{*,i_o})(\hat{y},\hat{x})\right).
\end{gather} 
Then, it holds that, for each $ n\in\N, $
\begin{align}
	\tilde{w}_{i_o}(y_n,x_n)&=w_{*,i_o}(\hat{y},\hat{x})+h(y,x,\delta^*)\beta(\delta^*)+\alpha_o (y_n-\hat{y})+\left\langle p_o,x_n-\hat{x}\right\rangle+\nonumber\\
	 &+\frac{1}{2}\left\langle X_o(x_n-\hat{x}),(x_n-\hat{x}) \right\rangle.
\end{align}
Using now (\ref{negativity_rel}), and the fact that $ h $ is a continuous function, we obtain
\begin{gather}
	\tilde{w}_{i_o}(y_n,x_n)\xrightarrow{n\to\infty}w_{*,i_o}(\hat{y},\hat{x})+h(\hat{y},\hat{x},\delta^*)\beta(\delta^*)>^{(\ref{negativity_rel})} w_{*,i_o}(\hat{y},\hat{x})=^{(\ref{approaching sequence})}\lim_{n\to\infty}w^*_{i_o}(y_n,x_n).
\end{gather}
Finally, we receive that,
\begin{gather}
\lim_{n\to\infty}\tilde{w}_{i_o}(y_n,x_n)>\lim_{n\to\infty}w^*_{i_o}(y_n,x_n).
\end{gather}
From the last, we receive that 
\begin{gather}
	\text{there exists\ } n_o\in\N,\text{for all\ } n\geq n_o, \tilde{w}_{i_o}(y_n,x_n)>w^*_{i_o}(y_n,x_n).
\end{gather}
Due to the fact that, $ (\hat{y},\hat{x}) $ is an internal point of $ Q_R $, there exists $ \tilde{n}\in\N $ such that\\ $ (y_n,x_n)\in Q_R $. We set $ n_o^*:=\max\{n_o,\tilde{n}\}, $ we observe by the definition of $ \hat{u}_{i_o}  $, that $ \text{for all\ } n\geq n_o^*, $
\begin{gather}
	\hat{u}_{i_o}(y_n,x_n) :=\max\{w^*_{i_o}(y_n,x_n),\tilde{w}_{i_o}(y_n,x_n)\}=\tilde{w}_{i_o}(y_n,x_n)>w^*_{i_o}(y_n,x_n)=w_{i_o}(y_n,x_n)
\end{gather}
At this point, we have a contradiction from the definition of $ w_{i_o} $ and from the fact that $ \hat{u}_{i_o} $ is a subsolution of $ (BVP). $\\
Next, we consider the case where $ w_{*,i_o} $ fails to be supersolution of the $ (BVP) $ violating the boundary condition i.e
\item $ \text{there exist}\ (\hat{y},\hat{x}) \in(0,L)\times\partial{\Omega}$ such that 
		\begin{gather}
			w_{*,i_o}(\hat{y},\hat{x})<f_{i_o}(\hat{y},\hat{x})\nonumber
		\end{gather}
	From the definition of $ \tilde{w}_{i_o} $, we obtain that 
	$ \tilde{w}_{i_o}(\hat{y},\hat{x}) =w_{*,i_o}(\hat{y},\hat{x})+\beta(\delta)h(\hat{y},\hat{x},\delta)$. For the given functions $ \beta $ and $ h $, we select a sufficently small parameter $ \delta>0 $ such that 
	\begin{gather}\label{boundary condition_perron}
		\tilde{w}_{i_o}(\hat{y},\hat{x})=w_{*,i_o}(\hat{y},\hat{x})+\beta(\delta)h(\hat{y},\hat{x},\delta)<f_{i_o}(\hat{y},\hat{x}).
	\end{gather}
Then, from continuity of $ \tilde{w}_{i_o}  $ and $ f_{i_o} $ at point $ (\hat{y},\hat{x}) $ we receive from $ (\ref{boundary condition_perron}) $, that there exists an appropriate $ r>0 $ such that $ \text{for all\ } (y,x)\in B_{\rho_2}\left((\hat{y},\hat{x}) ,r\right)\cap\left( (0,L)\times\partial\Omega\right) ,\ \tilde{w}_{i_o}(y,x)<f_{i_o}(y,x).  $ As in the previous analysis, we select appropriate functions $ \beta\ \text{and} h $ and radious $ R, $ such that $ \tilde{w}_{i_o} $ satisfies the subsolution property of claim $ (\ref{claim_subsolution property}) $ for $ (y,x)\in\Omega_{L} $ sufficently close to $ (\hat{y},\hat{x}). $ Then, we deduce that $ \hat{u}_{i_o} $ is a subsolution of the problem $ (BVP) $ which dominates $ w^*_{i_o}=w_{i_o}, $ which again is a contradiction relative to the definition of $ w_{i_o}. $
\end{itemize} 

	\end{proof}

	\begin{claim}\label{wi_continuous}
		The function $ w_i:[0,L]\times\bar{\Omega}\to\R$ is continuous on each point of $ [0,L)\times\bar{\Omega} $.
	\end{claim}
	\begin{proof}
		From the claim $ \ref{viscosity_property} $, we have the result that $ w_{*,i}(y,x)=w_i(y,x)=w^*_i(y,x),\ \\ \text{for all\ }(y,x)\in[0,L)\times\bar{\Omega} $. So the restriction of $ w_i $ on the set $ [0,L)\times\bar{\Omega} $, i.e $ w_i |_{[0,L)\times\bar{\Omega}} $ is continuous. Then, from the fact that for each point $ (y,x)\in[0,L)\times\bar{\Omega} $ there exists $ \delta>0 $ such that 
		$ B_{\rho_2}\left( (y,x),\delta\right) \cap \left([0,L)\times\bar{\Omega} \right) = B_{\rho_2}\left( (y,x),\delta\right) \cap \left([0,L]\times\bar{\Omega} \right) $, combined with the continuity of the restriction $ w_i|_{[0,L)\times\bar{\Omega}}  $ on the point $ (y,x), $ we also receive
		\footnote{Let $ (X,\rho) $ be a metric space, $ A\subset B\subset X,\ x_o\in A $ and $ f:B\to\R.  $ If the following holds\\ \begin{gather}
				\exists\delta>0,\ A\cap B_{\rho}(x_o,\delta)=B\cap B_{\rho}(x_o,\delta) \nonumber
			\end{gather} and the restriction $ f|_A  $ is continuous on $ x_o $, then the extension $ f $ also is continuous on $ x_o. $} 
			the continuity of $ w_i $ at $ (y,x). $
		From the above, it is extracted that $ w_i:[0,L]\times\bar{\Omega}\to\R $ is a continuous function on each point of $ [0,L) \times\bar{\Omega}.$
	\end{proof}
	From the last two claims, it is extracted that the function $ w_i:[0,L]\times\bar{\Omega}\to\R $ satisfies the definition of viscosity solution of $ (BVP). $
\end{proof}
\section*{Acknowledgments}
The authors would like to thank the anonymous referee for carefully reading the manuscript and for his/her valuable comments. This work was supported by the University of Cyprus research funds.

\begin{tabular}{l}
Savvas Andronicou\\ University of Cyprus \\ Department of Mathematics \& Statistics \\ P.O. Box 20537\\
Nicosia, CY- 1678 CYPRUS
\\ {\small \tt andronikou.savvas@ucy.ac.cy}
\end{tabular}
\begin{tabular}{lr}
Emmanouil Milakis\\ University of Cyprus \\ Department of Mathematics \& Statistics \\ P.O. Box 20537\\
Nicosia, CY- 1678 CYPRUS
\\ {\small \tt emilakis@ucy.ac.cy}
\end{tabular}

\end{document}